\documentclass{amsart}
\usepackage{amsmath,amscd,amssymb}
\usepackage{amsfonts}
\usepackage[all]{xy}
\usepackage{mathtools}
\usepackage{enumerate}
\usepackage[retainorgcmds]{IEEEtrantools}
\usepackage{color}
\usepackage{tikz-cd}
\usepackage{hyperref}

\theoremstyle{plain}
\newtheorem{thm}{Theorem}[section]
\numberwithin{equation}{thm}

\newtheorem{prop}[thm]{Proposition}
\newtheorem{propn}[thm]{Proposition}
\newtheorem{lemma}[thm]{Lemma}
\newtheorem{cor}[thm]{Corollary}

\theoremstyle{definition}
\newtheorem{defn}[thm]{Definition}
\newtheorem{notn}[thm]{Notation}
\newtheorem{cont}[thm]{Contents}

\theoremstyle{remark}
\newtheorem{rem}[thm]{Remark}

\newtheorem{example}[thm]{Example}

\newcommand{\Z}{\mathbb{Z}}
\newcommand{\F}{\mathbb{F}}
\def\no{n\up{o}\kern 2pt}

\def\rond{\kern 1pt{\scriptstyle\circ}\kern 1pt}
\def\iso{\mathrel{\mathop{\kern 0pt\longrightarrow }\limits^{\sim}}}

\def\Hom{\mathop{\rm Hom}\nolimits}
\def\iHom{\mathcal{H}\mathrm{om}}

\newcommand{\xto}{\xrightarrow}
\newcommand{\Spec}{\mathrm{Spec}}
\newcommand{\kbar}{\overline{k}}
\def\Q{\mathbb{Q}}

\def\det{\mathop{\rm det}\nolimits}

\def\dim{\mathop{\rm dim}\nolimits}

\def\Gal{\mathop{\rm Gal}\nolimits}
\newcommand{\et}{\acute{e}t}

\begin{document}
\title{An arithmetic Yau--Zaslow formula}
\author{Jesse Pajwani} \email{jesse.pajwani@bristol.ac.uk} 
\address{School of Mathematics, School of Mathematics, University of Bristol, Bristol,
BS8 1TW, UK, and the Heilbronn Institute for Mathematical Research, Bristol,
UK}
\author{Ambrus P\'al} \email{ambrus.pal@ttk.elte.hu} \address{Mathematical Institute, Eötvös Loránd University, 1117 Budapest, Pázmány Péter sétány 1/C, Hungary.} 

\begin{abstract} We prove an arithmetic refinement of the Yau--Zaslow formula by replacing the classical Euler characteristic in Beauville's argument by a motivic Euler characteristic, related to the work of Levine. Our result implies similar formulas for other related invariants, including a generalisation of a formula of Kharlamov and R\u asdeaconu on counting real rational curves on real $K3$ surfaces, and Saito's determinant of cohomology. 
\end{abstract}


\maketitle
\pagestyle{myheadings}
\setcounter{tocdepth}{1}
\tableofcontents

\section{Introduction}
Let $X$ be a $K3$ surface over a field $k$ which admits a complete linear system $\mathcal{C}$ consisting of curves of genus $g$. Assume that the curves in $\mathcal{C}$ are geometrically integral and belong to a primitive divisor class. It is well known that $\mathcal{C}$ contains a finite number $n_g$ of rational curves. The starting point for this paper is the Yau--Zaslow formula from \cite{Be} and \cite{YZ}.
\begin{thm}[(Complex Yau--Zaslow formula, \cite{Be}, \cite{YZ})]
Let $k=\mathbb{C}$. Then $n_g$ depends only on $g$, not on the choice of $K3$ surface $X$ or the choice of linear system of genus $g$ curves $\mathcal{C}$. Moreover there is a generating series for $n_g$ in $\mathbb{Z}[[t]]$
$$
1+\sum_{g=1}^{\infty} n_g t^g = \prod_{n=1}^{\infty} (1-t^n)^{-24}.
$$
\end{thm}
Perhaps the historical significance of the above complex Yau--Zaslow formula is that it predicted certain genus $0$ Gromov--Witten invariants; it counts the number of stable maps from $\mathbb P^1$ to a linear system $|L|$ with $c_1(L)^2=2g-2$. Curves in $|L|$ have arithmetic genus $g$ so we can expect the number of rational curves to be finite, under the natural assumption that the rational curves are precisely those with $g$ nodes, which is generically the case by Chen's theorem, Theorem 1.1 of \cite{Ch}. Beauville, whose work we follow closely in this paper, noted that under good conditions these curves are precisely the ones whose compactified Jacobian has nonzero Euler characteristic, and that Euler characteristic is $1$, so he proposed taking the Euler characteristic of the relative compactified Jacobian over the linear system as another way to compute this genus $0$ Gromov--Witten invariant. There has been subsequent work, see for example \cite{BL} and \cite{FGvS}, giving proofs in more general situations, while the state of the art is \cite{KMPS}. From the modern perspective this equality between these genus $0$ Gromov--Witten invariants and Euler characteristic of the compactified Jacobian can be understood as a special case of the MNOP conjecture relating Gromov--Witten and Donaldson--Thomas theories, with the latter simplifying to the Euler characteristic of the compactified Jacobian under good conditions. In this paper we are concerned with giving arithmetic refinements of both sides of this equation with the aid of tools from motivic homotopy. It should be noted that another refinement is provided by \cite{GS} in a different setting. While we take a different approach to these, recent papers such as \cite{Vi} and \cite{EW} compute refined Donaldson--Thomas invariants for $\mathbb{P}^3$ and $\mathbb{A}^3$ respectively using motivic homotopy, providing evidence that it should be possible to refine this side in a general case. Similarly in \cite{KLSW}, the authors use a different approach to compute refinements of genus $0$ Gromov--Witten invariants using tools from motivic homotopy.

For the Yau--Zaslow formula, Kharlamov and R\u asdeaconu also proved a version of the Yau--Zaslow formula over $\mathbb{R}$ in \cite{KR1}.
\begin{thm}[(Real Yau--Zaslow formula, \cite{KR1})]
Let $k=\mathbb{R}$. Let $n_+$ (resp. $n_-$) denote the number of rational curves in $\mathcal{C}$ with Welschinger number $+$ (resp. $-$), and let $\mathbb{W}(X,g) := n_+ - n_-$. Then $\mathbb{W}(X,g)$ is given by the coefficient of $t^g$ in the power series
$$
\prod_{i=1}^{\infty}\big(1+(-t^i)\big)^{-a} \cdot \prod_{j=1}^{\infty}\big(1-t^{2j}\big)^{-b} \in \Z[[t]],
$$
where $a$ is the compactly supported Euler characteristic of $X(\mathbb{R})$, and $b = \frac12 (24-a)$. 
\end{thm}
Recent papers show that for $k$ any field of characteristic not $2$, we can often refine results in enumerative geometry over $\mathbb{C}$ to obtain an element of the Grothendieck--Witt ring of our base field $\widehat{\mathrm{W}}(k)$, a ring which is widely studied for its number theoretic properties. The element we obtain is an arithmetic refinement in the sense that we may recover the integer obtained by the classical count over $\mathbb{C}$ by applying a certain homomorphism out of this ring. Examples of such arithmetic refinements are an arithmetic refinement of the count of $27$ lines on a smooth cubic surface in \cite{KW}, and an arithmetic refinement of the count of lines meeting $4$ lines in $\mathbb{P}^3$ in \cite{SW}. As a result of this paper we obtain an arithmetic refinement of the Yau--Zaslow formula over any characteristic $0$ field. To do this, we use a motivic Euler characteristic $\chi^{mot}$ taking values in $\widehat{\mathrm{W}}(k)$. The field of arithmetic curve counting has its origins in $\mathbb{A}^1$-homotopy theory and $\chi^{mot}$ is no different. The original definition of $\chi^{mot}$ comes from the categorical Euler characteristic in the stable motivic homotopy category over our base field. A theorem of Levine and Raksit in \cite{LR} allows us to compute this Euler characteristic geometrically when our varieties are smooth and projective. By Theorem 2.13 of \cite{ABOWZ}, this extends uniquely to a motivic measure on the ring of varieties, making it an excellent candidate for curve counting, and since the motivic Euler characteristic can be defined geometrically, we may use geometric techniques in our proofs.

 To state our arithmetic Yau--Zaslow formula, we need some notation.
\begin{notn}
Let $k$ be a characteristic $0$ field, and let $X$ and $\mathcal{C}$ be as above. Without loss of generality, we may assume that $k$ is finitely generated over $\Q$, and let $k^{const}$ denote the algebraic closure of $\Q$ in $k$. Suppose that the rational curves in $\mathcal{C}$ are nodal, which is generically the case by Chen's theorem, Theorem 1.1 of \cite{Ch}. For $C$ a curve, let $\mathcal S(C)$ be the set of singular points of $C$. Let $p \in \mathcal S(C)$ be a nodal point so that the tangents to $C$ at $p$ are either defined over $k(p)$ or a quadratic extension, and write $\alpha_p$ for the element of $k(p)^\times/(k(p)^\times)^2$ such that the tangents are defined over $k(p)(\sqrt{\alpha_p})$. For $q$ a closed point of $\mathbb{P}^g$, write $C_q$ for the curve lying over $q$ and $G(X)$ for the set of closed points of $\mathcal{C}$ such that $C_q$ is rational. Let $X^{(m)}$ denote the $m^{\text{th}}$ symmetric power of $X$. Let $\Delta(g)$ be the integer from Theorem $\ref{bettibound}$. Let $\Delta^{k^{const}}_g$ be the group of order $2$ characters of $\Gal_{k^{const}}$ that are unramified at all primes of $k^{const}$ whose residue characteristic $p$ satisfies $p > \Delta(g)\cdot [k^{const}:\Q]+1$, and let $\Delta^k_g$ denote the image of $\Delta_g^{k^{const}}$ under the map $\Hom(\Gal_{k^{const}}, \Z/2\Z) \to \Hom(\Gal_k, \Z/2\Z)$. Let $J_g$ denote the subset of elements $q$ of $\widehat{\mathrm{W}}(k)$ such that $\mathrm{rank}(q)=0$, $\mathrm{sign}(q)=0$, and $\mathrm{disc}(q) \in \Delta^k_g$, which is an ideal by Definition $\ref{jgdefinition}$. For $L/k$ a finite field extension, let $\mathrm{Tr}_{L/k}: \widehat{\mathrm{W}}(L) \to \widehat{\mathrm{W}}(k)$ be the trace transfer map defined by taking a quadratic form $q$ over $L$ to $\mathrm{Tr}_{L/k} \circ q$, where $\mathrm{Tr}_{L/k}$ denote the Galois trace. Let $N_{L/k}: \widehat{\mathrm{W}}(L) \to \widehat{\mathrm{W}}(k)$ be Rost's norm transfer map from Corollary 5 of \cite{Ro}. Let $\mathcal{J}$ be the ideal of $\widehat{\mathrm{W}}(k)[[t]]$ such that $t^g[q] \in \mathcal{J}$ if and only if $[q] \in J_g$. Finally define
$$
\mathbf B^{mot}_{\mathcal{C}}(X) :=  \sum_{q \in G(X)} \mathrm{Tr}_{k(q)/k} \left( \prod_{p\in\mathcal S(C_q)}\left( 1+ N_{k(p)/k(q)}([\langle -2 \rangle] - [\langle 2\alpha_p \rangle])\right) \right) \in \widehat{\mathrm{W}}(k)/J_g.
$$
\end{notn}

\begin{thm}[(Arithmetic Yau--Zaslow formula)]\label{main}
The element $\mathbf{B}^{mot}_{\mathcal{C}}(X)$ is given by $\langle (-1)^g \rangle$ multiplied by the coefficient of $t^g$ in the power series
$$
\prod_{d=1}^\infty \sum_{m=0}^\infty \chi^{mot}(X^{(m)})\langle (-1)^d \rangle t^{md} \in \widehat{\mathrm{W}}(k)[[t]]/\mathcal{J}.
$$
\end{thm}
Because the ideal $J_g$ is sufficiently small and because we compute $\chi^{mot}(X^{(m)})$ modulo $J$ in Corollary $\ref{totsymm}$, we recover both the complex Yau--Zaslow formula of \cite{Be} and the real Yau--Zaslow formula of \cite{KR1} by taking the rank and signature of the above formula respectively. We also obtain a Yau--Zaslow type formula by taking the discriminant of the above equation, modulo a finite error term in $\Delta^k_g$. Our proof closely follows the steps laid out in Beauville's orignal proof (see \cite{Be}) of the classical Yau--Zaslow formula, in that there are four main steps:
\begin{enumerate}
\item[$(i)$] the computation of the motivic Euler characteristic of compactified Jacobians of curves, which captures both the number of rational curves on the surface, as well as arithmetic information about their singularities,
\item[$(ii)$] applying a Fubini theorem, which says that the motivic Euler characteristic of the compactified Jacobian of the whole linear system is expressible in data that solely depends on the rational curves,
\item[$(iii)$] applying a Batyrev-Kontsevich type theorem to relate the motivic Euler characteristic of the compactified Jacobians to the motivic Euler characteristic of certain Hilbert schemes,
\item[$(iv)$] applying a G\"ottsche formula to compute the motivic Euler characteristic of these Hilbert schemes in terms of the Euler characteristic of the surface.
\end{enumerate}

\begin{cont} The next three sections are dedicated to proving properties of the motivic Euler characteristic. In section $2$, we define our motivic Euler characteristic, before showing its relationship to other invarants including Levine's categorical Euler characteristic from \cite{Le} and Saito's determinant of cohomology from \cite{S0}. In section $3$, we show a Galois descent type result for the motivic Euler characteristic. Section $4$ is concerned with changing the base field, and showing that the motivic Euler characteristic is compatible with the natural transfer maps on the Grothendieck ring of varieties and the Grothendieck--Witt ring respectively by using the results from section $3$. Sections $5$ to $9$ are concerned with the proof strategy laid out above. In section $5$, we prove our version of a Fubini theorem by using an argument of Saito from \cite{S0}. The Fubini theorem we obtain isn't perfect: as well as information directly related to the motivic Euler characteristic of the singular curves, we also need to take into account the monodromy action around these curves. The monodromy action gives rise to a Saito $\epsilon$-factor for computing the discriminant of the motivic Euler characteristic. We therefore compute these Saito factors, and show that for a given $g$ these $\epsilon$ factors lie in $\Delta^k_g$, so in particular there are only finitely many possibilities. In section $6$, we use a stratification of compactified Jacobians of curves in terms of their partial normalisations in order to directly compute their motivic Euler characteristics, and then we can apply the Fubini theorem to obtain an expression for the whole compactified Jacobian. Section $7$ is concerned with a Batyrev--Kontsevich type theorem, allowing us to relate the motivic Euler characteristic of the compactified Jacobian of the linear system to the Hilbert scheme of points of the K3 surface. We prove this over real closed fields by using a model theoretic argument to reduce to the real case, which follows by Proposition 3.1 of \cite{KR1}. For the rank and discriminant, we adapt Batyrev's original $p$-adic integration argument from \cite{Ba}. In section $8$, we derive a motivic G{\"o}ttsche formula, which allows us to compute the motivic Euler characteristic of these Hilbert schemes in terms of the motivic Euler characterstic of the underlying K3 surface by showing that we can apply a result of de Cataldo and Migliorini from \cite{dCM}. Finally in section $9$, we put these steps together to arrive at the arithmetic Yau--Zaslow formula, Theorem $\ref{main}$. We then show that Theorem $\ref{main}$ recovers the complex Yau--Zaslow formula of \cite{Be}, the real Yau--Zaslow formula of \cite{KR1}, and we deduce a new formula for discriminants in Corollary $\ref{yau-zaslow5}$.
\end{cont}

\subsection{Notation and conventions}
We fix the following conventions for this paper.

\begin{itemize}
\item  We let $k$ denote the base field for our Yau--Zaslow formula. For this paper, we will often restrict to $\mathrm{char}(k)=0$, and we will always restrict to $\mathrm{char}(k) \neq 2$. 
\item Fields denoted by $\F$ will be finite fields which normally arise through looking at the reduction of a model of our variety.
\item The symbol $\ell$ will always denote a prime number different than $\mathrm{char}(k)$ and $\mathrm{char}(\F)$. 
\item If $k$ is a field, a variety over $k$ will mean a reduced separated scheme of finite type over $\Spec(k)$. We will not require varieties to be irreducible. We will consider the empty scheme $\emptyset=\Spec(\{0\})$ to be a variety over $k$.
\item All orderings on fields will be assumed to be total orderings that are compatible with the operations on fields.  A real closed field will mean a field equipped with an ordering such that no algebraic extension of the field can be equipped with an ordering.
\item A Calabi--Yau variety will always mean a smooth projective variety with trivial canonical bundle. In particular, a $K3$ surface is a $2$-dimensional, simply connected, Calabi--Yau variety. 
\item A morphism of varieties will always denote a finitely presented morphism $f: X \to Y$ which is compatible with the structure maps to the base.
\end{itemize}
We also fix the following notation for the whole paper. 

\begin{tabular}{p{1.7cm}|p{13cm}}
$k^\times$ &The multiplicative group of $k$.\\
$\kbar$& A fixed separable closure of $k$.\\
$\Gal_k$& The absolute Galois group $\Gal(\kbar/k)$ of $k$.\\
$\widehat{\mathrm{W}}(k)$ & The Grothendieck--Witt ring of $k$ (see Definition $\ref{gwring}$). \\
$K_0(\mathrm{Var}_k)$ & The Grothendieck ring of varieties over $k$ (see Definition $\ref{k0vark}$). \\
$g$ & A fixed positive integer.\\
$\mathcal{C}$ & A linear system of curves of genus $g$ lying on the $K3$ surface.\\
$\mathbb{G}_{m,k}$& The split one dimensional affine torus, $\Spec(k[x,x^{-1}])$.
\end{tabular}
We should note that a choice of separable closure does not affect the results in this paper, since all results are isomorphism invariant.

\subsection{Acknowledgements}

The first author wishes to thank Johannes Nicaise, Tomer Schlank, and Toby Gee, for reading through early outlines of this work and giving helpful feedback, as well as the support of Imperial College London, and the University of Canterbury. Parts of this work appeared as part of the first author's thesis, and the first author would also like to thank Kirsten Wickelgren and Kevin Buzzard for their careful reading of the thesis in their role as examiners. This work was undertaken while the first author was supported by the London School of Geometry and Number Theory under the Engineering and Physical Sciences Research Council grant number [EP/S021590/1], and was revised while the author was supported by the Marsden Fund administered by the Royal Society of New Zealand. The second author wishes to thank Frank Neumann for several discussions about the contents of this work, and for his interest. He also wishes to acknowledge the generous support of the Imperial College Mathematics Department’s Platform Grant which made the visit of Frank Neumann to Imperial College possible. We both thank the anonymous referees for their careful reading of the paper and the suggested proof of Lemma $\ref{hodgederhamcomparison}$. We thank Stephen McKean and Dori Bejleri for helpful comments. We thank Richard Thomas for his discussions on the modern view of the complex Yau--Zaslow formula. We thank Marc Levine and the rest of the study group in Essen for pointing out some errors in earlier versions of this paper.

\section{The motivic Euler characteristic and related invariants}
In this section we first recall the definition of the motivic Euler characteristic from \cite{ABOWZ} that forms the basis for our arithmetic curve counting before showing how it relates to three different invariants, namely, the complex compactly supported Euler characteristic, the real compactly supported Euler characteristic, and the determinant of the $\ell$-adic cohomology.
\begin{defn}\label{gwring} Let $k$ be a field whose characteristic is different from $2$. For every $a\in k$ let $\langle a\rangle$ denote the quadratic form of rank $1$ such that $\langle a \rangle(\alpha) = a\alpha^2$ for any $\alpha \in k$. For every pair $q,q'$ of quadratic forms let $q\oplus q',q\otimes q'$ denote the orthogonal direct sum and the tensor product of $q$ and $q'$, respectively. The {\it Grothendieck--Witt ring} $\widehat{\mathrm{W}}(k)$ of quadratic forms over $k$ is the ring generated by symbols $[q]$, where $q$ is any quadratic form over $k$, subject to the relations $[q]=[q']$, when $q$ and $q'$ are isomorphic, the relations $[q\oplus q']=[q]+[q']$ and $[q\otimes q']=[q]\cdot[q']$, when $q,q'$ are any pair of quadratic forms over $k$, and the relation $[\langle0\rangle]=0$. 

Define $\mathbb{H} := \langle 1 \rangle + \langle -1 \rangle$. The \emph{Witt ring}, denoted by $\mathrm{W}(k)$, is the quotient $\widehat{\mathrm{W}}(k)/(\mathbb{H})$. We say an element $q \in \widehat{\mathrm{W}}(k)$ is \emph{hyperbolic} if it lies in the ideal generated by $\mathbb{H}$, which is the same as the abelian subgroup consisting of elements of the form $n\mathbb{H}$ for $n \in \Z$. 

For a non-degenerate quadratic form $q$, the {\it rank} of $q$ is the dimension of the underlying vector space, denoted by rank$(q)$. It extends uniquely to a ring homomorphism $\mathrm{rank}:\widehat{\mathrm{W}}(k)\to\mathbb Z$. Write $I$ for the kernel of this homomorphism. The subset $T\subset\widehat{\mathrm{W}}(k)$ of torsion elements is also an ideal, and let $J=I^2\cap T$.
\end{defn}

\begin{defn}\label{DREulerCharacteristic} Let $X$ be a smooth, irreducible projective variety over $k$ of dimension $n$. Let $H_{dR}^*(X/k)$ denote the de Rham cohomology of $X$ over $k$. Then the composition of the cup-product square $a \mapsto a \cup a$ and the trace map $\mathrm{Tr}:H_{dR}^*(X/k)\to k$ defines a quadratic form on the vector space $H_{dR}^{2*}(X/k)$, which we may see is non-degenerate using Poincaré duality. Let $[H_{dR}^{2*}(X/k)]$ denote the class of this quadratic form in $\widehat{\mathrm{W}}(k)$, and set
$$\chi^{dR}(X)=[H_{dR}^{2*}(X/k)]-\frac{1}{2}\big(\sum_{i\geq0}
\dim(H_{dR}^{2i+1}(X/k))\big)\cdot\mathbb H\in\widehat{\mathrm{W}}(k),$$
which we will call the {\it de Rham Euler characteristic} of $X$. 
\end{defn}
In this paper, we will often work with the de Rham Euler characteristic due to its compatibility with results of Saito from \cite{S0}. However, we may compare the de Rham Euler characteristic to other Euler characteristics taking values in $\widehat{\mathrm{W}}(k)$. 

\begin{defn}\label{HodgeEulerCharacteristic}
Fix $X/k$ to be a smooth, irreducible projective of dimension $n$.  Following Construction 1.2 of \cite{LR}, we define the \emph{Hodge Euler characteristic} of $X$ as follows.

Consider the trace map $\mathrm{Tr}: H^d(X, \Omega^d_{X/k}) \to k$. We can define a cohomologically bounded complex of $k$ vector spaces $\mathrm{Hdg}(X/k)$ such that differentials are all zero and the vector space in degree $n$ is given by $\bigoplus_{j=i+n} H^i(X, \Omega^j_{X/k})$. The cup product composed with the trace map then furnishes a nondegenerate symmetric bilinear form $\mathrm{Tr}$ on $\mathrm{Hdg}(X/k)$ as with the de Rham Euler characteristic. This therefore defines a class in the $0^{\text{th}}$ Grothendieck--Witt group of the exact category of cohomologically bounded cochain complexes of $k$ vector spaces in the sense of Section 2.2 of \cite{Sc}. Since cochain complexes of $k$ vector spaces are always split, we can identify this group with $\widehat{\mathrm{W}}(k)$.  We will write $\chi^{Hdg}(X)$ to mean the element $(\mathrm{Hdg}(X/k), \mathrm{Tr}) \in \widehat{\mathrm{W}}(k)$ constructed above.
\end{defn}

\begin{lemma}\label{hodgederhamcomparison}
Let $X/k$ be an irreducible smooth projective variety over $k$. Then we may identify $\chi^{dR}(X)=\chi^{Hdg}(X) \in \widehat{\mathrm{W}}(k)$.
\end{lemma}
The following elegant proof was suggested to us by one of the referees of this paper. There is an alternative, more complicated, proof of this statement based on an argument by David Speyer (\cite{Sp}), by closely studying the Hodge-de Rham spectral sequence.
\begin{proof}
We proceed by considering $X \times_k \mathbb{A}^1_k$, where parameterise $\mathbb{A}^1_k$ by a co-ordinate $t$. Let $p_1: X \times_k \mathbb{A}^1_k \to X$ denote projection onto the first factor and $p_2: X \times_k \mathbb{A}^1_k \to \mathbb{A}^1_k$ denote projection onto the second factor. Consider the de Rham complex $(\Omega^*_{X/k}, d)$ on $X$. Fibrewise cup product and the trace map furnishes a quadratic form on $Rp_{2*}(p_1^*\Omega^*_{X/k}, t\cdot d)$ in the category of cohomologically bounded cochain complexes of coherent sheaves on $\mathbb{A}^1_k$. This gives rise to an element of the $0^{\text{th}}$ Grothendieck--Witt group of this abelian category in the sense of Section 2.2 of \cite{Sc}, which we denote by $\widehat{\mathrm{W}}(\mathrm{Ch}^b\mathrm{Coh}(\mathcal{O}_{\mathbb{A}^1_k}))$. Since $\mathbb{A}^1_k$ is a separated noetherian scheme, we may argue as in Theorem 8.2 of \cite{We} to identify $\widehat{\mathrm{W}}(\mathrm{Ch}^b\mathrm{Coh}(\mathcal{O}_{\mathbb{A}^1_k})) = \widehat{\mathrm{W}}(\mathbb{A}^1_k) := KO^{[0]}_0(\mathbb{A}^1_k)$, where the Hermitian $K$-theory group $KO^{[0]}_0(\mathbb{A}^1_k)$ is as defined as in Section 5 of \cite{PW}. Similarly by the same logic as the previous definition we can identify $\widehat{\mathrm{W}}(\mathrm{Ch}^b\mathrm{Coh}(\mathcal{O}_{\Spec(k)})) = \widehat{\mathrm{W}}(k)$, which in turn we may identify with $KO^{[0]}_0(\Spec(k))$ by the definition of $KO$ as in Section $4$ of \cite{PW}.

 Let $i_j:\mathrm{Spec}(k)\to\mathbb A^1_K$ be the inclusion at the points $j=0,1$.  By the proper base-change theorem, we have an equality in the Grothendieck--Witt ring of $k$:
$$
\widehat{\mathrm{W}}(\mathrm{Ch}^b\mathrm{Coh}(\mathcal{O}_{\Spec(k)})) \ni i_1^* Rp_{2*}(p_1^*\Omega^*_{X/k}, t\cdot d)=
\bigoplus_n H^n_{dR}(X)[-n] = \chi^{dR}(X) \in \widehat{\mathrm{W}}(k)
$$
with cup product and trace form to $k$. Over $t=0$, the differential vanishes and so
$$\widehat{\mathrm{W}}(\mathrm{Ch}^b\mathrm{Coh}(\mathcal{O}_{\Spec(k)})) \ni i_0^* Rp_{2*}(p_1^*\Omega^*_{X/k}, t\cdot d)= \bigoplus_{p,q}H^p(X,\Omega^q)[-p-q] = \chi^{Hdg}(X) \in \widehat{\mathrm{W}}(k)$$
again with cup product and trace form to $k$, where we note that there is a potential even shift in the grading in the second complex, but this does not affect the class in $\widehat{\mathrm{W}}(k)$. Since the Hermitian $K$-theory groups are representable by a spectrum in the stable motivic homotopy category of $k$ by Theorem 1.1 of \cite{PW}, they are $\mathbb A^1$-invariant. Therefore $\widehat{\mathrm{W}}(k) = KO^{[0]}_0(\Spec(k)) = KO^{[0]}_0(\mathbb{A}^1_k)$, and $i^*_0, i^*_1$ induce identifications $KO^{[0]}_0(\mathbb{A}^1_k) = \widehat{\mathrm{W}}(k)$ with $i_0^* = i_1^*$. In particular, there is an equality in $\widehat{\mathrm{W}}(k)$:
$$
\chi^{dR}(X) = i_0^*(Rp_{2*}(p_1^*\Omega^*_{X/k}, t\cdot d))= i_1^*(Rp_{2*}(p_1^*\Omega^*_{X/k}, t\cdot d)) = \chi^{Hdg}(X),
$$
as required.
\end{proof}

\begin{defn} Remark 2.1 (2) of \cite{Le} defines the {\it categorical Euler characteristic} $\chi^{cat}(X)$ for an arbitrary algebraic variety $X$ over $k$ by using the infinite suspension spectrum $\Sigma_T^\infty X_+$  as an object of $\textrm{SH}(k)$, the stable motivic homotopy category over $k$. Since $\Sigma_T^{\infty} X_+$ is dualisable, it gives us a canonical endomorphism of the unit object $\mathbb S_k$
$$\chi^{cat}(X)\in
\textrm{End}_{\textrm{SH}(k)}(\mathbb S_k).$$
By Morel's theorem, Theorem 6.4.1 of \cite{Mo}, we can identify $\textrm{End}_{\textrm{SH}(k)}(\mathbb S_k)\cong\widehat{\mathrm{W}}(k)$ when $k$ has characteristic $0$, and $\textrm{End}_{\textrm{SH}(k)}(\mathbb S_k)\cong\widehat{\mathrm{W}}(k)[\frac1p]$ when $k$ has positive odd characteristic.
\end{defn}

One of the main theorems of \cite{LR} is the following.
\begin{thm}[(Theorem 1.3 of \cite{LR})]\label{gauss-bonet} Let $X$ be a smooth, projective variety over $k$. Then there is an equality of Euler characteristics $\chi^{cat}(X)=
\chi^{Hdg}(X)$.
\end{thm}
\begin{cor}
 Let $X$ be a smooth, projective variety over $k$. Then there is an equality of Euler characteristics $\chi^{cat}(X)=
\chi^{dR}(X)$.
\end{cor}
\begin{proof}
Combine the above theorem with Lemma $\ref{hodgederhamcomparison}$.
\end{proof}
Assume $\mathrm{char}(k)=0$ for the rest of the paper, unless otherwise stated.
\begin{defn}\label{k0vark} Let $k$ be any field. Let $K_0(\mathrm{Var}_k)$ denote the {\it Grothendieck ring of varieties over $k$}. This is the ring generated by symbols $[X]$, where $X$ is any variety over $k$, subject to the relation $[X]=[X']$, when $X$ and $X'$ are isomorphic, the relation $[X]=[X-X']+[X']$, when $X'$ is a closed subvariety in $X$, and $[X\times X']=[X]\cdot[X']$, when $X,X'$ are any pair of varieties over $k$. The $0$ element of this ring is $[\emptyset]$, and the multiplicative unit is $[\Spec(k)]$. 

A \emph{motivic measure} is a ring homomorphism out of $K_0(\mathrm{Var}_k)$. 
\end{defn}
The above Euler characteristic does not give us a motivic measure; see Remark \ref{not_measure}. However, Theorem 2.13 of \cite{ABOWZ} gives the following.
\begin{thm}\label{euler} There is a unique ring homomorphism $\chi^{mot}:K_0(\mathrm{Var}_k)\to\widehat{\mathrm{W}}(k)$ such that for every smooth, irreducible, projective variety $X/k$
$$\chi^{mot}([X])=\chi^{Hdg}(X) = \chi^{dR}(X).$$
\end{thm}
A key ingredient in the proof of the above is Bittner's presentation of the ring of varieties, Theorem 3.1 of \cite{Bi}, which tells us that the underlying abelian group of $K_0(\mathrm{Var}_k)$ can be generated by smooth projective varieties, subject to only a smooth blow-up excision relation. In particular, the above theorem requires us to restrict to characteristic $0$.

\begin{rem}\label{not_measure} 
We can show that the categorical Euler characteristic is not a motivic measure. For example, consider $\mathbb{P}^1$. Let $\{\infty\}$ denote the closed point at infinity. Then $X \setminus \{\infty\} \cong \mathbb{A}^1$, so in $K_0(\mathrm{Var}_k)$, $[\mathbb{P}^1] = [\mathbb{A}^1]+[\Spec(k)]$.  Proposition 2.5 of \cite{Le} computes $\chi^{cat}(\mathbb{P}^1) = \langle 1 \rangle + \langle -1 \rangle$, but since $\mathbb{A}^1$ is clearly $\mathbb{A}^1$-contractible, we have $\chi^{cat}(\mathbb{A}^1) = \chi^{cat}(\Spec(k)) = \langle 1 \rangle$. Therefore $\chi^{cat}(\mathbb{P}^1) \neq \chi^{cat}(\mathbb{A}^1) + \chi^{cat}(\Spec(k))$ whenever $\langle -1 \rangle \neq \langle 1 \rangle$, i.e.~whenever $-1$ is not a square in $k$. 
\end{rem}

For the rest of the paper we will write $\chi^{mot}(X)$ for $\chi^{mot}([X])$ for every variety $X$ over $k$ and call this the {\it motivic Euler characteristic} of $X$. This is a slight abuse of notation, since $\chi^{mot}$ depends on $k$. When the base field is ambiguous, we will write $\chi^{mot}_k$, however for the most part we will omit the subscript for readability. 

\begin{defn}\label{3.8dim}Suppose the dimension $n$ of $X$ is even. The composition of the cup-product and the trace map $\mathrm{Tr}:H_{dR}^{2n}(X/k)\to k$ defines a non-degenerate quadratic form on $H_{dR}^{n}(X/k)$. Let $[H_{dR}^{n}(X/k)]\in\widehat{\mathrm{W}}(k)$ denote the class of this quadratic form. Define integers $b^+$ and $b^-$ by
\begin{align*}
b^{-}&:=\sum_{i<n}\dim_k H_{dR}^{i}(X/k)\\
b^+ &:= \sum_{i < n} (-1)^i \dim_k H_{dR}^{i}(X/k).
\end{align*}
\end{defn}
\begin{lemma}\label{even} Suppose that $X$ has even dimension $n$. Then
$$\chi^{dR}(X)=b^{+}\cdot\mathbb H+
[H_{dR}^{n}(X/k)].$$
If $X$ has odd dimension, then $\chi^{dR}(X) = \left(b^+ - \frac{1}{2}\mathrm{dim}_k(H^n_{dR}(X/k))\right)\cdot \mathbb{H}$.
\end{lemma}
\begin{proof} 
Suppose $n$ is even, and consider the quadratic form on $H_{dR}^{2*}(X/k)$. This decomposes as an orthogonal direct sum of the middle term $H_{dR}^{n}(X/k)$ and of $H_{dR}^{2i}(X/k)\oplus H_{dR}^{2n-2i}(X/k)$ as the integer $i$ runs between $0$ and $n-1$. For every such $i$ the direct sum $H_{dR}^{2i}(X/k)\oplus H_{dR}^{2n-2i}(X/k)$ is a non-degenerate subspace of $H_{dR}^{2*}(X/k)$ of rank $2\dim(H_{dR}^{2i}(X/k))$ and $H_{dR}^{2i}(X/k)$ is an isotropic subspace of dimension $\dim(H_{dR}^{2i}(X/k))$. Therefore the quadratic form $H_{dR}^{2i}(X/k)\oplus H_{dR}^{2n-2i}(X/k)$ is hyperbolic, so is isomorphic to the orthogonal sum of $\dim(H_{dR}^{2i}(X/k))$ copies of $\mathbb H$, so the claim is clear in the even case. The odd case follows by an identical argument.
\end{proof}

We are interested in recovering invariants related to quadratic forms from the Grothendieck--Witt ring. The motivic Euler characteristic encodes classical invariants on varieties, the first being the compactly supported Euler characteristic of the complex points.

\begin{defn}
 For every variety $X$ over $k$ let $e(X)$ denote the compactly supported Euler characteristic of the base change $X_{\overline{k}}$ of $X$ to $\overline k$. 
\end{defn}
\begin{thm}\label{levine1} Let $\mathrm{rank}:\widehat{\mathrm{W}}(k) \to \mathbb Z$ be the rank homomorphism from Definition $\ref{gwring}$. We have $e(X)=\mathrm{rank}(\chi^{mot}(X))$ for every algebraic variety $X$ over $k$.
\end{thm}
\begin{rem} 
We may use any Weil cohomology theory to compute the compactly supported Euler characteristic and we will get the same value, for example, we may work with $\ell$-adic cohomology. Independence from the cohomology theory is underlined by Theorem $\ref{gauss-bonet}$, since $\chi^{mot}(X)$ only depends on $\chi^{cat}(X)$; an invariant that comes from the stable motivic homotopy category. For $\ell$-adic cohomology, Laumon proved in \cite{Lau} that for any finite type separated scheme over an algebraically closed field and any prime $\ell$ different from the characteristic, the $\ell$-adic Euler characteristic and the compactly supported $\ell$-adic Euler characteristic are equal.
\end{rem}
\begin{proof}[Proof of Theorem $\ref{levine1}$]
For $X$ a projective variety over $\mathbb{C}$, this is proven in Remark 2.3 (1) of \cite{Le}. The general case can be reduced to this case as follows. The assignment $X\mapsto e(X)$ gives rise to a ring homomorphism $ K_0(\mathrm{Var}_k) \to\mathbb Z$. Since $K_0(\mathrm{Var}_k)$ is generated by smooth projective varieties by Theorem 3.1 of \cite{Bi}, it is enough prove the claim for $X$ smooth and projective. Let $X$ be such a variety and let $F\subset k$ be a subfield which is finitely generated over $\mathbb{Q}$ such that $X$ is already defined over $F$, i.e.~there is a smooth projective variety $Y$ over $F$ whose base change to $k$ is $X$. Fix an embedding of $F$ into $\mathbb C$, which exists since $F$ is countable. Using the naturality of the categorical Euler characteristic with respect to base change, and the invariance of the rank function on the Grothendieck--Witt ring under base change, we get that
$$\mathrm{rank}(\chi^{mot}(X))=\mathrm{rank}(\chi^{mot}(Y\times_Fk))=
\mathrm{rank}(\chi^{mot}(Y))=\mathrm{rank}(\chi^{mot}(Y\times_F\mathbb C)).$$
By proper base change for our choice of Weil cohomology theory we have
$$e(X)=e(Y\times_Fk)=e(Y)=e(Y\times_F\mathbb C).$$
Applying Remark 2.3 (1) of \cite{Le} gives $\mathrm{rank}(\chi^{mot}(Y\times_F\mathbb C))=e(Y\times_F\mathbb C)$, as required.
\end{proof}

As well as the complex points of our variety, the motivic Euler characteristic allows us to recover information about the real points of our variety and natural generalisations of this. 

\begin{defn} Assume $k$ is a real closed field. Then by Sylvester's law of inertia, every non-degenerate quadratic form $q$ over $k$ is isomorphic to a diagonal form
$$\underbrace{\langle1\rangle\oplus\langle1\rangle\oplus
\cdots\oplus\langle1\rangle}_m\oplus
\underbrace{\langle-1\rangle\oplus\cdots\oplus\langle-1\rangle}_n,$$
and $m-n$ only depends on the isomorphism class of $q$. This is called the {\it signature} of $q$ and we write $\mathrm{sign}(q)$. It induces an isomorphism $\mathrm{sign}:\mathrm{W}(k)\to \mathbb Z.$
\end{defn}
Now let $k$ be again any field whose characteristic is not two.

\begin{defn}\label{realspec} The {\it real spectrum} Spr$(k)$ of $k$ is the set of all orderings of the field $k$, see Definition 1.2 of Chapter 8, \S1 of \cite{La}. Note that we may have $\mathrm{Spr}(k)=\emptyset$, for example if $-1$ is the sum of squares in $k$. In particular $\mathrm{Spr}(k)=\emptyset$ if $k$ has positive characteristic. Consider Spr$(k)$ as a topological space by equipping it with the {\it Harrison topology}, which may be a slight abuse of notation if $\mathrm{Spr}(k) = \emptyset$. The definition can be found in Chapter 8, \S6 of \cite{La}, which we recall here for convenience. Let $S\subset k$ be a finite subset and let $\prec$ be a total ordering on the set $S$. Let $U(S,\prec)\subset\mathrm{Spr}(k)$ be the set consisting of all orderings $<$ on $k$ such that $<$ restricted to $S$ is $\prec$. The sets $U(S,\prec)$ form a sub-basis of the Harrison topology, which makes Spr$(k)$ a compact and totally disconnected topological space.
\end{defn}
\begin{defn}\label{2.13a} For a topological space $X$ and a ring $A$ let $\mathcal C(X,A)$ denote the ring of continuous $A$-valued functions, where we give $A$ the discrete topology. For $<$ in $\mathrm{Spr}(k)$ let $k_<$ denote the real closure of the ordered field $(k,<)$. For every non-degenerate quadratic form $q$ over $k$ let $\mathrm{sign}(q):\mathrm{Spr}(k)\to\mathbb Z$ be the function sending each $<$ in $\mathrm{Spr}(k)$ to the signature of the base change of $q$ to $k_<$. This is continuous with respect to the Harrison topology, so we get a ring homomorphism:
$$\mathrm{sign}:\mathrm{W}(k)\to\mathcal C(\mathrm{Spr}(k),\mathbb Z),$$
where we note that if $\mathrm{Spr}(k)=\emptyset$, then $\mathcal C(\mathrm{Spr}(k), \mathbb{Z}) = \{0\}$ so $\mathrm{sign}$ is the trivial homomorphism. By slight abuse of notation let sign$:\widehat{\mathrm{W}}(k)\to\mathcal C(\mathrm{Spr}(k),\mathbb Z)$ denote the composition of the quotient map $\widehat{\mathrm{W}}(k)\to
\mathrm{W}(k)$ with the map sign above.
\end{defn}

\begin{defn}\label{2.13b} Assume again that $k$ is a real closed field, and let $X$ be an algebraic variety over $k$. Then $X(k)$ is a semi-algebraic set over $k$, so its {\it compactly supported Euler characteristic} $\chi^{rc}(X(k))\in\mathbb Z$ can be defined using real algebraic geometry (see for example Section 11.2 of \cite{BCR} on pages 266-67). Now let $k$ be again any field whose characteristic is not two, and let $X$ be as above. We will write $\epsilon(X):\mathrm{Spr}(k)\to\mathbb Z$ for the function which assigns each $<$ in $\mathrm{Spr}(k)$ to the compactly supported real Euler characteristic of $X(k_<)$. 
\end{defn}
\begin{thm}\label{levine2}For every algebraic variety $X$ over $k$ we have
$$\mathrm{sign}(\chi^{mot}(X))=\epsilon(X).
$$
\end{thm}
\begin{proof} If $k$ has positive characteristic, this is trivial, since it has no real closures, so without loss of generality assume $k$ is characteristic $0$. When $k=\mathbb R$ and $X$ is projective, this is Remark 2.3 (2) of \cite{Le}. For $k$ a general field, using the naturality of the motivic Euler characteristic with respect to base change, we may assume without loss of generality that $k$ is real closed by passing to the real closure of $(k,<)$ for every point $<$ in $\mathrm{Spr}(k)$. Since the compactly supported Euler characteristic over real closed fields is a motivic measure, we may reduce to the case where $X$ is smooth and projective as in the proof of Theorem \ref{levine1}, so we can use the categorical Euler characteristic instead. 

As we already remarked, for every quasi-projective variety $X$ over $k$ the set $X(k)$ has a structure of semi-algebraic set over $k$. Let $S_\bullet(X(k))$ denote the {\it simplicial set of semi-algebraic singular simplexes} of the semi-algebraic space $X(k)$ of $k$-valued points of $X$ as defined in Definition 2.5 of \cite{Pa2}. The assignment $X\mapsto S_\bullet(X(k))$ induces a functor $S_\bullet$ from the category of smooth quasiprojective varieties $Sm/k$ to $s\mathbf{Set}$, the category of simplicial sets. Let $\mathrm{SH}$ denote the stable homotopy category of simplicial sets. The functor $S_\bullet$ gives rise to a realisation functor $\mathrm{Re}_B:\textrm{SH}(k)\to\textrm{SH}$ which for $X\in Sm/k$ sends the suspension spectrum $\Sigma_T^\infty X_+$ to the suspension spectrum $\Sigma^{\infty}(S_\bullet(X(k)))$. This induces a map $\sigma:\textrm{End}_{\textrm{SH}(k)}(\mathbb S_k)\to\textrm{End}_{\textrm{SH}}(\mathbb S_k)\cong\mathbb Z$. We claim that $\sigma$ is the signature map. It is enough to show that $\langle-1\rangle$ goes to $-1$. Using the naturality of the constructions we only need to check this for any real closed subfield $L \subseteq k$. By taking $L$ to be the algebraic closure of $\mathbb{Q}$ in $k$, assume without loss of generality that $L$ is the real closure $\mathbb{Q}^{real}$ of $\mathbb{Q}$. Since $\mathbb R$ contains $\mathbb{Q}^{real}$, we may use naturality again to reduce to the case  $k=\mathbb R$. In this case, Remark 2.3 (2) of \cite{Le} applies, since the forgetful map from the simplicial set of semi-algebraic singular simplexes to the usual simplicial set of continuous singular simplexes is a homotopy equivalence by Theorem 2.20(ii) of \cite{Pa2}. 
\end{proof}

\begin{rem} By Pfister's local-global principle (see Theorems 3.2, 4.9 and 6.9 of Section VIII of \cite{La}) the induced map
$$\mathrm{sign}\otimes\mathrm{id}_{\mathbb Z[\frac{1}{2}]}:\mathrm{W}(k)\otimes\mathbb Z[\tfrac{1}{2}]\to
\mathcal C(\mathrm{Spr}(k),\mathbb Z[\tfrac{1}{2}])$$
is an isomorphism. Therefore the invariants rank and sign compute the Grothen\-dieck--Witt ring up to torsion, so the motivic Euler characteristic recovers both the algebraically closed Euler characteristic in the form of the rank, as well as the real closed Euler characteristic in the form of the signature. Conversely, the motivic Euler characteristic, up to torsion, can be computed from these invariants. The full motivic Euler characteristic contains interesting torsion information which cannot be recovered from these invariants. We turn our attention to some torsion information about $\chi^{mot}$ which also relates to a more classical invariant.
\end{rem}

\begin{defn} Assume that $k$ is a field whose characteristic is not $2$. Then every non-degenerate quadratic form $q$ is isomorphic to a diagonal form
$$\langle a_1\rangle\oplus\langle a_2\rangle\oplus
\cdots\oplus\langle a_n\rangle.$$
The {\it discriminant} disc$(q)$ of the form $q$ above is defined as the class of the product
$$a_1\cdot a_2\cdots a_n\textrm{ in }k^\times/k^{\times2}.$$
This is independent of the diagonalisation of $q$, since it is the determinant of the symmetric matrix defining the bilinear form associated to $q$. It induces a homomorphism of the underlying abelian groups
$$\mathrm{disc}:\widehat{\mathrm{W}}(k)
\to k^\times/k^{\times2}.$$
We will freely use the identification $k^\times/k^{\times2} = \Hom(\Gal_k, \Z/2\Z)$.
\end{defn}
For $X$ a variety over $k$, we would like to relate $\mathrm{disc}(\chi^{mot}(X))$ to other invariants on $X$. 

\begin{defn}\label{2.1.1}
Let $\ell$ be a prime number different from the characteristic of $k$, let $U$ be a smooth scheme over $k$, and let $\mathcal{F}$ be a smooth $\ell$-adic sheaf on $U$. Following Saito in \cite{S0} we will consider the one-dimensional $\ell$-adic representation of $\Gal_k$ given by the determinant 
$$\det R\Gamma_c(U_{\overline{k}}, \mathcal{F}):=\bigotimes_i \det H^i_c(U_{\overline{k}}, \mathcal{F})^{\otimes(-1)^i}.$$
The representation above is called the {\it determinant of cohomology} of
$\mathcal F$, and for $\mathcal{F}=\Q_{\ell}$, we will call it \emph{the $\ell$-adic determinant of cohomology} of $X$. This is a $1$ dimensional $\ell$-adic representation, so we will view it as an element of the abelian group $\Hom(\Gal_k, \Q_{\ell}^\times)$.  For $\mathcal{F}=\mathbb{Q}_{\ell}$, we will write $\mathrm{det}_\ell(U)$ to mean $\det R\Gamma_c(U_{\overline{k}}, \mathbb{Q}_\ell)$.

Let $X$ be a smooth proper variety over $k$ of dimension $n$. As shown in Lemma 5 of \cite{S0}, Poincar\'e duality allows us to obtain the following identity
$$
\det R\Gamma(X_{\overline{k}}, \mathbb{Q}_{\ell})=\mathbb{Q}_{\ell}(-\frac{1}{2}ne(X))\otimes \kappa(X)
$$
where $\kappa(X)$ is a character of order at most $2$ of $\Gal_k$, which is independent of $\ell$. Suppose that $n$ is odd. As in Lemma 5 of \cite{S0}, the existence of a non-degenerate alternating form on $H^n_{\et}(X_{\kbar},\Q_{\ell})$ given by the cup product allows us to obtain that $\mathrm{det}( H^n_{\et}(X_{\overline{k}}, \Q_{\ell}))=\Q_{\ell}( -j)$ where $j=\frac12 \mathrm{dim}_{\Q_{\ell}} (H^n_{\et}(X_{\kbar}, \Q_{\ell}))$. In particular, $\kappa(X)=1$ if $\mathrm{dim}(X)$ is odd.
\end{defn}

\begin{thm}\label{etale1} For every smooth projective variety $X$ over $k$ of dimension $n$, we have
$$\mathrm{disc}(\chi^{cat}(X))=\mathrm{disc}(\chi^{dR}(X))=(-1)^{\frac12 n \cdot e(X)}\cdot\kappa(X).$$
\end{thm}
If $n$ is odd, then $e(X)$ is even, so $\frac12 n \cdot e(X)$ is always a positive integer. For $X$ a $K3$ surface, this formula gives $\mathrm{disc}(\chi^{mot}(X))=\kappa(X)$. This result is essentially equivalent to a theorem of Saito from \cite{S0}. First we are going to formulate this as Saito did. The main result of Saito's paper \cite{S0} linking the determinant of cohomology with de Rham cohomology is the following. 
\begin{thm}[(Saito's theorem, \cite{S0})]\label{saito1} Assume that $X$ has even dimension $n$. Then
$$\kappa(X)=(-1)^{\frac12e(X)+b^{-}}\cdot
\mathrm{disc}([H_{dR}^{n}(X/k)]).$$
\end{thm}
\begin{proof} This is Theorem 2 on page 412 of \cite{S0}. 
\end{proof}
\begin{proof}[Proof of Theorem \ref{etale1}]
By Lemma $\ref{even}$ the result is trivial if $\mathrm{dim}(X)$ is odd, since we have $\chi^{dR}(X) = \frac12 e(X) \cdot \mathbb{H}$. Therefore, suppose $\mathrm{dim}(X)$ is even.

 We can work with $\chi^{dR}$ instead of $\chi^{cat}$, by Theorem \ref{gauss-bonet}. For all $q,q' \in \widehat{\mathrm{W}}(k)$, we have that the discriminant $\mathrm{disc}(q+q')=\mathrm{disc}(q)\cdot\mathrm{disc}(q')$. Applying the map disc to both sides of the equation in Lemma \ref{even} therefore gives
\begin{align*}
\mathrm{disc}(\chi^{cat}(X))=\mathrm{disc}(\chi^{dR}(X))&=
\mathrm{disc}(\mathbb H)^{b^{+}}\cdot\mathrm{disc}([H_{dR}^{n}(X/k)])\\
&=(-1)^{b^{-}}\cdot\mathrm{disc}([H_{dR}^{n}(X/k)]),
\end{align*}
where we use that $b^{-}\equiv b^{+}\!\!\mod2$. The claim follows by Theorem \ref{saito1}.\end{proof}

We want to extend Theorem $\ref{etale1}$ to all varieties by replacing $\chi^{dR}$ with $\chi^{mot}$, and we do this by showing that the terms in Theorem $\ref{etale1}$ extend to motivic measures. Assume $\mathrm{char}(k)=0$.
\begin{lemma} The $\ell$-adic determinant of cohomology $\mathrm{det}_{\ell}: K_0(\mathrm{Var}_k)\to\mathrm{Hom}(\Gal_k,\mathbb{Q}_{\ell}^\times)$ is a homomorphism of abelian groups.
\end{lemma}
\begin{proof} Let $X$ be a variety over $k$, let $j:U\to X$ be an open immersion, and let $i:Z\to X$ the closed immersion which is the complement of $j$. We claim $\mathrm{det}_{\ell}(X)=\mathrm{det}_{\ell}(U)\cdot
\mathrm{det}_{\ell}(Z)$. Consider the short exact sequence of sheaves on $X$
$$0 \to  j_!(\mathbb{Q}_{\ell}) \to
\mathbb{Q}_{\ell} \to
i_*(\mathbb{Q}_{\ell}) \to 0.$$
This short exact sequence gives rise a distinguished triangle:
$$R_c\Gamma(X_{\overline{k}},j_!(\mathbb{Q}_{\ell}))
\to
R_c\Gamma(X_{\overline{k}}, \mathbb{Q}_{\ell})
\to
R_c\Gamma(X_{\overline{k}},i_*(\mathbb{Q}_{\ell}))
\to 
R_c\Gamma(X_{\overline{k}},j_!(\mathbb{Q}_{\ell}))[1],$$
so taking the determinant gives 
$$\det R_c\Gamma(X_{\overline{k}}, \mathbb{Q}_{\ell})=
\det R_c\Gamma(X_{\overline{k}},j_!(\mathbb{Q}_{\ell}))\cdot
\det R_c\Gamma(X_{\overline{k}},i_*(\mathbb{Q}_{\ell})).$$
Since $j_!$ is exact we have $\det R_c\Gamma(U_{\overline{k}},\mathbb{Q}_{\ell})= 
\det R_c\Gamma(X_{\overline{k}},j_!(\mathbb{Q}_{\ell}))$, and the proper base change theorem gives $\det R_c\Gamma(Z_{\overline{k}}, \mathbb{Q}_{\ell})= \det R_c\Gamma(X_{\overline{k}},i_*(\mathbb{Q}_{\ell}))$, as required.
\end{proof}

\begin{propn}\label{bidefn} For each non-negative integer, $i$, there is a unique homomorphism of abelian groups $b_i: K_0(\mathrm{Var}_k)\to\mathbb Z$ which extends the $i$-th Betti number of smooth projective varieties.
\end{propn}
\begin{proof} 
Theorem 4.5 of \cite{NS} gives rise to a ring homomorphism 
$$
P(-, t): K_0(\mathrm{Var}_k) \to \Z[t]
$$
 such that if $X$ is a smooth projective variety, $P(X,t) = \sum_{i=0}^{\infty} (-1)^i b_i(X) t^i$, where $b_i(X)$ denotes the $i^{\text{th}}$ Betti number of $X$.  Let $c_i: \Z[t] \to \Z$ be the homomorphism of abelian groups given by taking $\sum_i \alpha_i t^i$ to $(-1)^i \alpha_i$. Defining $b_i$ to be the composition of $P(-,t)$ with $c_i$ gives the result.
\end{proof}
\begin{defn}\label{totweight} Since $b_i(X)=0$ for $i>2n$ for every smooth and projective of dimension $n$, every formal linear combination $\sum_{i\in\mathbb N}c_ib_i$ where $c_i \in \mathbb{R}$ is also well-defined and gives rise to a unique homomorphism of abelian groups $K_0(\mathrm{Var}_k) \to\mathbb R$. When $X$ is smooth and projective of dimension $n$ then
\begin{align*}
\sum_{i\in\mathbb N}\frac{(-1)^ii}{2}b_i&=\frac{(-1)^nn}{2}b_n+
\sum_{i<n}\left(\frac{(-1)^ii}{2}b_i+\frac{(-1)^{2n-i}(2n-i)}{2}b_{2n-i}\right)\\
&=\frac{(-1)^nn}{2}b_n+
\sum_{i<n}\left(\frac{(-1)^ii}{2}b_i+\frac{(-1)^{i}(2n-i)}{2}b_i\right)
\\
&=\frac{n}{2}\left((-1)^nb_n+\sum_{i<n}(-1)^i2b_i\right)=
\frac{e(X)n}{2}
\end{align*}
using Poincar\'e duality. This gives a unique homomorphism of abelian groups $w: K_0(\mathrm{Var}_k)\to\mathbb R$ extending the function $w(X) = \frac{e(X)n}2$ on smooth projective varieties over $k$, which we call the {\it total weight}. Since $\frac{e(X)n}{2}$ is an integer for $X$ smooth and projective, Theorem 3.1 of \cite{Bi} tells us that $w(X)$ is always an integer. 
\end{defn}

\begin{thm}\label{saito-full} For every algebraic variety $X$ over $k$ we have:
$$\mathrm{disc}(\chi^{mot}(X))(-1)^{w(X)}=\mathrm{det}_{\ell}(X)\cdot
\mathbb{Q}_{\ell}(w(X)) \in \Hom(\Gal_k, \Q_{\ell}^\times).$$
\end{thm}
\begin{proof} The left hand side lies in $k^\times/k^{\times2}$, which we identify with $\Hom(\Gal_k, \{\pm1\})$ so that we may view it as a subgroup of $\Hom(\Gal_k, \Q_{\ell}^\times)$. By Theorem \ref{etale1} the relation holds when $X$ is smooth and projective. Since both sides of the equation are homomorphisms of abelian groups, the general case follows from the presentation of $K_0(\mathrm{Var}_k)$ from Theorem 3.1 of \cite{Bi}.
\end{proof}
\begin{rem}\label{wxtorsion}
The above gives us an alternative way of characterising $w(X)$ when $k$ is finitely generated over $\Q$. In this case, the cyclotomic character has infinite order so $w(X)$ is the unique integer such that $\mathrm{det}_\ell(X) \cdot \Q_\ell(w(X))$ is torsion. Theorem $\ref{saito-full}$ then guarantees that this character is $2$-torsion, so can be identified with an element of $k^\times/(k^{\times2}$, and we may recover $\mathrm{disc}(\chi^{mot}(X))$ as $(-1)^{w(X)}$ multiplied by this element. In particular, we may recover $\mathrm{disc}(\chi^{mot}(X))$ and $w(X)$ both from $\mathrm{det}_\ell(X)$. 
\end{rem}

\begin{cor} Assume $k$ is finitely generated over $\mathbb{Q}$. Then for every variety $X$ over $k$, we have $\det_\ell(X)$ is trivial if and only if both $w(X)$ and $\mathrm{disc}(\chi^{cat}(X))$ are trivial. 
\end{cor}
\begin{proof} By Theorem \ref{saito-full} it will be sufficient to show that $w(X)$ vanishes when $\det_\ell(X)$ vanishes. Assume now the latter. By the observation in Definition $\ref{2.1.1}$ the composition of $\det_\ell(X)$ and the map
$\mathbb{Q}^\times_{\ell}\to\mathbb{Q}^\times_{\ell}$ given by squaring each element is the character $\mathbb{Q}_{\ell}(-2w(X))$ when $X$ is smooth and projective. Therefore since $\mathrm{det}_{\ell}$ and $w(-)$ are homomorphisms of abelian groups, this holds for general varieties by Theorem 3.1 of \cite{Bi}. Since the cyclotomic character has infinite order when $k$ is finitely generated over $\mathbb{Q}$ we get that $-2w(X)=0$, and hence $w(X)=0$. 
\end{proof}

\begin{rem}
It is often easier to prove results about the $\ell$-adic determinant of cohomology than the motivic Euler characteristic. The remainder of this section is dedicated to showing that we can deduce results about the motivic Euler characteristic from the $\ell$-adic determinant of cohomology. To do this, we introduce the following terminology so that we obtain ring homomorphisms related to the discriminant, the $\ell$-adic determinant of cohomology, and total weight.
\end{rem}
\begin{defn} We say that a pair $(R,\epsilon)$ is an {\it augmented ring} if $R$ is a commutative ring with unity and $\epsilon$ is a ring homomorphism $R\to\mathbb Z$ such that the composition of the canonical ring homomorphism $\mathbb Z\to R$ and $\epsilon$ is an isomorphism. If in addition $\mathrm{Ker}(\epsilon)^2=0$, then we say that $ (R,\epsilon)$ is {\it elementary}.
\end{defn}
\begin{notn} Let $\mathbf W(k)$ denote the quotient of $\widehat {\mathrm{W}}(k)$ by $I^2$. Then the homomorphism $\mathrm{rank}: \widehat{\mathrm{W}}(k)\to\mathbb Z$ furnishes an augmentation $\epsilon: \mathbf W(k)\to\mathbb Z$. Since we see that $\mathrm{Ker}(\epsilon)^2=0$, the pair $ (\mathbf W(k),\epsilon)$ is elementary. Moreover, the discriminant function $\widehat{\mathrm{W}}(k) \to \Hom(\Gal_k, \Z/2\Z)$ factors through $\mathbf W(k)$, and when restricted to $I/I^2$, it furnishes a group isomorphism $\mathrm{Ker}(\epsilon)\cong I/I^2\to
\mathrm{Hom}(\Gal_k,\mathbb Z/2\mathbb Z)$. 
\end{notn}
\begin{example}\label{augmentedhom} For every $\mathbb Z$-module $M$ let $E(M)$ denote the ring with underlying additive group $\mathbb Z\oplus M$ and multiplication $(a,m)\cdot(b,n)=(ab,an+bm)$. If we let $\epsilon:E(M)\to\mathbb Z$ be the projection onto the first factor, then $(E(M),\epsilon)$ is an elementary augmented ring. This construction is functorial, in that a group homomorphism $E(M)\to E(N)$ induces a ring homomorphism $M\to N$ respecting the augmentations. 
\end{example}
\begin{rem}\label{augdet} For $(R,\epsilon)$ an augmented ring, we have a decomposition of the underlying additive group $R=\mathbb Z\oplus\mathrm{Ker}(\epsilon)$. When $(R,\epsilon)$ is elementary, the decomposition above is a natural isomorphism of augmented rings between $R$ and $E(\mathrm{Ker}(\epsilon))$. In particular, we get a natural isomorphism $\mathbf W(k)\cong E(\mathrm{Hom}(\Gal_k,\mathbb Z/2\mathbb Z))$. Write $E_{disc}$ for the quotient map $\widehat{\mathrm{W}}(k) \to E(\mathrm{Hom}(\Gal_k,\mathbb Z/2\mathbb Z))$, and similarly write $\chi^{disc}: K_0(\mathrm{Var}_k) \to E(\mathrm{Hom}(\Gal_k,\mathbb Z/2\mathbb Z))$ to denote the composition $E_{disc} \circ \chi^{mot}$. 
\end{rem}
\begin{defn}\label{chil} Let $\chi_{\ell}$ and $\chi_w$ denote the maps
$$\chi_{\ell}: K_0(\mathrm{Var}_k)\to E(\mathrm{Hom}(\Gal_k,\mathbb{Q}_{\ell}^\times))
\quad\textrm{and}\quad
\chi_w:  K_0(\mathrm{Var}_k) \to E(\mathbb Z)$$
given by:
$$[X]\mapsto(e(X),\mathrm{det}_{\ell}(X))
\quad\textrm{and}\quad
[X]\mapsto(e(X),w(X))$$
respectively. We will call $\chi_{\ell}$ and $\chi_w$ the {\it augmented determinant of cohomology} and the {\it augmented total weight}, respectively. 
\end{defn}
\begin{propn}\label{elemental} The maps
$$\chi_{\ell}: K_0(\mathrm{Var}_k)\to E(\mathrm{Hom}(\Gal_k,\mathbb{Q}_{\ell}^\times))
\quad\textrm{and}\quad
\chi_w:  K_0(\mathrm{Var}_k)\to E(\mathbb Z)$$
are ring homomorphisms. 
\end{propn}
\begin{proof} We may assume without loss of generality that $k$ is finitely generated over $\mathbb{Q}$. Define $s:E(\mathrm{Hom}(\Gal_k,\mathbb{Q}_{\ell}^\times))
\to E(\mathrm{Hom}(\Gal_k,\mathbb{Q}_{\ell}^\times))$ to be the ring homomorphism induced by the group homomorphism $\mathrm{Hom}(\Gal_k,\mathbb{Q}_{\ell}^\times)\to\mathrm{Hom}(\Gal_k,\mathbb{Q}_{\ell}^\times)$
induced by the squaring map
$\mathbb{Q}^\times_{\ell}\to\mathbb{Q}^\times_{\ell}$. Let $c:E(\mathbb Z)
\to E(\mathrm{Hom}(\Gal_k,\mathbb{Q}_{\ell}^\times))$ be the ring homomorphism induced by the unique group homomorphism $\mathbb Z
\to\mathrm{Hom}(\Gal_k,\mathbb{Q}_{\ell}^\times)$ such that the image of $-1$ is the cyclotomic character. Then $c\circ\chi_w=s\circ\chi_{\ell}$. Since $c$ is injective we get that $\chi_w$ is a ring homomorphism if $\chi_{\ell}$ is. Since $\chi_{\ell}$ is a group homomorphism, we only need that for any $X,Y$, we have
$$\mathrm{det}_{\ell}(X\times Y)=\mathrm{det}_{\ell}(X)^{\otimes e(Y)}\otimes
\mathrm{det}_{\ell}(Y)^{\otimes e(Y)}.$$
Using the K\"unneth formula:
\begin{align*}
\mathrm{det}_{\ell}(X\times Y)
&=\bigotimes_n
\det H^n_c(X_{\overline{k}}\times Y_{\overline{k}},\mathbb{Q}_{\ell})^{\otimes(-1)^n}\\
&=\bigotimes_n\bigotimes_{i+j=n}
\left(\det H^i_c(X_{\overline{k}},\mathbb{Q}_{\ell})\otimes
H^j_c(Y_{\overline{k}},\mathbb{Q}_{\ell})\right)^{\otimes(-1)^n}
\\
&=\bigotimes_{i,j}
\det H^i_c(X_{\overline{k}},\mathbb{Q}_{\ell})
^{\otimes(-1)^{i+j}b_j(Y)}\otimes
\det H^j_c(Y_{\overline{k}},\mathbb{Q}_{\ell})
^{\otimes(-1)^{j+i}b_i(X)}
\\
&=\bigotimes_i
\det H^i_c(X_{\overline{k}},\mathbb{Q}_{\ell})
^{\otimes(-1)^ie(Y)}\otimes\bigotimes_j
\det H^j_c(Y_{\overline{k}},\mathbb{Q}_{\ell})
^{\otimes(-1)^je(X)},
\end{align*}
so the claim is now clear.
\end{proof}

\begin{propn}\label{rel} Assume that $k$ is finitely generated over $\mathbb{Q}$. Let $X,Y,Z$ be elements of $K_0(\mathrm{Var}_k)$ such that $\chi_{\ell}(X)=\chi_{\ell}(Y)\cdot\chi_{\ell}(Z)$. Then $\chi^{mot}(x)=\chi^{mot}(y)\cdot\chi^{mot}(z)\mod I^2$ and $\chi_w(X) = \chi_w(Y)\chi_w(Z)$.
\end{propn}
\begin{proof}Since $\chi_{\ell}(X) = \chi_{\ell}(Y) \chi_{\ell}(Z)$, we have $e(X) = e(Y)e(Z)$. This allows us to say the following.
\begin{align*}
\chi_{\ell}(X)=\chi_{\ell}(Y)\cdot\chi_{\ell}(Z)
&\Longleftrightarrow
\mathrm{det}_{\ell}(X)=\mathrm{det}_{\ell}(Y)^{\otimes e(Z)}\otimes
\mathrm{det}_{\ell}(Z)^{\otimes e(Y)},\\
\chi_w(X)=\chi_w(Y)\cdot\chi_w(Z)
&\Longleftrightarrow
w(X)=e(Z)w(Y)+e(Y)w(Z),\\
\chi^{mot}(X)=\chi^{mot}(Y)\cdot\chi^{mot}(Z)
\!\!\!\!\mod I^2
&\Longleftrightarrow\\
\mathrm{disc}(\chi^{mot}&(X))
=\mathrm{disc}(\chi^{mot}(Y))^{e(Z)}\cdot\mathrm{disc}(\chi^{mot}(Z))^{e(Y)},
\end{align*}
using that all three augmented rings $E(\mathrm{Hom}(\Gal_k,\mathbb{Q}_{\ell}^\times)),E(\mathbb Z)$ and
$\mathbf W(k)$ are elementary. By Theorem $\ref{saito-full}$, it is enough to show that $\chi_{\ell}(X)=\chi_{\ell}(Y)\cdot\chi_{\ell}(Z)$ implies
$\chi_w(X)=\chi_w(Y)\cdot\chi_w(Z)$. Let $c$ and $s$ be the homomorphisms from the proof of Proposition \ref{elemental}. Since $c\circ\chi_w=s\circ\chi_{\ell}$ and $c$ is injective, the latter is clear.  
\end{proof}

The main reason it is easier to work with the $\ell$-adic determinant of cohomology rather than the determinant is that we may relate the $\ell$-adic determinant of cohomology to the zeta functions of reductions of the variety over finite fields.
For the purpose of this paper, it is often helpful to have a formulation of the $\ell$-adic determinant of cohomology of a variety in terms of $\zeta$-functions of the reduction of the variety to finite fields.
\begin{notn} Let $F$ be a field finitely generated over $\mathbb{Q}$. Let $X$ and $Y$ be two quasi-projective varieties defined over $F$. Let $U$ be a smooth connected scheme of finite type over $\mathrm{Spec}(\mathbb Z)$ with function field $F$ and choose models $\mathcal X,\mathcal Y$ for $X,Y$ over $U$, respectively. 

Since $U$ is of finite type over $\Spec(\mathbb{Z})$, if $s$ is a closed point of $U$ then the residue field $\F_s$ of $s$ is finite. Write $\mathbb F_{s^n}$ to be the unique extension of $\mathbb F_s$ of degree $n$. Consider the variety $\mathcal{X}_s/\F_s$.  let $\zeta(\mathcal X_s, t) := \mathrm{exp}\left(  \sum_{n \geq 1} \# \mathcal X_s(\mathbb F_{s^n}) t^n \right)$. 

By the proof of the Weil conjectures (see for example Section 7.5.7 of \cite{Po}), we have
$$
\zeta(\mathcal X_s, t) =\prod_i\det(\textrm{Id}-t\cdot\mathrm{Frob}_s|_{H^i_c(\overline{\mathcal{X}_s}, \mathbb{Q}_{\ell})})^{(-1)^i},
$$
where $\mathrm{Frob}_s$ denotes the geometric Frobenius at $s$ acting on the compactly supported $\ell$-adic cohomology, and $\ell$ is any prime different from the characteristic of $\mathbb F_s$. 
\end{notn}
We claim that the $\ell$-adic determinant of cohomology of $X$ is determined by the $\zeta$-functions of the reductions $\mathcal{X}_s$. In order to prove this we first need some linear algebra. 
\begin{notn}\label{gradedtrdet}
Let $V=\bigoplus_iV_i$ be a finite dimensional graded vector space over a field $L$ of characteristic zero, let $\phi:V\to V$ be an invertible $L$-linear map respecting the grading, and define
$$\mathrm{Tr}(\phi) :=\sum_i(-1)^i\mathrm{Tr}(\phi|_{V_i}),\quad\det(\phi):=\prod_i\det(\phi|_{V_i})^{(-1)^i}.$$
Define $e(V) := \mathrm{Tr}(\mathrm{Id}_V)$. Finally, define
$$\zeta(\phi,t)=\prod_i\det(\textrm{Id}_{V_i}-t\cdot\phi|_{V_i})^{(-1)^i}.$$
\end{notn}
\begin{lemma} We have
$$\zeta(\phi,t)=\mathrm{det}(\phi)(-t)^{e(V)}\cdot u,\quad  \text{where }u\in1+L[[t^{-1}]].$$
\end{lemma}
\begin{proof} Since $\phi$ is the direct sum of the restrictions $\phi|_{V_i}$, and the terms $\zeta(\phi,t),\mathrm{det}(\phi)$ and $(-t)^{e(V)}$ are all multiplicative with respect to direct sums, we may assume without loss of generality that only one graded piece of $V$ is non-zero. Shifting the grading by $j$ changes these terms by raising them to the $(-1)^j$-th power, and the operation of raising to the $(-1)^j$-th power preserves power series in $L[[t^{-1}]]$, so we may also assume that $V=V_0$. Note that
$$
\det(\textrm{Id}-t\cdot\phi) = (-t)^{\mathrm{dim}(V)} \cdot \det(\phi - t^{-1}\textrm{Id}) = (-t)^{\mathrm{dim}(V)} c_\phi(t^{-1}),
$$
where $c_{\phi}$ denotes the characteristic polynomial of $\phi$, and so $c_{\phi}(t^{-1}) = \mathrm{det}(\phi) + t^{-1}u'$ for some polynomial $u' \in L[t^{-1}]\subseteq L[[t^{-1}]]$.  Since $\phi$ is invertible, we may relabel $u'' = \frac{1}{\mathrm{det}{\phi}}u'$ and obtain
$$
\det(\textrm{Id}-t\cdot\phi) = (-t)^{\mathrm{dim}(V)} \cdot \mathrm{det}(\phi) \cdot (1 + t^{-1}u''),
$$
which gives the result. 
\end{proof}

\begin{prop} Let $X$ be a quasi-projective variety over a finite field $\mathbb F$ and let $\ell$ be a prime different from $\mathrm{char}(\mathbb F)$. Then
$$\zeta(X,t)=\mathrm{det}_{\ell}(\mathrm{Frob}_{\mathbb{F}}) \cdot (-t)^{\chi(X)}\cdot u,\quad u\in1+\mathbb{Q}_{\ell}[[t^{-1}]],$$
where $\chi(X)$ is the compactly supported Euler characteristic of $X$, $\mathrm{Frob}_{\mathbb{F}} \in\mathrm{Gal}(\overline{ \mathbb F}/\mathbb F)$ is the geometric Frobenius and $\mathrm{det}_{\ell}(X)$ is the $1$-dimensional representation given by the determinant of cohomology of $X$ as in Definition $\ref{2.1.1}$. 
\end{prop}
\begin{proof}
Let $V:= \bigoplus_i H^i_c(X_{\kbar}, \Q_\ell)$ be the compactly supported $\ell$-adic cohomology of $X$, which is a graded vector space over $\mathbb{Q}_{\ell}$, and let $\phi$ be the action of the geometric Frobenius $\mathrm{Frob}_{\mathbb{F}}$ on $V$. Then there are equalities $e(V)=\chi(X)$, $\zeta(\phi,t)=\zeta(X,t)$ and $\det(\phi)=\mathrm{det}_{\ell}(\mathrm{Frob}_{\mathbb{F}})$, so the claim follows. 
\end{proof}
\begin{thm}\label{zetafunctionsdetermine} Assume that for every closed point $s$ of U we have an equality of $\zeta$-functions $\zeta(\mathcal X_s,t)=\zeta(\mathcal Y_s,t)$. Then there is an equality of $\Gal_F$ representations $\mathrm{det}_\ell(X) = \mathrm{det}_\ell(Y)$.
\end{thm}
\begin{proof}
Let $s$ be a closed point of $U$. By the proper base change theorem for $\ell$-adic cohomology and the previous Lemma, if we pick $\mathrm{Frob}_s \in \Gal_F$ to be a lift of the Frobenius of $\F_s$, we may obtain $\mathrm{det}_{\ell}(X)(\mathrm{Frob}_{\F_s})$ from $\zeta(\mathcal X_s,t)$. By assumption $\zeta(\mathcal X_s,t) = \zeta(\mathcal Y_s, t)$, and so $\mathrm{det}_{\ell}(X)(\mathrm{Frob}_{s}) = \mathrm{det}_{\ell}(Y)(\mathrm{Frob}_{s})$. The Chebotar\"ev density theorem then shows that we have an equality of $\Gal_F$ representations $\mathrm{det}_\ell(X) = \mathrm{det}_\ell(Y)$.
\end{proof}

We will often prove facts about the motivic Euler characteristic by proving results separately for the rank, signature and discriminant, then combining them to obtain an element in $\widehat{\mathrm{W}}(k)/J$. This is made possible by the following lemma.

\begin{lemma}\label{modJ}
Let $[q] \in \widehat{\mathrm{W}}(k)$. Then $[q] \in J$ if and only if all of the discriminant, signature and rank vanish. In particular, if $[q]$ is an element of $\widehat{\mathrm{W}}(k)$, we can reconstruct the class of $[q]$ in $\widehat{\mathrm{W}}(k)/J$ from $\mathrm{rank}([q]), \mathrm{sign}([q])$ and $\mathrm{disc}([q])$.
\end{lemma}
\begin{proof}
This is clear by definition of $J:=I^2 \cap T$, since $T$ is the intersection of $\mathrm{ker}(\mathrm{rank})$ and $\mathrm{ker}(\mathrm{sign})$, and $I^2 = \mathrm{ker}(\widehat{\mathrm{W}}(k) \to \mathbf{W}(k))$, and $q$ lies in this kernel if and only if $\mathrm{rank}(q)$ and $\mathrm{disc}(q)$ both vanish.
\end{proof}
\begin{cor}\label{ringhommodj}
There is an injective ring homomorphism
$$
\mathrm{rank} \times \mathrm{sign} \times E_{disc}: \widehat{\mathrm{W}}(k)/J \to \mathbb Z \times \mathcal{C}(\mathrm{Spr}(k), \mathbb{Z}) \times E(\mathrm{Hom}(\Gal_k, \mathbb Z/2\mathbb Z)).
$$
\end{cor}
\begin{proof}
This is essentially a restatement of the above Lemma.
\end{proof}

\section{Galois descent for the motivic Euler characteristic}
In this section, we give a detailed account of a Galois descent type result for $\chi^{mot}$. Fix $K/k$ a finite Galois extension with Galois group $\Gal(K/k)$. If $k$ has positive characteristic, assume further that $[K:k]$ is coprime to $\mathrm{char}(k)$.

\begin{defn}\label{DescentDataVar}
Let $\mathrm{Var}_{K, desc}$ denote the category of pairs $(X, \varphi)$ of $K$-varieties $X$ equipped with descent data $\varphi$. Explicitly, the descent data $\varphi$ is a group homomorphism $\Gal(K/k) \to \mathrm{Aut}(X/k)$ such that for all $\tau \in \Gal(K/k)$ we have a commutative diagram
\begin{center}
\begin{tikzcd}
X \ar[r, "\varphi(\tau)"] \ar[d] & X \ar[d] \\
\Spec(K) \ar[r, "\tau"] & \Spec(K),
\end{tikzcd}
\end{center}
where the vertical arrows denote the structure morphism $X \to \Spec(K)$. A morphism in $\mathrm{Var}_{K, desc}$, $f: (X, \varphi) \to (X', \varphi')$ is a morphism of $K$-varieties $f: X \to X'$ which is compatible with the descent data in the sense that the obvious diagram commutes. We say a morphism is an \emph{open} (resp. \emph{closed}) \emph{immersion} if the underlying morphism of varieties is an \emph{open} (resp. \emph{closed}) \emph{immersion}. This category also has a unit object with respect to the monoidal structure induced by the product: namely, the variety $\Spec(K)$ equipped with the identity as the structure morphism, with the only possible descent data. We will write $(\Spec(K), \mathrm{Id})$ for this unit object. This allows us to form the Grothendieck ring of $K$ varieties with descent data, $K_0(\mathrm{Var}_{K, desc})$, which is the ring generated by symbols $[(X, \varphi)]$, where $(X,\varphi)$ is a $K$ variety with descent data, subject to the relations
\begin{enumerate}
\item $[(X, \varphi)] = [(X', \varphi')]$ if $(X,\varphi)$ and $(X', \varphi')$ are isomorphic,
\item $[(X, \varphi)] = [(U, \varphi|_U)] + [( X \setminus U, \varphi_{X\setminus U})]$ for $(U, \varphi_U) \hookrightarrow (X, \varphi)$ an open immersion, where the descent data on $X \setminus U$ is given by the restriction of the descent data on $X$,
\item $[(X, \varphi)] [(X', \varphi')] = [( X \times_K X', \varphi \times \varphi')]$ for any pair of $K$ varieties with descent data.
\end{enumerate}

By Galois descent (see for example, Corollary 4.4.6 of \cite{Po}), there is an equivalence of categories $\mathrm{Var}_{K, desc} \cong \mathrm{Var}_k$. This equivalence of categories preserves the unit object in the category since it sends $(\Spec(K), \mathrm{Id})$ to $\Spec(k)$, where $(\Spec(K), \mathrm{Id})$. If $(X, \varphi)$ is a $K$ variety with descent data, write $\mathrm{des}_{K/k}(X,\varphi)$ to mean the associated $k$ variety, and similarly if $Y/k$ is a variety, write $(Y_K, \phi_Y)$ to mean the associated $K$ variety with descent data.

This equivalence of categories preserves open and closed immersions (see Theorem 4.3.7 of \cite{Po}), so it induces an isomorphism $K_0(\mathrm{Var}_{K, desc}) \cong K_0(\mathrm{Var}_k)$ given on the level of quasiprojective varieties by sending $[(X, \varphi)] \mapsto [\mathrm{desc}_{K/k}(X, \varphi)]$, with inverse $[X] \mapsto [(X_K, \phi_X)]$. 
\end{defn}

We'd like to apply a Galois descent style result for the motivic Euler characteristic. In order to do this, we need a Galois descent result for quadratic forms. 
\begin{defn}
Let $K$ be a finite Galois extension of $k$, let $V$ be a vector space over $K$ with a semilinear $\Gal(K/k)$ action, and let $q: V \rightarrow K$ be a quadratic form. We say $q$ is {\it Galois compatible} if, for all $\tau \in \Gal(K/k)$ and $v \in V$, we have $q(\tau(v))=\tau(q(v)).$

Write $\varphi: \Gal(K/k) \to \mathrm{Aut}_k(V)$ for the group homomorphism defining this semilinear action. We will write $(q, \varphi)$ to denote the quadratic form on $V$ along with the semilinear $\Gal(K/k)$ action on $V$ induced by $\varphi$.  Let $\widehat{\mathrm{W}}(K)^{desc}$ denote the ring generated by the symbols $[(q, \varphi)]$, where $[(q, \varphi)]$ is a Galois compatible quadratic form over $K$, subject to the relations:
\begin{enumerate}
\item $[(q,\varphi)] = [(q',\varphi')]$ if there is an isomorphism between the underlying vector spaces of $q,q'$ which is compatible with the quadratic forms and the semilinear $\Gal(K/k)$ actions.
\item $[(q, \varphi)] + [(q', \varphi')] = [(q \oplus q', \varphi \oplus \varphi')]$, where $q \oplus q'$ (resp. $\varphi \oplus \varphi'$) represents the natural quadratic form (resp. semilinear $\Gal(K/k)$ action) on $V \oplus V'$.
\item $[(q, \varphi)]\cdot[(q',\varphi')] = [(q \otimes_K q', \varphi \otimes_K \varphi')]$, where $q \otimes_K q'$ (resp. $\varphi \otimes_K \varphi'$) represents the natural quadratic form (resp. semilinear $\Gal(K/k)$ action) on $V \otimes_K V'$. 
\end{enumerate}
For $a \in k^\times$, write $\langle a \rangle$ to mean the Galois compatible quadratic form on $K$ given by $x \mapsto ax^2$, and write $\mathbb{H} := \langle 1 \rangle + \langle -1 \rangle \in \widehat{\mathrm{W}}(K)^{desc}$.
\end{defn}
\begin{thm}\label{GaloisDescentQuadForm} There is an equivalence of categories between the category of non-degenerate quadratic forms over $k$ and the category of non-degenerate Galois compatible quadratic forms over $K$.
\end{thm}
\begin{proof}
This follows by Galois descent for vector spaces. If $q: V \rightarrow k$ is a quadratic form over $k$, we obtain a Galois compatible quadratic form over $K$
$$
q \otimes Id_K: V \otimes_k K \rightarrow K,
$$
where $\varphi: \Gal(K/k) \to \mathrm{Aut}_k(V \otimes_k K)$ is given by $\tau \mapsto \mathrm{Id} \otimes \tau$.

If $(q', \varphi) : V' \rightarrow K$ is a Galois compatible quadratic form over $K$, and $v'$ is an element of $V'^{\Gal(K/k)}$, then $\tau(q'(v')) = q'(\tau(v')) = q'(v')$ for all $\tau \in \Gal(K/k)$, which implies that $q'(v') \in k$. Restricting to the Galois invariant subspace gives us a quadratic form over $k$: 
$$
q'|_{V'^{\Gal(K/k)}}: V'^{\Gal(K/k)} \rightarrow k.
$$
W claim that $q'|_{V'^{\Gal(K/k)}}$ is non-degenerate. Let $B: V' \otimes_K V' \to K$ be the underlying bilinear form, and let $v \in V'^{\Gal(K/k)}$. Since $B'$ is non-degenerate, pick $w \in V'$ such that $B(v,w)=1$. Let $\hat{w} = \frac{1}{[K:k]} \sum_{\tau \in \Gal(K/k)} \tau(w)$. Then $\hat{w} \in V'^{\Gal(K/k)}$, and $B'(v,\hat{w}) = 1$, so $B'|_{V'^{\Gal(K/k)}}$ is non-degenerate. Finally, note that the functors $q \mapsto q \otimes Id_K$ and $q' \mapsto q'|_{V'^{\Gal(K/k)}}$ are quasi-inverse, which completes the proof.
\end{proof}
\begin{cor}\label{descentisomorphismgw}
There is a canonical isomorphism $\widehat{\mathrm{W}}(K)^{desc} \cong \widehat{\mathrm{W}}(k)$, given on the level of quadratic forms by sending $(q, \varphi)$ to $q|_{V^{\Gal(K/k)}}$.
\end{cor}
\begin{proof}
Immediate by the above.
\end{proof}

We claim that Galois descent for varieties is compatible with a Galois descent result for the Galois-compatible quadratic forms on their de Rham cohomology. This is made exact in the following.
\begin{lemma}\label{DescentCompat}
Let $X$ be a smooth projective $k$ variety. Then there is an isomorphism of graded vector spaces with a semilinear $\mathrm{Gal}(K/k)$ action:
$$
H^{*}_{dR}(X/k) \otimes_k K \cong H^{*}_{dR}(X_K/K).
$$
Let $\cup_k$ denote the cup product on $H^{*}_{dR}(X/k)$, and $\cup_K$ denote the cup product on $H^{*}_{dR}(X_K/K)$. Let $\alpha, \alpha' \in H^*_{dR}(X/k)$ and $\lambda, \lambda'$ in $K$, then under this isomorphism
$$
(\alpha \otimes \lambda) \cup_K (\alpha' \otimes \lambda') = (\alpha \cup_k \alpha') \otimes \lambda \lambda',
$$
and $Tr_{X_K}(\alpha \otimes \lambda) = \lambda Tr_X(\alpha)$. 
\end{lemma}
\begin{proof}
By the K{\"u}nneth formula, there is an isomorphism of $K$ vector spaces
$$
H^*_{dR}(X/k) \otimes_k H^*_{dR}(\Spec(K)/k) \cong H^*_{dR}(X_K/K),
$$
which respects the cup product and the trace map. Moreover, we see that we can identify $H^*_{dR}(\Spec(K)/k) \cong K$. The left hand side of the above isomorphism has a canonical semilinear action of $\mathrm{Gal}(K/k)$ on it, given by $\tau(x \otimes \alpha) = x \otimes \tau(\alpha)$. The right hand side has a left $\mathrm{Gal}(K/k)$ action on it, induced by the right action on $X_K$. We claim that these two actions are the same, which is enough to prove the theorem.

First, consider $\mathrm{Spec}(K)$ as a $k$ variety. The $\mathrm{Gal}(K/k)$ action on $\mathrm{Spec}(K)$ gives us a left action of $\mathrm{Gal}(K/k)$ on $H^*_{dR}(K/k)\cong K$, which corresponds to the standard action of $\mathrm{Gal}(K/k)$ on $K$. For $\tau \in \Gal(K/k)$, write $\hat{\tau}$ for the automorphism of $X_K$ given by the following commutative diagram
\begin{center}
\begin{tikzcd}[cramped]
X_K \ar[rdd, "p"', bend right =10] \ar[rrd, "\tau \circ q", bend left = 10] \ar[rd, dashed, "\hat{\tau}"] &&\\
& X_K \ar[d, "p"] \ar[r, "q"] & \Spec(K) \ar[d] \\
& X \ar[r] & \Spec(k),
\end{tikzcd}
\end{center}
where $\hat{\tau}$ is the morphism induced by the universal property of the fibre product. Write $p^*$, $q^*$, $\tau^*$ and $\hat{\tau}^*$ for the images of the morphisms $p$, $q$, $\tau$ and $\hat{\tau}$ after applying the contravariant functor, $H^*_{dR}(-/k)$. For $x \in H^*_{dR}(X)$, and $y \in K$, we have that
\begin{align*}
p^*(x) &= x \otimes 1 \\
q^*(y) &= 1 \otimes y \\
\tau^*(y) &= \tau(y).
\end{align*}
Therefore the map $\hat{\tau}^*: H^*_{dR}(X_K/k) \rightarrow H^*_{dR}(X_K/k) \otimes_k K$ is $x \otimes y \mapsto x \otimes \tau(y)$. 

Note that $H^*_{dR}(X_K/k)$ is the image of $H^*_{dR}(X_K/K)$ under the forgetful functor from $K$ vector spaces to $k$ vector spaces, and the map $\hat{\tau}^*$ is the map on $H^*_{dR}(X_K/K)$ induced by the semilinear $\Gal(K/k)$ action. This gives us an isomorphism $H^*_{dR}(X_K/k) \cong H^*_{dR}(X) \otimes_k K$  as required.
\end{proof}
\begin{defn}
For $X/K$ a smooth projective variety over $K$ of dimension $n$ with descent data, there is a natural Galois compatible quadratic form on $H^{2*+1}_{dR}(X/K)$ defined as follows. Note that $H^{2i+1}_{dR}(X/K)$ and $H^{2n-(2i+1)}_{dR}(X/K)$ are canonically dual and both have a semilinear $\Gal(K/k)$ action on them. If $2i+1 \neq n$, then $2n-(2i+1) \neq 2i+1$. Therefore the isomorphism from Poincaré duality $H^{2n-(2i+1)}_{dR}(X/K) \cong H^{2i+1}_{dR}(X/K)^\vee$ induces a canonical Galois compatible quadratic form on the direct sum $H^{2i+1}_{dR}(X/K)\oplus H^{2n-(2i+1)}_{dR}(X)$ given by taking a pair $(\alpha, f)$ to $f(\alpha)$. If $2i+1 = n$, then $H^n_{dR}(X/K)$ is canonically self dual. Writing $\phi: H^n_{dR}(X/K) \to H^n_{dR}(X/K)^\vee$, we obtain a Galois compatible quadratic form on $H^n_{dR}(X/K)$ by taking $\alpha$ to $\phi(\alpha)(\alpha)$, which is hyperbolic since $n$ is odd. This furnishes a Galois compatible quadratic form on all of $H^{2*+1}_{dR}(X/K)$, and we will write $[H^{2*+1}_{dR}(X/K)]$ to be the corresponding element of $\widehat{\mathrm{W}}(K)^{desc}$. Note that if we forget the descent data, the quadratic form is clearly hyperbolic. 
\end{defn}

\begin{cor}\label{DescentIsomorphism}
Let $(X, \varphi)$ be a $K$ variety with descent data $\varphi$. Then there is a natural Galois compatible non-degenerate quadratic form on $H^{2*}_{dR}(X/K)$ with Galois action induced by $\varphi$ given by the composition $x \mapsto \mathrm{Tr}_{X}(x \cup x)$.
\end{cor}
\begin{proof}
Immediate, since $(X,\varphi)$ is given by the base change of $\mathrm{desc}_{K/k}(X,\varphi)$ to $K$. 
\end{proof}
\begin{thm}\label{GaloisDescent}
There is a unique ring homomorphism 
$$
\chi^{mot,desc}_k: K_0(\mathrm{Var}_{K,desc}) \to \widehat{\mathrm{W}}(K)^{desc}$$ such that if $(X, \varphi)$ is a smooth projective variety, then
$$
\chi^{mot,desc}_k( [(X,\varphi)]) = [H^{2*}_{dR}(X/K)] - [H^{2*+1}_{dR}(X/K)].
$$
\end{thm}
\begin{proof}
Define $\chi^{mot,desc}$ to be the composition 
$$
\chi^{mot,desc}_k: K_0(\mathrm{Var}_{K,desc}) \cong K_0(\mathrm{Var}_k) \xto{\chi^{mot}} \widehat{\mathrm{W}}(k) \to \widehat{\mathrm{W}}(K)^{desc}.
$$ The above corollary guarantees that this  satisfies the desired property. Uniqueness is also immediate by the above corollary, since if $\varphi$ is any other morphism as above then the composition
$$
K_0(\mathrm{Var}_k) \to K_0(\mathrm{Var}_{K,desc}) \xto{\varphi} \widehat{\mathrm{W}}(K)^{desc} \cong \widehat{\mathrm{W}}(k)
$$
is such that for a smooth projective variety $X$ 
$$
[X] \mapsto [H^{2*}_{dR}(X/k)] - \frac12 \mathrm{dim}_k H^{2*+1}_{dR}(X/k) \cdot \mathbb{H} = \chi^{dR}(X),
$$
so the composition is equal to $\chi^{mot}$. As noted in Theorem $\ref{euler}$, the morphism $\chi^{mot}$ is unique with respect to this property, so $\chi^{mot,desc}$ will also be unique. 
\end{proof}
\begin{rem}
A similar argument to the one above was noted in section 8E of \cite{LR}.
\end{rem}

\section{The motivic Euler characteristic and transfer maps}
Let $L/k$ be a finite separable field extension. We can define various transfer maps between the rings $K_0(\mathrm{Var}_L)$ and $K_0(\mathrm{Var}_k)$, given by base change, forgetting the base, and Weil restriction, and analogously for $\widehat{\mathrm{W}}(L)$ and $\widehat{\mathrm{W}}(k)$ by base change, the trace transfer and Rost's norm transfer from \cite{Ro}. In this section, we show that the motivic Euler characteristic is compatible with these transfer maps.

\begin{defn}
Let $\iota: k \hookrightarrow L$ be a field extension. Then there is a natural ring homomorphism $\iota_*: \widehat{\mathrm{W}}(k) \to \widehat{\mathrm{W}}(L)$ given on the level of quadratic forms by sending a quadratic form $q: V \to k$ to $q \otimes \mathrm{Id}: V \otimes_k L \to L$. Similarly, there is a ring homomorphism $(\Spec(\iota))^*: K_0(\mathrm{Var}_k) \to K_0(\mathrm{Var}_L)$, given by sending a variety $X$ to $X_L$.
\end{defn}
\begin{lemma}\label{basechange}
The following diagram is commutative
\begin{center}
\begin{tikzcd}[cramped]
K_0(\mathrm{Var}_k) \ar[d, "\Spec(\iota)^*"] \ar[r, "\chi^{mot}_k"] & \widehat{\mathrm{W}}(k) \ar[d, "\iota_*"] \\
K_0(\mathrm{Var}_L) \ar[r, "\chi^{mot}_L"] & \widehat{\mathrm{W}}(L).
\end{tikzcd}
\end{center}
\end{lemma}
\begin{proof}
Theorem 3.1 of \cite{Bi} means it is enough to prove this result for smooth projective varieties, so we may use de Rham cohomology. Flat base change for de Rham cohomology gives us an isomorphism $H^*_{dR}(X_L/L) \cong H^*_{dR}(X/k) \otimes_k L$. Under this isomorphism, the cup product on $H^*_{dR}(X_L)$ is the extension of the cup product on $H^*_{dR}(X/k)$ via $L$-linearity, so it only remains to examine what happens to the trace map.

Since $X/k$ is smooth and projective, the cycle class map gives a canonical isomorphism $H^{2n}_{dR}(X/k) \cong H^n(X, \Omega^n_{X/k})$. The trace map $H^n(X, \Omega^n_{X/k}) \to k$ is defined to be the map given by the image of the identity under the isomorphism 
$$
Rp_{X/k,*} R\iHom_{\mathcal{O}_X}(\Omega_{X/k}, \Omega_{X/k})^0 \cong H^n(X, \Omega^n_{X/k})^\vee
$$
induced by Grothendieck--Verdier duality (see Corollary 11.2(f) of Chapter III of \cite{Ha}). In particular, we see that the trace map commutes with arbitrary base extension by Corollary 11.2(c) of loc. cit, so the trace map $\mathrm{Tr}_{X/L}$ is given by $\mathrm{Tr}_{X/k} \otimes_k \mathrm{Id}_L$, which gives us the result. 
\end{proof}
We now focus on two other natural transfer maps, which are maps $\widehat{\mathrm{W}}(L) \to \widehat{\mathrm{W}}(k)$ for quadratic forms and maps $K_0(\mathrm{Var}_L) \to K_0(\mathrm{Var}_k)$ for varieties.

\begin{defn}\label{tracedefn}
Let $L/k$ be a finite separable field extension. Write $N_{L/k}$ to be the norm map $L^\times \to k^\times$, and $\mathrm{Tr}_{L/k}$ to mean the trace map $L \to k$.  We have a natural homomorphism of abelian groups $\mathrm{Tr}_{L/k}: \widehat{\mathrm{W}}(L) \to \widehat{\mathrm{W}}(k)$, given on the level of quadratic forms by sending the quadratic form $q: V \to L$ to the quadratic form $\mathrm{Tr}_{L/k} \circ q$. This is a widely studied homomorphism in quadratically enriched enumerative geometry, appearing for example in \cite{KW} on page 19.

Write $\mathrm{disc}(L/k) := \mathrm{disc}(\mathrm{Tr}_{L/k}(\langle 1 \rangle))$. This is not misleading: for $k=\Q$, this agrees with the discriminant of the number field $L$. For $w \in \widehat{\mathrm{W}}(k)$, we have $\mathrm{rank} (\mathrm{Tr}_{L/k}(w)) = [L:k] \mathrm{rank}(w)$, so $\mathrm{Tr}_{L/k}$ is not a ring homomorphism.

Analogously to the above, if we have a variety $X \xto{p} \Spec(L)$, we can consider $X$ as a variety over $k$ with structure map $X \xto{p} \Spec(L) \to \Spec(k)$. This induces a homomorphism of abelian groups $K_0(\mathrm{Var}_L)\to K_0(\mathrm{Var}_k)$, though not a ring homomorphism.
\end{defn}

As in Remark $\ref{augdet}$, identify $\widehat{\mathrm{W}}(k)/I^2 = \mathbf{W}(k)$ with $E(\Hom(\Gal_k, \Z/2\Z))$. Let $E_{disc}$ denote the quotient map
 \begin{align*}
E_{disc}:\widehat{\mathrm{W}}(k) &\to E(\Hom(\Gal_k, \Z/2\Z))\\
w &\mapsto (\mathrm{rank}(w), \mathrm{disc}(w)).
\end{align*}

\begin{lemma}\label{tracedisc}
Let $w \in \widehat{\mathrm{W}}(L)$, such that $E_{disc}(w) = (\alpha, \beta)$. Then $\mathrm{Tr}_{L/k}(w)$ satisfies
$$
E_{disc}(\mathrm{Tr}_{L/k}(w)) = ([L:k]\alpha, \mathrm{disc}(L/k)^{\alpha}N_{L/k}(\beta)).
$$
\end{lemma}
\begin{proof}
The statement about the rank is immediate, so we only need to show this for the discriminant. Since both $\mathrm{disc}$ and $\mathrm{Tr}_{L/k}$ are abelian group homomorphisms, it is enough to show the statement on generators of $\widehat{\mathrm{W}}(L)$, so assume $w=\langle \beta \rangle$. Let $m_{\beta}$ be the matrix defining the $k$ linear map on $L$ given by multiplication by $\beta$, so that $\mathrm{det}(m_\beta) = N_{L/k}(\beta)$. Let $A$ be a symmetric matrix representing the symmetric $k$-bilinear form on $L$ given by $(x,y) \mapsto \mathrm{Tr}_{L/k}(xy)$ for all $x, y \in L$. Then $A$ is a symmetric matrix such that $\mathrm{det}(A) = \mathrm{disc}(L/k)$. We obtain a matrix representing the quadratic form $\mathrm{Tr}_{L/k}(\langle \beta \rangle)$, given by $Am_\beta$. The result follows since
$$
\mathrm{disc}(\mathrm{Tr}_{L/k}(\langle \beta \rangle)) = \mathrm{det}(Am_\beta) =  \mathrm{det}(A)\mathrm{det}(m_\beta) = \mathrm{disc}(L/k)N_{L/k}(\beta).
$$
\end{proof}

\begin{lemma}\label{trace}
The following diagram is commutative.
\begin{center}
\begin{tikzcd}[cramped]
K_0(\mathrm{Var}_L)\ar[r] \ar[d, "\chi^{mot}_L"] & K_0(\mathrm{Var}_k) \ar[d, "\chi^{mot}_k"] \\
\widehat{\mathrm{W}}(L) \ar[r, "\mathrm{Tr}_{L/k}"] & \widehat{\mathrm{W}}(k).
\end{tikzcd}
\end{center}
\end{lemma}
\begin{proof}
We first note we can reduce to the case where $X/L$ is smooth, projective and connected, so that we may work with the de Rham Euler characteristic. Let $X/L$ be a smooth projective variety of dimension $n$. There is an isomorphism of $k$-vector spaces $H^*_{dR}(X/L) \cong H^*_{dR}(X/k)$, and this isomorphism is compatible with cup products, so it only remains to check what happens to the trace map $H^{2n}_{dR}(X/k) \to k$. From here, the proof is largely an application of Corollary 11.2 of Chapter 3 of \cite{Ha}, which we present here for the convenience of the reader.

After applying the cycle class isomorphism $H^{2n}_{dR}(X/k) \cong H^n(X, \Omega^n_{X/k})$, we see that the trace map $H^n(X, \Omega^n_{X/k}) \to k$ is the image of the identity under the isomorphism 
$$
Rp_{X/k,*}\iHom_{\mathcal{O}_X}(\Omega_{X/k}, \Omega_{X/k})^0 \cong H^n(X, \Omega^n_{X/k})^\vee
$$
induced by Grothendieck--Verdier duality, and similarly if we replace $k$ by $L$ as above. Corollary 11.2(e) of Chapter 3 of \cite{Ha} implies that the image of the identity under the isomorphism
$$
Rp_{L/k,*}\iHom_{L}(\Omega_{L/k}, \Omega_{L/k})^0 \cong H^n(\Spec(L), \Omega^n_{L/k})^\vee
$$
is given by the standard trace map $\mathrm{Tr}_{L/k}$. Since $p_{L/k}$ is étale $p^!_{L/k}=p^*_{L/k}$ is the left adjoint of $Rp_{L/k,*}$, and since it is finite $Rp_{L/k,*}=p_{L/k,*}$. Moreover, since $p^!$ is right adjoint to $p_*$, we have canonical isomorphisms $p^!_{X/k}(\mathcal{O}_k) = \Omega^n_{X/k}[-n]$, and similarly when replacing $k$ with $L$. Similarly we see
$$
p^!_{X/k}(k) = p^!_{X/L}p^!_{L/k}(k) = p^!_{X/L}( p^*_{L/k}(k)) = p^!_{X/L}(L).
$$
Using this, we have the following commutative diagram where all of the arrows are isomorphisms
\begin{center}
\begin{tikzcd}[cramped]
Rp_{X/k,*} R\iHom_{\mathcal{O}_X}(p^!_{X/k}(k), p^!_{X/k}(k))^0 \ar[r] \ar[d] & p_{L/k,*} Rp_{X/L,*} {R}\iHom_{\mathcal{O}_X}(p^!_{X/L}(L), p^!_{X/L} p^!_{L/k}(k) )^0 \ar[d] \ar[d] \\
{R}\iHom_k( Rp_{X/k, *}p^!_{X/k}(k), k)^0 \ar[dd]   & p_{L/k, *} {R}\iHom_{L}( Rp_{X/L, *} p^!_{X/L}(L), p^!_{L/k}(k))^0 \ar[d] \\
& p_{L/k, *} \Hom_L( H^n(X, \Omega^n_{X/L}), p^!_{L/k}k) \ar[d, "\Phi"] \\
\Hom_k( H^n(X, \Omega^n_{X/k}), k) \ar[r, "="]& \Hom_k(p_{L/k, *} H^n(X, \Omega^n_{X/L}), k).
\end{tikzcd}
\end{center}
Again by Corollary 11.2(e) and (f) of Chapter 3 of \cite{Ha}, the isomorphism $\Phi$ takes a morphism $f \in \Hom_L( H^n(X, \Omega^n_{X/L}), L)$ to $\mathrm{Tr}_{L/k} \circ f \in \Hom_k(H^n(X, \Omega^n_{X/L}), k)$.  Consider the identity element $\mathrm{Id} \in  Rp_{X/k,*} {R}\iHom_{\mathcal{O}_X}(p^!_{X/k}(k), p^!_{X/k}(k))^0$. Going clockwise around the above diagram maps the element $\mathrm{Id}$ to $\mathrm{Tr}_{L/k} \circ \mathrm{Tr}_{X/L}$. However, going counter clockwise maps $\mathrm{Id}$ to $\mathrm{Tr}_{X/k}$, giving the desired equality. 
\end{proof}

We now turn our attention to the norm and Weil restriction transfer maps. This first requires us to use the machinery of polynomial maps: while the transfer maps defined by the trace and forgetful functor are group homomorphisms, the transfer maps we obtain by looking at Weil restrictions and norms have a different structure.
\begin{defn} Let $B$ be an abelian group, let $A$ an abelian monoid and let $P:A\to B$ be a function on the underlying sets. We say $P$ is a \emph{polynomial map of degree $\leq -1$} if $P=0$.  For $n\geq0$, we say $P$ is a \emph{polynomial map of degree $\leq n$} if for each $x\in A$ the map
\begin{align*}\Delta_xP:A&\to B\\
\Delta_xP(y)&=P(y+x)-P(y)\end{align*}
is a polynomial map of degree $\leq n-1$. We say $P$ is a \emph{polynomial map} if it is a polynomial map of degree $\leq n$ for some $n$. Note that $P$ is a polynomial map of degree $\leq 0$ if and only if $P$ is constant.

Suppose now that $A$ and $B$ are commutative rings. A polynomial map is called \emph{multiplicative} if $P(1)=1$ and if for every $x,y \in A$, we have $P(xy)=P(x)P(y)$. Note that a multiplicative polynomial map of degree $\leq 1$ is either identically $1$ or a ring homomorphism. 
\end{defn}
Polynomial maps can be extended from monoids to rings by the following.
\begin{lemma}[(Proposition 1.6 of \cite{SJ})]\label{extend}
Let $\tilde{A}$ be an abelian monoid, let $\tilde{f}: \tilde{A} \to B$ be a polynomial map of degree $\leq n$, and let $A$ be the group completion of $\tilde{A}$. Then there exists a unique polynomial map of degree $\leq n$ $f: A \to B$ which extends $\tilde{f}$. 
\end{lemma}
The statement about the degree of the polynomial map is not in the statement of Proposition 1.6 of \cite{SJ}, however it follows from the proof.

For every finite separable extension $K/F$ of fields let
$\mathrm{N}_{K/F}:K\to F$ denote the norm map. The following proposition also serves as a definition of the norm map.
\begin{propn}[(Corollary 5 of \cite{Ro})]\label{normdefn}
Let $F$ be a field of characteristic $\neq 2$, and let $K/F$ be a finite separable field extension of degree $n$. Then there exists a unique multiplicative polynomial map $\mathrm{N}_{K/F}: \widehat{\mathrm{W}}(K)\to\widehat{\mathrm{W}}(F)$ of degree $\leq n$ such that $\mathrm{N}_{K/F}(\langle b\rangle)=\langle\mathrm{N}_{K/F}(b)\rangle$ for all $b \in k^\times$.
\end{propn}

This rest of this section is dedicated to proving Theorem $\ref{WRs}$, which says that this norm polynomial map computes the motivic Euler characteristic of the Weil restriction of a variety. In order to prove this, we first need some smaller lemmas.
\begin{lemma}
Suppose that $f,g: A \to B$ are polynomial maps of degree $\leq n$, $\leq m$ respectively. Then $f+g$ is a polynomial maps of degree $\leq \mathrm{max}(n,m)$.  Similarly, if $m=1$ and $f,g$ are multiplicative, then $fg$ is a multiplicative polynomial maps of degree $\leq n+1$.
\end{lemma}
\begin{proof}
We prove the first statement by induction on the degree of $f$. If $\mathrm{deg}(f)=0$, then $f$ is constant, so $\Delta_x(f+g) = \Delta_x(g)$. Suppose $\mathrm{deg}(f)=n$. Then consider $\Delta_x (f+g): A \to B$. We compute $\Delta_x(f+g)(y) = \Delta_x(f)(y) + \Delta_x(g)(y)$. Since $\Delta_x(f)$ and $\Delta_x(g)$ are polynomial maps of degree $\leq n-1$ and $\leq m-1$ respectively, the result holds by induction.

For the second part of the lemma, we proceed by induction on the degree of the polynomial map, with the case that $f$ has degree $\leq 0$ being trivial, since $f$ is then identically $1$. We compute
\begin{align*}
\Delta_x(fg)(y) &= g(x+y)f(x+y) - g(y)f(y)\\
&= g(x)f(x+y) + g(y)f(x+y) - g(y)f(y) \\
&= g(x)f(x+y) + g(y) (\Delta_xf(y)).
\end{align*}
By induction, we have that $g \cdot (\Delta_xf)$ is a polynomial map of degree $\leq n$. Note that we have a polynomial map $y \mapsto g(x) f(x+y)$ of degree $\leq n$, since $g(x)$ is constant, and $f(x+y) = \Delta_xf + f$ is a polynomial map of degree $\leq n$ by the first part of the lemma. Therefore $\Delta_x(fg)$ is polynomial of degree $\leq n$, so $fg$ is polynomial of degree $\leq n+1$. The fact that $fg$ is multiplicative is immediate.
\end{proof}
\begin{cor}\label{ringpol}
Let $g_1, \ldots, g_n: A \to B$ be ring homomorphisms. Then the product $g = \prod_{i=1}^n g_i$ is a polynomial function of degree $\leq n$.
\end{cor}
\begin{proof}
Proceed by induction on $n$, with the case $n=1$ being trivial, and the induction step following by the second part of the above lemma.
\end{proof}
The next few results are dedicated to showing that the Weil restriction functor from $L$ to $k$ gives rise to a polynomial map $K_0(\mathrm{Var}_L)\to K_0(\mathrm{Var}_k)$.

\begin{defn}
Let $L/k$ be a finite separable field extension. We let $I_L := \{\iota: L \to \kbar \}$ be the set of embeddings of $L$ into $\kbar$ with $\iota|_k = \mathrm{id}_k$. This set has a natural right action of the absolute Galois group $\Gal_k$, given by $\iota \cdot \sigma = \sigma^{-1} \circ \iota$. 

If $K/k$ is a finite separable field extension with $L \subseteq K$, we have a $\Gal_k$ equivariant map $I_K \to I_L$ given by restriction, and the fibres of this map have cardinality $[K:L]$.  From here, fix $L/k$ a finite field extension, let $K$ be the Galois closure of $L$ in $\kbar$, and let $G=\Gal(K/k)$. If $H$ is the subgroup of $G$ such that $L=K^H$, then $I_L = G/H$ as a $G$ set.
\end{defn}
\begin{defn}\label{descentdatawr}
Let $X$ be a quasi-projective variety over $L$. Define a $K$ variety $\mathrm{Ind}^K_LX$ by setting $\mathrm{Ind}^K_LX := \prod_{ \iota \in I_L} X_{\iota}$, where $X_{\iota}$ is the pullback of $X \to \Spec(L)$ along the inclusion $\Spec(K) \xto{\iota} \Spec(L)$. We give $\mathrm{Ind}^K_LX$ a canonical semilinear $G$ action. Let $g \in G$. Define a right $G$ action on $\mathrm{Ind}^K_LX$ by $g: X_{g \circ \iota} \to X_{\iota}$, where this morphism is the pullback
\begin{center}
\begin{tikzcd}[cramped]
X_{g \circ \iota} \ar[r] \ar[d] & X_{\iota} \ar[d] \\
\Spec(K) \ar[r, "g"] & \Spec(K).
\end{tikzcd}
\end{center}
This defines a semilinear action of $G$ on $\mathrm{Ind}^K_LX$, which we call $\varphi_G$. This in turn gives us a variety $\mathrm{des}_{K/k}(\mathrm{Ind}^K_LX, \varphi_G)$. by Galois descent. By Theorem 1.3.2 of \cite{Weil}, the variety $\mathrm{des}_{K/k}(\mathrm{Ind}^K_LX, \varphi_G)$ is isomorphic to $\mathrm{Res}_{L/k}(X)$.
\end{defn}

\begin{lemma}\label{polymap}
The assignment $X \to \mathrm{Ind}^K_LX$ extends to a polynomial map of degree $\leq n$, denoted $\mathrm{Ind}^K_L: K_0(\mathrm{Var}_L) \to K_0(\mathrm{Var}_K)$.
\end{lemma}
\begin{proof}
For each $\iota \in I_L$, note that $g_{\iota}: [X] \mapsto [X \times_{L, \iota}K]$ is a ring homomorphism. Applying Corollary $\ref{ringpol}$ gives us the result. 
\end{proof}
In order to show that maps into $K_0(\mathrm{Var}_k)$ are polynomial maps, it is helpful to note the following, which is likely well known.
\begin{lemma}\label{zerovariety}
Let $Y/k$ be a variety. Then the class $[Y]$ in $K_0(\mathrm{Var}_k)$ is trivial if and only if $Y=\emptyset$. 
\end{lemma}
\begin{proof}
By naturality of $K_0(\mathrm{Var}_k)$, it is enough to prove this in the case where $k$ is algebraically closed. Suppose that $Y$ is a non-empty variety, and let $\mathrm{dim}(Y) = d$. As in Proposition $\ref{bidefn}$, we have abelian group homomorphisms $b_i: K_0(\mathrm{Var}_k) \to \mathbb{Z}$, which, if $X/k$ is a smooth projective variety, has that $b_i(X)$ is the $i^{\text{th}}$ Betti number of $X$. The proof of Theorem 3.1 of \cite{Bi} shows that we may write $[Y] = [\overline{X}] + S$, where $S \in K_0(\mathrm{Var}_k)$ can be written in terms of classes of smooth projective varieties of dimension $<d$. In particular, $b_{2d}(Y) = b_{2d}(\overline{X}) + b_{2d}(S)$. Since $S$ can be written in terms of smooth projective varieties of dimension $<d$, we see $b_{2d}(S)=0$. Moreover, since $\overline{X}$ is a smooth projective variety of dimension $d$, we see $b_{2d}(\overline{X}) = \# \pi_0(\overline{X}) \neq 0$. Therefore, $b_{2d}(Y) \neq 0$, so $[Y] \neq 0$. 
\end{proof}

\begin{lemma}\label{multpoly}
There is a multiplicative polynomial map $K_0(\mathrm{Var}_L)\to K_0(\mathrm{Var}_k)$, given on smooth quasi-projective varieties by $[X] \mapsto [\mathrm{Res}_{L/k}(X)]$.
\end{lemma}
\begin{proof}
By Lemma $\ref{extend}$, we only need to define the polynomial map on the subset of $K_0(\mathrm{Var}_L)$ consisting of elements which are given by $[X]$, where $X$ is a smooth projective variety.  As in Definition $\ref{descentdatawr}$, we note that $\mathrm{des}_{K/k}(\mathrm{Ind}^K_LX, \varphi_G)$. Therefore as noted in Definition $\ref{DescentDataVar}$, it is enough to define a map $K_0(\mathrm{Var}_L) \to K_0(\mathrm{Var}_{K,desc})$ such that for $X/L$ a quasiprojective variety, we have that $[X] \mapsto [(\mathrm{Ind}^K_LX, \varphi_G)]$.

By Lemma $\ref{extend}$, it is enough to show that the assignment 
$$
\mathrm{Res}_{L/k}: [X] \mapsto [(\mathrm{Ind}^K_LX, \varphi_G)]
$$
satisfies the axioms to be a polynomial map when $X$ is projective. That is, we claim that for projective $Y_0, \ldots, Y_n$, we have $\Delta_{Y_0} \Delta_{Y_1} \ldots \Delta_{Y_n} \mathrm{Res}_{L/k}([X]) = 0$.  Note that
 $$
 \Delta_Y\mathrm{Res}_{L/k}([X]) = [(\mathrm{Ind}^K_L(X \amalg Y), \varphi_G)] - [(\mathrm{Ind}^K_LX, \varphi_G)].
 $$ We see that $\mathrm{Ind}^K_LX$ embeds as a closed subvariety into $\mathrm{Ind}^K_L(X \amalg Y)$ in a way that is compatible with the descent data on both varieties. Therefore we can write $\Delta_Y \mathrm{Res}_{L/k}([X]) = [(X_Y, \varphi_G)]$, where $(X_Y, \varphi_G)$ is the $K$ variety with descent data given by the complement of $(\mathrm{Ind}^K_L(X), \varphi_G)$ in $(\mathrm{Ind}^K_L(X \amalg Y), \varphi_G)$.

 Therefore $\Delta_Y \mathrm{Res}_{L/k}([X])$ is represented by the class of an actual $K$ variety with descent data, $(X_Y, \varphi_G)$ so gives us a $k$ variety $\mathrm{des}_{K/k}(X_Y, \varphi_G)$.  Iterating this process, if $Y_0, \ldots, Y_n$ are $L$ varieties, we can write 
$$
\Delta_{Y_0} \Delta_{Y_1} \ldots \Delta_{Y_n} \mathrm{Res}_{L/k} ([X]) = [(X_{Y_0, \ldots, Y_n}, \varphi_G)],
$$
where $(X_{Y_0, \ldots, Y_n}, \varphi_G)$ is a $K$ variety with descent data. We claim that the class $[(X_{Y_0, \ldots, Y_n}, \varphi_G)] = 0$. Note that  under the isomorphism $K_0(\mathrm{Var}_{K, desc}) \cong K_0(\mathrm{Var}_k)$, the class $[(X_{Y_0, \ldots, Y_n}, \varphi_G)]$ is identified with $[\mathrm{des}_{K/k}((X_{Y_0, \ldots, Y_n}), \varphi_G)]$, so it is enough to show that $[\mathrm{des}_{K/k}((X_{Y_0, \ldots, Y_n}), \varphi_G)] = 0$. Since this is the class of a $k$ variety, $\mathrm{des}_{K/k}((X_{Y_0, \ldots, Y_n}), \varphi_G)$, we may apply Lemma $\ref{zerovariety}$ to see that this class vanishes if and only if $\mathrm{des}_{K/k}((X_{Y_0, \ldots, Y_n}), \varphi_G)$ is the empty variety.

We see that $\mathrm{des}_{K/k}((X_{Y_0, \ldots, Y_n}), \varphi_G) = \emptyset$ if and only if $\mathrm{des}_{K/k}((X_{Y_0, \ldots, Y_n}), \varphi_G)_K$ is a trivial $K$-variety. However, $[\mathrm{des}_{K/k}((X_{Y_0, \ldots, Y_n}), \varphi_G)_K]  = [\Delta_{Y_0} \Delta_{Y_1} \ldots \Delta_{Y_n} \mathrm{Ind}^K_L(X)]$, which vanishes since $\mathrm{Ind}^K_L$ is a polynomial map of degree $\leq n$. In particular, $[\mathrm{des}_{K/k}((X_{Y_0, \ldots, Y_n}), \varphi_G)] = 0$, and so $\Delta_{Y_0} \Delta_{Y_1} \ldots \Delta_{Y_n} \mathrm{Res}_{L/k} ([X])=0$ as required.
\end{proof}

\begin{defn}\label{normdef2}
Let $q$ be a quadratic form over $L$, on a vector space, $V$. For $\iota \in I_L$, define $q_{\iota}$ to be the quadratic form on $V \otimes_{L, \iota} K$ given by $q \otimes_{L, \iota} \langle 1 \rangle$, where we embed $L$ into $K$ by $\iota$. It is clear that $f_{\iota}: \widehat{\mathrm{W}}(L) \to \widehat{\mathrm{W}}(K)$ given by $[q] \mapsto [q_{\iota}]$ is a homomorphism. Define $\mathrm{Ind}^K_L(q) := \otimes_{\iota} q_\iota$ as a quadratic form on $K$. Note that $\mathrm{Ind}^K_L(q)$ has  underlying vector space $\bigotimes_{\iota \in I_L} (V \otimes_{\iota} K)$ and we may give this space a semilinear $G$ action, where $g \in G$ acts by 
$$
g(v \otimes \lambda)_{\iota} = (v \otimes g(\lambda))_{g \iota},
$$
 where $(v \otimes \lambda)_{\iota}$ means $v \otimes \lambda \in V \otimes_{\iota} K$.  We see that $\mathrm{Ind}^K_L(q)$ is a Galois compatible quadratic form on $\bigotimes_{\iota \in I_L} (V \otimes_{\iota} K)$ with respect to this Galois action. Write $\mathrm{Ind}^K_L(q)$ for the Galois compatible quadratic form on $\bigotimes_{\iota \in I_L} (V \otimes_{\iota} K)$, and let $N_{L/k}(q)$ be the associated quadratic form over $k$ which results from applying the equivalence of categories in Theorem $\ref{GaloisDescentQuadForm}$ to $\mathrm{Ind}^K_L(q)$ . 
 
 By the above work, it's clear that $\mathrm{rank}(N_{L/k}(q)) = \mathrm{rank}(q)^{[L:k]}$.
\end{defn}
\begin{rem}
A priori, it is not clear that $N_{L/k}(q)$, as defined above, is equivalent to the norm as in Proposition $\ref{normdefn}$. The following lemma shows that this is not something to be concerned about, and the maps are the same.
\end{rem}

\begin{lemma}
Let $q=\langle \alpha \rangle$ for $\alpha \in L$. Then $[N_{L/k}(q)] = \langle N_{L/k}(\alpha) \rangle \in \widehat{\mathrm{W}}(k)$. 
\end{lemma}
\begin{proof}
Note that $g_{\iota}(\langle \alpha \rangle) = \langle \iota(\alpha) \rangle$. Similarly, we see $\bigotimes_{\iota \in I_L} k$ is the subspace of $\bigotimes_{\iota \in I_L} L \otimes_{\iota} K$ fixed by $G$. In particular 
$$
N_{L/k}( \langle \alpha \rangle ) (1 \otimes \ldots \otimes 1) = \prod_{\iota \in I_L} \iota(\alpha) = N_{L/k}(\alpha),
$$ which gives us that $N_{L/k}(\langle \alpha \rangle) = \langle N_{L/k} (\alpha) \rangle$ as required.
\end{proof}
\begin{cor}
The map $N_{L/k}(q)$ from Proposition $\ref{normdefn}$ and the map $N_{L/k}(q)$ from Definition $\ref{normdef2}$ agree.
\end{cor}
\begin{proof}
This is immediate by the above lemma and Proposition $\ref{normdefn}$.
\end{proof}
\begin{rem}
Since $N_{L/k}: \widehat{\mathrm{W}}(L) \to \widehat{\mathrm{W}}(k)$ is a polynomial map, we obtain a polynomial map $\mathrm{Ind}^K_L: \widehat{\mathrm{W}}(L) \to \widehat{\mathrm{W}}(K)^{desc}$ by composing with the isomorphism in Corollary $\ref{descentisomorphismgw}$. 
\end{rem}

\begin{cor}\label{normdisc}
Let $w \in \widehat{\mathrm{W}}(L)$ with $E_{disc}(w) = (\alpha, \beta) \in E( (L^\times)/(L^\times)^2)$. Then
$$
E_{disc}(N_{L/k}(w)) = (\alpha^{[L:k]}, N_{L/k}(\beta)).
$$ 
\end{cor}
\begin{proof}
We know that $E_{disc} \circ N_{L/k}$ is a polynomial map of degree $\leq n$, so we can reduce to the case where $w$ is a quadratic form, and the result will follow by Lemma $\ref{extend}$. This is then a straightforward computation, similar to the lemma above.
\end{proof}

\begin{thm}\label{WRs}
The following diagram is commutative.
\begin{center}
\begin{tikzcd}[cramped]
K_0(\mathrm{Var}_L)\ar[r, "\chi^{mot}_L"] \ar[d, "\mathrm{Res}_{L/k}"] & \widehat{\mathrm{W}}(L) \ar[d, "N_{L/k}"] \\
K_0(\mathrm{Var}_k) \ar[r, "\chi^{mot}_k"] & \widehat{\mathrm{W}}(k).
\end{tikzcd}
\end{center}
In particular, if $X/L$ is a variety such that $\mathrm{Res}_{L/k}(X)$ exists, then
$$
\chi^{mot}_k(\mathrm{Res}_{L/k}(X))= N_{L/k}(\chi^{mot}_L(X)).
$$
\end{thm}
\begin{proof}
The composition of two multiplicative polynomial functions is again a multiplicative polynomial function. Therefore if for $X$ a smooth projective variety over $L$, we have
$$
\chi^{mot}(\mathrm{Res}_{L/k}(X)) = N_{L/k}(\chi^{mot}(X)),
$$
then $\chi^{mot} \circ \mathrm{Res}_{L/k}$ and $N_{L/k} \circ \chi^{mot}$ define the same polynomial function from $K_0(\mathrm{Var}_L)$ to $\widehat{\mathrm{W}}(k)$ by Lemma $\ref{extend}$. Since $\mathrm{Res}_{L/k}(X) = \mathrm{des}_{K/k}( (\mathrm{Ind}^K_L(X), \varphi_G))$, and $N_{L/k}(\chi^{mot}_L(X))$ is the image of $\mathrm{Ind}^K_L( \chi^{mot}_L(X)^{desc})$ under the isomorphism $\widehat{\mathrm{W}}(k) \to \widehat{\mathrm{W}}(K)^{desc}$, the theorem follows if we can show the following diagram is commutative:
\begin{center}
\begin{tikzcd}
K_0(\mathrm{Var}_L)\ar[r, "\chi^{mot}_L"] \ar[d, "\mathrm{Ind}^K_L"] & \widehat{\mathrm{W}}(L) \ar[d, "Ind^K_L"] \\
K_0(\mathrm{Var}_{K,desc}) \ar[r, "\chi^{mot, desc}_{k}"] & \widehat{\mathrm{W}}(K)^{desc}.
\end{tikzcd}
\end{center}
Therefore we need to show
$$
\chi^{mot,desc}( \mathrm{Ind}^K_LX) = \prod_{\iota} g_{\iota}(\chi^{mot}(X)) \in \widehat{\mathrm{W}}(k),
$$
which is clear by naturality of the motivic Euler characteristic, the semilinear actions on both sides, and the fact that Galois descent is compatible with the motivic Euler characteristic as in Corollary $\ref{DescentIsomorphism}$.
\end{proof}

\section{Fubini theorems}

 Historically, since the motivic Euler characteristic is a motivic measure, we might think of it as a volume. Indeed, the motivic Gauss--Bonnet formula of \cite{LR} says that we may obtain the motivic Euler characteristic as the integral of an Euler class across our space. Let $X,Y$ be varieties over $k$, and suppose we have a morphism $f: X \to Y$. For the purposes of this paper, a Fubini theorem is a theorem that allows us to compute properties of $\chi^{mot}(X)$ from the motivic Euler characteristic $\chi^{mot}(Y)$ as well as the motivic Euler characteristics of the fibres $\chi^{mot}(X_y)$. For example, suppose that our base field is $\mathbb{C}$, so that $\chi^{mot}$ is the compactly supported Euler characteristic applied to the complex points. By the standard Fubini theorem $\chi_c(X_y(\mathbb{C}))=0$ for all $y$ implies that $\chi_c(X(\mathbb{C}))=0$. This section is concerned with proving certain Fubini theorems for the motivic Euler characteristic modulo $J$. We first look at the situation where $f: X \to Y$ satisfies that $\chi^{mot}(X_y) = 0$ for every closed point $y$ of $Y$. The principle that ``if an invariant is zero on the fibres, then it is zero on the total space" will be referred to as a na{\"i}ve Fubini theorem. It does not tend to hold in our context, and we obtain a Fubini theorem for the $\ell$-adic determinant of cohomology in Corollary $\ref{saitofactorsurface}$, which does not satisfy this na{\"i}ve property. However, in Corollary $\ref{restrictedramification}$, we show that we may bound the ramification of the character we obtain by looking at the $\ell$-adic determinant of cohomology of the total space, which allows us to restrict the possibilities of this character to a finite set. We then turn to looking at a different Fubini theorem, one for maps whose fibres are algebraic groups, where we can use the structure of algebraic groups in order to prove a na{\"i}ve Fubini theorem for torsors in Corollary $\ref{fubini1b}$.
\subsection{The Fubini theorem over real closed fields}
The goal of this subsection is to prove the first part of our Fubini theorem.
\begin{thm}
Suppose that $X$ and $Y$ are varieties over $k$, and let $f: X \to Y$ be a morphism such that for every $y$ a closed point of $Y$, we have $\chi^{mot}(X_y)$ is torsion. Then $\chi^{mot}(X)$ is torsion.
\end{thm}
We call this the real closed Fubini theorem as, by Pfister's local global principle, we have that $\chi^{mot}(X)$ is torsion if and only if $\mathrm{rank}(\chi^{mot}(X))=0$ and $\mathrm{sign}(\chi^{mot}(X))=0$. This is therefore is a true na{\"ive} Fubini theorem. This holds if we can show the corresponding statement for the rank and sign of motivic Euler characteristics.  We first prove this statement for real closed fields and their algebraic closures. Until specified otherwise, let $k$ be either a real closed field or the algebraic closure of a real closed field.
\begin{defn}
Let $\mathcal{SAS}_k$ be the category of semi-algebraic sets over $k$. Define $K_0(\mathcal{SAS}_k)$ to be the ring generated by isomorphism classes of semi-algebraic sets over $k$ subject to the following relations:
\begin{enumerate}
\item $[\emptyset]=0$.
\item For $U$ a closed semialgebraic subset of $X$, we have $[X] = [X \setminus U] + [U]$.
\item $[X][Y] = [X \times Y]$.
\end{enumerate}
\end{defn}
There are canonical homomorphisms $K_0(\mathrm{Var}_k) \rightarrow K_0(\mathcal{SAS}_k)$, given by sending a variety $X$ to $X(k)$ or to $X(\overline{k})$.

Let $X$ be a semi-algebraic set. Write $H^i_c(X;k)$ to mean the $i^{\text{th}}$ real closed singular cohomology with compact support, and coefficients in $k$. Note that the semialgebraic simplicial cohomology is isomorphic to the sheaf cohomology with coefficients in the constant sheaf $\underline{k}$. This is proven in the non-compactly supported case in Theorem 3.3 of \cite{KD}. This isomorphism extends to the compactly supported case, using that the compactly supported cohomology is given by
$$
H^*_c(X; k) = \varinjlim H^*(X, X\setminus V; k),
$$
where $V$ runs over all of the compact subsets of $X$. Since we have canonical isomorphisms 
$$
H_{sheaf}^*(X,X\setminus V; \underline{k}) \cong H_{sing}^*(X, X\setminus V; k),$$
this isomorphism carries over to the limit.
\begin{defn}
 Let $X$ be a semi-algebraic set over $k$. Define the \emph{real closed Euler characteristic} $\chi^{rc}(X) := \sum_{i=0}^{\infty} (-1)^i \mathrm{dim}_k( H^i_c(X;k)) \in \mathbb{Z}$.
\end{defn}
\begin{lemma}
The real closed Euler characteristic $\chi^{rc}$, as defined above, furnishes a ring homomorphism $\chi^{rc}: K_0(\mathcal{SAS}_k) \rightarrow \mathbb{Z}$.
\end{lemma}
\begin{proof}
We only need to show the three relations defining $K_0(\mathcal{SAS}_k)$ are satisfied. The relation that $\chi^{rc}([\emptyset])=0$ is trivial. Theorem 4.1 of \cite{EP1} tells us the sheaf cohomology with compact support and coefficients in $\underline{R}$ satisfies the exactness and excision axioms. Therefore for $U$ an open subset of $X$, we have a long exact sequence
$$
\ldots \rightarrow H^n_c(U;k) \rightarrow H^n_c(X;k) \rightarrow H^n_c(X \setminus U; k) \rightarrow H^{n+1}_c(U;k) \rightarrow \ldots.
$$
This immediately tells us that $\chi^{rc}(X) = \chi^{rc}(U) + \chi^{rc}(X \setminus U)$. Multiplicativity of $\chi^{rc}$ follows from the K\"{u}nneth formula: Theorem 3.40 of \cite{EP2}.
\end{proof}
\begin{thm}\label{chircfubini}
Let $f: W \to Z$ be a morphism of semialgebraic sets over $k$. Suppose for each $p \in Z$, the real closed Euler characteristic of the fibre $\chi^{rc}(W_p)=q$. Then $\chi^{rc}(W) = q \chi^{rc}(Z)$.
\end{thm}
\begin{proof}
By Lemma 9.3.2 of \cite{BCR}, there exists a stratification of $Z$ into constructible subsets $Z_1, \ldots, Z_n$, such that for each $i$, we have an isomorphism $f^{-1}(Z_i) \cong F_i \times W_i$ for some $F_i$ fitting into the following commutative diagram
\begin{center}
\begin{tikzcd}[cramped]
f^{-1}(Z_i) \ar[rd, "f"] \ar[rr] & & F_i \times Z_i \ar[ld] \\
& Z_i, &
\end{tikzcd}
\end{center}
where the morphism $F_i \times Z_i \rightarrow Z_i$ is the canonical projection. This means that for $p \in Z_i$, there is an isomorphism $W_p \cong F_i$, so $\chi^{rc}(F_i) = \chi^{rc}(W_p)=q$. In $K_0(\mathcal{SAS}_k)$ we have $[W] = \sum_{i=1}^n [F_i][Z_i]$, so applying $\chi^{rc}$ and noting that $\chi^{rc}(F_i)=q$, we get that $\chi^{rc}(W)=q\chi^{rc}(Z)$.
\end{proof}

\begin{lemma}\label{motrc}
For $X$ a $k$ variety, we have an equality in $\widehat{\mathrm{W}}(k)$:
$$
\chi^{mot}(X) = \frac12 \left(\chi^{rc}(X(\overline{k}))+\chi^{rc}(X(k))  \right)  \langle 1 \rangle + \frac12 \left( \chi^{rc}(X(\overline{k}))-\chi^{rc}(X(k)) \right) \langle -1 \rangle.
$$
\end{lemma}
\begin{proof}
By Theorem $\ref{levine1}$, we have that $\mathrm{rank}(\chi^{mot}(X)) = e(X) = \chi^{rc}(X(\overline{k}))$. Moreover, we have that $\mathrm{sign}(\chi^{mot}(X)) = \epsilon(X) = \chi^{rc}(X(k))$ by Theorem $\ref{levine2}$. The result follows since if $q$ is any element of $\widehat{\mathrm{W}}(k)$, we have
$$q = \frac12 (\mathrm{rank}(q) + \mathrm{sign}(q))\langle 1\rangle + \frac12 (\mathrm{rank}(q) - \mathrm{sign}(q))\langle-1\rangle.$$
\end{proof}

\begin{thm}\label{fubini2}
Let $X$ and $Y$ be $k$ varieties, and let $f: X \rightarrow Y$ be a morphism. For $p$ a closed point of $Y$ let $X_p$ be the fibre of $f$ over $p$. Then if $\chi^{mot}_k(X_p)$ is trivial for all $p$, then so is $\chi^{mot}(X)$.
\end{thm}
\begin{proof}
By Lemma $\ref{motrc}$
$$
\chi^{mot}(X) = \frac12 \left(\chi^{rc}(X(\overline{k}))+\chi^{rc}(X(k))  \right)  \langle 1 \rangle + \frac12 \left( \chi^{rc}(X(\overline{k}))-\chi^{rc}(X(k)) \right) \langle -1 \rangle,
$$
so it is enough to prove the result for $\chi^{rc}(X(\kbar))$ and $\chi^{rc}(X)$. First suppose that $k$ is a real closed field. If $p$ has residue field $k$, then $\chi^{rc}(X_p(k)) = \mathrm{sign}(\chi^{mot}(X_p))=0$ by assumption. If $p$ has residue field $\kbar$, then $X_p(k)=\emptyset$, so $\chi^{rc}(X_p(k))=0$, and therefore $\chi^{rc}(X_p(k))=0$ for all $p$. Theorem $\ref{chircfubini}$ gives that $\chi^{rc}(X(k))=0$. Similarly $\chi^{rc}(X_p(\kbar))=\mathrm{rank}(\chi^{mot}(X_p))$, so Theorem $\ref{chircfubini}$ gives $\chi^{rc}(X(\kbar))=0$, as required. 

If $k$ is the algebraic closure of a real closed field, then $\chi^{rc}(X(k))=\chi^{mot}(X)$, so we again apply Theorem $\ref{chircfubini}$ to $X(k)$ and obtain $\chi^{mot}(X)=0$.
\end{proof}
We wish to extend the above result to our real closed Fubini theorem for any field of characteristic $0$. We first need the following lemma.
\begin{lemma}\label{sprfg}
Let $n$ be a non-negative integer. Then $\mathrm{Spr}(\Q(t_1, \ldots, t_n)) \neq \emptyset$.
\end{lemma}
\begin{proof}
We prove this by induction, with the case $n=0$ being trivial. Suppose it's true for $n' < n$, let $K = \Q(t_1, \ldots, t_{n-1})$, and let $\prec \in \mathrm{Spr}(K)$. As in Example 1.1.2 of \cite{BCR}, we may define an order $<$ on $K(t_n) = \Q(t_1, \ldots, t_n)$ as follows. Let $p(t_n) \in K[t_n]$ be non-zero. Then $p(t_n) = a_k t_n^k + \ldots + a_{m} t_n^m$ for some non-negative integers $k\leq m$ with $a_k \neq 0$. Then say $0 < p(t_n)$ if $0 \prec a_k$. For a general element $\frac{p(t_n)}{q(t_n)}$ of $K(t_n)$, we say $0 < \frac{p(t_n)}{q(t_n)}$ if and only if $0 < p(t_n)q(t_n)$. As in Example 1.1.2 of \cite{BCR}, this defines a total order on $K(t_n)$ extending $\prec$, as required.
\end{proof}
Suppose now $k$ is an arbitrary field of characteristic $0$.
\begin{lemma}\label{fubinialgclosed}
Let $f: X \to Y$ be a morphism such that for every $y$ a closed point of $Y$, we have $\mathrm{rank}(\chi^{mot}_k(X_y))=0$. Then $\mathrm{rank}(\chi^{mot}(X))=0$.
\end{lemma}
\begin{proof}
This lemma follows by reducing to the case of Theorem $\ref{fubini2}$, and so we need to reduce to the case where $k$ is the algebraic closure of a real closed field. We first reduce to the case where our base field is the algebraic closure of a finitely generated extension of $\Q$. 

Since $X,Y$ are varieties over $k$, we may assume that $X,Y$ are defined over a field $F$ which is finitely generated over $\Q$. Let $X_0$ and $Y_0$ be varieties over $F$ such that $X = X_0 \times_F \Spec(k)$ and $Y=Y_0 \times_F \Spec(k)$. Pick an algebraic closure $\overline{F}$ of $F$, and for every finite extension $F'/F$, pick compatible embeddings $F' \hookrightarrow \overline{F}$. Let $\overline{X} := X_0 \times_F \Spec(\overline{F})$, and similarly for $\overline{Y}$. Then by Lemma $\ref{basechange}$, we see $\mathrm{rank}(\chi^{mot}_k(X)) = \mathrm{rank}(\chi^{mot}_F(X_0)) = \mathrm{rank}(\chi^{mot}_{\overline{F}}(\overline{X}))$. We claim that for every $\overline{y}$ a closed point of $\overline{Y}$ we have $\mathrm{rank}( \chi^{mot}_{\overline{F}}(\overline{X}_{\overline{y}}))=0$.  Let $y \in Y$ be the closed point of $Y$ lying under $\overline{y}$ under the morphism $\overline{Y} \to Y$. Let $F_y$ denote the residue field at $y$, and let $n := [F_y:F]$. Note that $\overline{X}_{\overline{y}}$ is defined over $F_y$, and $\overline{X}_{\overline{y}} = X_y \times_{\Spec(F_y)} \Spec(\overline{F})$. By Lemma $\ref{basechange}$ we see $\mathrm{rank}(\chi^{mot}_{\overline{F}}(\overline{X}_{\overline{y}})) = \mathrm{rank}( \chi^{mot}_{F_y}(X_y))$. Lemma $\ref{trace}$ then gives $\chi^{mot}_F(X_y) = \mathrm{Tr}_{F_y/F}(\chi^{mot}_{F_y}(X_y))$, so taking ranks, we see $0 = n_y \cdot \mathrm{rank}(\chi^{mot}_{F_y}(X_y))$. Therefore $\mathrm{rank}( \chi^{mot}_{F_y}(X_y))=0$ as required. 

Therefore it suffices to prove the statement for $\overline{X} \to \overline{Y}$, where $\overline{X}$ and $\overline{Y}$ are varieties over $\overline{F}$ such that the fibres of the morphism have trivial motivic Euler characteristic. Let $n$ denote the transcendence degree of $F/\Q$. By Noether's normalisation lemma, see for example Theorem 2.1 of Chapter 8 of \cite{La2}, the field $F$ is an finite extension of $\Q(t_1, \ldots, t_n)$. By Lemma $\ref{sprfg}$ we have $\mathrm{Spr}(\Q(t_1, \ldots, t_n)) \neq \emptyset$. Pick $< \in \mathrm{Spr}(\Q(t_1, \ldots, t_n))$, and let $F^{real}$ be the real closure of $\Q(t_1, \ldots, t_n)$ with respect to this ordering. Then $\overline{F}$ is the algebraic closure of $F^{real}$, so the extension $\overline{F}/F^{real}$ is a degree $2$ field extension. We may therefore think of $\overline{X}, \overline{Y}$ and the fibres $\overline{X}_y$ as varieties over $F^{real}$. The result follows by Theorem $\ref{fubini2}$.
\end{proof}

\begin{cor}\label{fubinitorsion}
Let $k$ now be any characteristic $0$ field. Let $X$ and $Y$ be varieties over $k$, and let $f: X \rightarrow Y$ be a morphism. Let $p$ be a closed point of $Y$, and let $X_p$ be the fibre of $f$ over $p$. If $\chi^{mot}(X_p)$ is torsion in $\widehat{\mathrm{W}}(k)$ for all $p$, then $\chi^{mot}(X)$ is torsion.
\end{cor}
\begin{proof}
By Pfister's local-global principle, we see $\chi^{mot}(X_p)$ is torsion if and only if we have $\mathrm{rank}(\chi^{mot}(X_p))=0$ and $\chi^{mot}(X_{k_<, p})=0$ for all $< \in \mathrm{Spr}(k)$, and similarly for $\chi^{mot}(X)$. If $\mathrm{Spr}(k) = \emptyset$, then the only conditions are on the rank, so the result follows by Lemma $\ref{fubinialgclosed}$. Suppose then that $\mathrm{Spr}(k) \neq \emptyset$. Let $< \in \mathrm{Spr}(k)$. Since $\chi^{mot}$ commutes with base change, we have $\chi^{mot}(X_{k_<,y})=0$ for all closed points $y$ of $Y$. Theorem $\ref{fubini2}$ gives $\chi^{mot}(X_{k_<})=0$, as required.
\end{proof}
\subsection{The determinant of cohomology}

Suppose now that $k$ is an arbitrary field of characteristic $0$. In this subsection, we show a Fubini theorem for the $\ell$-adic determinant of cohomology which is not a na{\"i}ve Fubini theorem. Let $U$ be a smooth scheme over $k$ and $\mathcal{F}$ be a smooth
$\ell$-adic sheaf for $\ell\neq\mathrm{char}(k)$. Consider the one-dimensional $\ell$-adic representation of $\Gal_k$
$$\det R\Gamma_{c}(U_{\overline{k}}, \mathcal{F})=\bigotimes_{i}\det H_{c}^{i}(U_{\overline{k}}, \mathcal{F})^{\otimes(-1)^i}.$$
In order to prove our Fubini theorem, we make use of the following theorem, which is the main theorem of \cite{S0}.
\begin{thm}\label{saito-big} Assume that there is a smooth projective variety $X/k$ which contains $U$ such that $D= X \setminus U$ is a divisor with simple normal crossings. Moreover, assume that the ramification of $\mathcal{F}$ on $D$ is tame. Finally, assume that $\mathcal{F}$ is defined on a model of $U$ which is defined over a ring of finite type over $\mathbb Z$. 
Then we have
$$\det R\Gamma_{c}(U_{\overline{k}},\mathcal{F})\otimes\det R\Gamma_{c}(U_{\overline{k}}, \mathbb{Q}_{\ell})^{\otimes-\mathrm{rank}\mathcal F}=c_{X,U/k}^{*}(\det \mathcal{F})\otimes J_{D,\mathcal{F}}^{\otimes(-1)}$$
as one-dimensional $\ell$-adic representations of
$\mathrm{Gal}_k$.
\end{thm}
The first term on the right hand side is the pullback of the determinant character $\det \mathcal{F}$ of $\pi^{\et}_{1}(U)^{ab,tame}$ to $\mathrm{Gal}_k$ by the pairing with the relative canonical class $c_{X,U/k}\in CH_{0}(X, D)$. In this paper, we only need that $c^*_{X,U/k}$ respects tensor products, and we will only ever consider this in the case when $\mathrm{det}(\mathcal{F})$ is trivial, so this term will always vanish. We refer to \cite{S0}, the bottom of page 414, for the precise construction of $c_{X,U/k}^*$. We will however require the definition of $J_{D,\mathcal{F}}$, as defined in the paragraph before Theorem 1 on page 424, as well as the definition on page 416 and Proposition 2 on page 417. We will recall these below for the convenience of the reader. 
\begin{defn}Let $A$ denote a regular noetherian ring, let $S=\Spec(A)$, and suppose that $S$ is connected with function field $k$. Following \cite{S0}, define a \emph{Jacobi datum} on $S$ to be a triple $(T, \chi, n)$ such that
\begin{enumerate}
\item[(1)] $T=(T_{i})_{i\in I}$ is a finite family of finite étale covers $T_i \to S$. 
\item[(2)] $\chi=(\chi_{i})_{i\in I}$ is a family of characters $\chi_{i}:\mu_{d_i}(T_i)\to \overline{\mathbb{Q}_{l}}^\times$ of the group of
$d_i$-th roots of unity for some integer $d_i$ invertible in $A$ such that $\mu_{d_i}\simeq\mathbb Z/d_i\Z$ on $T_i$.
\item[(3)] $n=(n_{i})_{i\in I}$ is a family of integers satisfying the norm condition below. 
\end{enumerate}
Write $T_i= \coprod_{j \in J_i} T_{ij}$ where each $T_{ij}$ is connected. Let the norm $N_{T_{ij}/S}(\chi_i)$ denote the product of the conjugates, and let the norm $N_{T_i/S}$ be given by
\begin{align*}
N_{T_i/S}(\chi_i):\mu_{d_i}(T_i)&\to \overline{\mathbb{Q}_{\ell}}^\times\\
 \zeta: &\mapsto \prod_{j} \chi_i(N_{T_{ij}/S}(\zeta)).
\end{align*}
By construction, each $N_{T_i/S}(\chi_i)$ factors through some $\mu_{d_i'}(T_i)$ such that $d_i' | d_i$ and $\mu_{d_i'} = \Z/d_i'\Z$ on $S$. 
Let $d$ be a common multiple of all of the $d'_i$s, and regard $N_{T_i/S}(\chi_i)$ as a character on $\mu_d$ via the composition
\begin{align*}
\mu_d &\to \mu_{d'_i}\\
\zeta &\mapsto \zeta^{\frac{d}{d'_i}}.
\end{align*}
The norm condition for our Jacobi datum is that 
$\prod_{i\in I}N_{T_i/S}(\chi_i)^{n_i}=1$, where this product is considered as a product of characters of $\mu_d$ on $S$. 

If $A'$ is another ring equipped with a morphism $f: \Spec(A') \to \Spec(A)$, we may pull back the Jacobi datum $(T,\chi,n)$ on $\Spec(A)$ to $\Spec(A')$ to obtain a Jacobi datum $f^*(T, \chi, n) = (f^*T, f^*\chi, n)$, where $f^*(T) = (f^*T_i)$ is given by the pullback of the finite étale covers $T_i \to \Spec(A)$ along the morphism $f$, and $f^*\chi_i$ is the composition $\chi_i: \mu_{d_i}(f^*T_i) \to \mu_{d_i}(T_i) \to \Q_{\ell}^\times$. We say that the Jacobi datum $(T',\chi',n')$ on $A'$ is \emph{defined on $A$} if there exists a Jacobi datum $(T,\chi,n)$ on $A$ such that we may obtain $(T',\chi',n')$ by pulling back the Jacobi datum $(T,\chi,n)$ along the morphism $\Spec(A') \to \Spec(A)$.  If $s$ is a closed point of $A$ with residue field $\kappa(s)$, we will say that the \emph{reduction of $(T,\chi,n)$ at $s$} is the Jacobi datum on $\kappa(s)$ given by pulling back $(T,\chi,n)$ along the closed immersion corresponding to $s$.
\end{defn}

\begin{defn}
Let $\F$ be a finite field of order $q$. Let $\chi$ be a multiplicative character, i.e.~a group homomorphism $\chi: \F^\times \to \overline{\mathbb{Q}_{l}}^\times$. Let $\psi$ be an additive character, i.e.~a group homomorphism $\psi: \F \to \overline{\mathbb{Q}_{\ell}}^\times$. Define the \emph{Gauss sum}
$$
G_\F(\chi, \psi) := -\sum_{a \in \F^\times} \chi^{-1}(a)\psi(a).
$$
\end{defn}

\begin{defn} Let $\F$ be a finite field of order $q$. Then if we have a Jacobi datum $(T, \chi, n)$ on $\Spec(\F)$, we see that each $T_i = \coprod_{j \in \overline{I}_i} \Spec(\F_{ij})$ for some $\F_{ij}/\F$ a finite extension and $\overline{I}_i$ a finite index set. Define the \emph{Jacobi sum} $j_{\chi}=j_{T,\chi,n}$ for a
Jacobi datum $(T, \chi, n)$ on $\Spec(\F)$ by
$$j_{\chi}:= \prod_{i\in I} \prod_{j \in \overline{I}_i}(G_{\F_{ij}}(\overline{\chi}_{ij}, \psi_0\circ\mathrm{Tr}_{\F_{ij}/\F}))^{n_i}.$$
Here $\psi_0$ is a non-trivial additive character of $k$, and if $\F_{ij}$ is of order $q_{ij}$ then $\overline{\chi}_{ij}$ is a multiplicative character of $\F_i$ defined by
$$\overline{\chi}_{i}(a):=\chi_{i}(a^{(q_{ij}-1)/d_i}).$$
As mentioned on page 417 of \cite{S0}, we can consider characters on $\F^\times$ as characters of $\mu_{q-1}(\F)$. Moreover we have $\prod_{j \in \overline{I}_i} \overline{\chi}_i |_{\F^\times} = N_{T_i/\Spec(S)}(\chi_i)$. As mentioned on page 417 of \cite{S0}, Proposition 4.15 of \cite{SGA} tells us that the product $j_{\chi}$ is independent of our choice of $\psi_{0}$, since $\prod_{i\in I}N_{\F_i/\F}(\chi_{i})^{n_i}=1$.
\end{defn}
\begin{defn}\label{jacobicharacter} As in Proposition $2$ of \cite{S0}, we can associate an $\ell$-adic representation $J_{(T,\chi,n)}$ of $\Gal_k$ to each Jacobi datum $(T, \chi, n)$ on a field $k$ as follows. A Jacobi datum on $k$ is defined on a regular ring $A$ of finite type over $\mathbb Z$, since each of the $T_i$s are varieties over $k$. The representation $J_{(T,\chi,n)}$ is the pullback of the representation of $\pi_{1}^{\et}(\mathrm{Spec}(A))^{ab}$ characterised by the following condition: for each closed point $s$ of $\mathrm{Spec}(A)$, the action of the geometric Frobenius $\mathrm{Frob}_s$ at $s$ is given by the multiplication by the Jacobi sum $j_{\chi}(s)$ defined by the reduction of the Jacobi datum $(T, \chi, n)$ at $s$. Proposition 2 of \cite{S0} guarantees that both that this character exists and is unique. Say that $J_{(T,\chi,n)}$ is the \emph{associated $\ell$-adic character} of the Jacobi datum $(T, \chi, n)$. 
\end{defn}
\begin{defn}[(Saito, \cite{S0})]\label{JSDdefn} Let $U$ and $\mathcal{F}$ be as in Theorem \ref{saito-big}. We can define a Jacobi datum on $k$ associated to the ramification of $\mathcal{F}$ along $D$. Let $D_i$ be the irreducible components of $D$ with constant field $k_i$. Let $\rho$ be the $\ell$-adic representation of $\pi_1
(U,\overline{x})^{tame}$ corresponding to $\mathcal{F}$. The kernel of the natural morphism $\pi^{\et}_1(U,\overline{x})^{tame}\to\pi_{1}^{\et}(X,\overline{x})$ is the normal subgroup topologically generated by the local monodromy groups $\widehat{\mathbb Z}'(1)_{D_i}$ along the $D_{i}$'s where $\widehat{\mathbb Z}'(1)=\varprojlim\mu_d$ with $d$ invertible in $k$. Let $\rho_{i}$ be the restriction of $\rho$ to $\widehat{\mathbb Z}'(1)_{D_{1}}\simeq\widehat{\mathbb Z}'(1)_{k_i}$. 

Since we have a model of finite type over $\mathbb Z$, Grothendieck's monodromy theorem, Théor\`eme de monodromie 1.2 of \cite{SGA71}, tells us that the $\rho_{i}$s are quasi-unipotent. Therefore the action of $\widehat{\mathbb Z}'(1)$ on the semi-simplification $\rho_{i}^{ss}$ factors through a finite quotient. This means can decompose 
$$
\rho_{i}^{ss}\simeq\bigoplus_{j\in I_{i}}\mathrm{Tr}_
{k_{ij}/k_{i}}(\chi_{ij}).
$$
Here $k_{ij}$ is the finite extension of $k_{i}$ obtained by adjoining the $d_{ij}$-th roots of unity, $\chi_{ij}$ is a character of $\mu_{d_{ij}}(k_{ij})$ of order $d_{ij}$, and where we define $\mathrm{Tr}_{k_{ij}/k_i}$ to be the $[k_{ij}:k_i]$ dimensional $\widehat{\mathbb Z}'(1)$ representation
$$
\mathrm{Tr}_{k_{ij}/k}(\chi_{ij}) := \bigoplus_{\sigma \in \Gal(k_{ij}/k_i)} \chi_{ij} \circ \sigma.
$$
For $i\in I$ let $D_{i}^{*}=D_{i}- \bigcup_{J\neq i}D_{j}$ and $c_{i}$ be the compactly supported Euler characteristic of $(D_{i}^{*}\times_{k_i} \Spec(\overline{k_i}))$. Define our Jacobi datum $(T, \chi, n)$ by $\overline{I}=\coprod_{i}I_{i},$ $T=(k_{ij}),$ $\chi=(\chi_{ij})$ and $n=(n_{ij})$ with $n_{ij}=c_{i}$ for $i\in I$ and $j\in I_{i}$. Define the term $J_{D,\mathcal{F}}$ to be the $\ell$-adic representation $J_{(T, \chi, n)}$.
\end{defn}
\begin{defn}\label{sfacdefn}
Let $f: W \to Y$ be a projective morphism of positive relative dimension. Assume there is a smooth and projective variety $Z$ over $k$ containing $Y$ such that the complement $D=Z-Y$ is a divisor with simple normal crossings and $R^if_*\mathbb{Q}_{\ell}$ is smooth over $Y$ and the ramification along $D$ is tame for all $i$. We define 
$$
\prod_{i=0}^{2(\dim(W)-\dim(Y))}J_{D,R^if_*\mathbb{Q}_{\ell}}^{\otimes(-1)^{i+1}}
$$
to be the \emph{Saito factor} of $f$ at $D$. 
\end{defn}

\begin{thm} Let $f: W \to Y$ be be a projective morphism of positive relative dimension. Assume that there is a projective and smooth variety $Z$ over $k$ containing $Y$ such that the complement $D=Z-Y$ is a divisor with simple normal crossings and $R^if_*\mathbb{Q}_{\ell}$ is smooth and the ramification is tame for all $i$. Moreover, assume that $\sum_i (-1)^i \mathrm{rank}(R^if_*\Q_{\ell})=0$, $\bigotimes_i \det R^if_*\mathbb{Q}_{\ell}^{(-1)^i}$ is trivial, and $f$ is defined on a model over a ring of finite type over $\mathbb{Z}$. Then
$$\mathrm{det}_{\ell}(W)=\prod_{i=0}^{2(\dim(W)-\dim(Y))}J_{D,R^if_*\mathbb{Q}_{\ell}}^{\otimes(-1)^{i+1}}.$$
\end{thm}
\begin{proof} By the Leray spectral sequence:
$$\mathrm{det}_{\ell}(W)=\!\!\!\!\!\!\!\!\!\!\!\!\!\!
\prod_{i=0}^{2(\dim(W)-\dim(Y))}
\!\!\!\!\!\!\!\!\!\!\!\!\!\!
\det R\Gamma_c(Y_{\overline{k}},R^if_*
\mathbb{Q}_{\ell})^{(-1)^i}.$$
 Since $f$ is defined on a model over a ring of finite type over
$\mathbb Z$, the same is true for the sheaves $R^if_*\mathbb{Q}_{\ell}$, so we apply Saito's theorem, Theorem $\ref{saito-big}$, to obtain
\begin{align*}
\prod_i (\det R\Gamma_c(Y_{\overline{k}},R^if_*
\mathbb{Q}_{\ell})^{(-1)^i} )   &\otimes
\mathrm{det}_{\ell}(Y)^{\otimes-\sum_i(-1)^i\mathrm{rank}R^if_*
\mathbb{Q}_{\ell}}=\\
&c_{Y,D/k}^{*}\left(\bigotimes_i
(\det R^if_*\mathbb{Q}_{\ell})^{(-1)^i}\right)\otimes
\prod_iJ_{D,R^if_*\mathbb{Q}_{\ell}}^{\otimes(-1)^{i+1}}
\end{align*}
using the bilinearity of the pairing $CH_{0}(X, D)\times
\pi^{\et}_{1}(U)^{ab,tame}\to\mathrm{Gal}_k$. By our assumptions on the rank and determinants of $R^if_*\mathbb{Q}_{\ell}$, we see that
 \begin{align*}
\mathrm{det}_{\ell}(Y)^{\otimes-\sum_i(-1)^i\mathrm{rank}R^if_*\mathbb{Q}_{\ell}} &= 1\\\
c_{Y,D/k}^{*}\left(\bigotimes_i(\det R^if_*\mathbb{Q}_{\ell})^{(-1)^i}\right)&=1,
\end{align*}
where $1$ denotes the trivial representation of $\Gal_k$. Since the sheaves $R^if_*\Q_\ell$ vanish whenever $i > 2(\mathrm{dim}(W)-\mathrm{dim}(Y))$, this gives the result.
\end{proof}

\begin{cor}\label{saitofactorsurface}
Let $f: W \to Y$ be be a projective morphism of positive relative dimension $g$. Assume that there is a projective and smooth variety $Z$ over $k$ containing $Y$ such that the complement $D=Z-Y$ is a divisor with simple normal crossings and $R^if_*\mathbb{Q}_{\ell}$ is smooth and the ramification along $D$ is tame for all $i$. Moreover, assume that for every closed point of $Y$ the fibre $W_y$ has $\chi^{mot}(W_y) \equiv 0 \pmod{J}$. Then the motivic Euler characteristic modulo $J$ satisfies $\chi^{mot}(W) \equiv \kappa_W \pmod{J}$, where $\kappa_W$ is an element of $\widehat{\mathrm{W}}(k)/J$ such that $\mathrm{rank}(\kappa_W)=0$, $\mathrm{sign}(\kappa_W)=0$ and $\mathrm{disc}(\kappa_W)=\mathrm{det}_{\ell}(W) =  \prod_{i=0}^{2g}J_{D,R^if_*\mathbb{Q}_{\ell}}^{\otimes(-1)^{i+1}}$.
\end{cor}
\begin{proof}
The fact that $\mathrm{rank}(\kappa_{\mathcal{C}})$ and $\mathrm{sign}(\kappa_{\mathcal{C}})$ both vanish is due to our real closed Fubini theorem, Corollary $\ref{fubinitorsion}$. Consider the condition $\chi^{mot}(W_y) \equiv 0 \pmod{J}$ for all $y$. Taking ranks tells us that $\mathrm{rank}(\chi^{mot}(W_y))=0$ for all $y$, and applying proper base change tells us that $\sum_i (-1)^i \mathrm{rank}(R^if_*\Q_{\ell})=0$. Similarly, we see $\mathrm{det}_{\ell}(W_y)=0$ for all $y$, which is equivalent to saying $\bigotimes_i \det (R^if_*\mathbb{Q}_{\ell})^{(-1)^i}$ is trivial as well. The above theorem gives us the result.
\end{proof}

\begin{rem}
This Fubini theorem reduces our computation of $\chi^{mot} \pmod{J}$ to Saito factors associated to the boundary of our morphism, thought of as characters in $\Hom(\Gal_k, \overline{\Q}_{\ell}^\times)$. Since $\mathrm{disc}(\kappa_W)=\mathrm{det}_{\ell}(W) =  \prod_{i=0}^{2g}J_{D,R^if_*\mathbb{Q}_{\ell}}^{\otimes(-1)^{i+1}}$, and $\mathrm{disc}(\kappa_W)$ is a character of order $2$ by construction, the Saito factor $\prod_{i=0}^{2g}J_{D,R^if_*\mathbb{Q}_{\ell}}^{\otimes(-1)^{i+1}}$ will also be a character of order $2$. For the rest of this section, we show that these characters of $\Gal_k$ naturally factor through the absolute Galois group of a number field, and we can control the ramification of these characters.
\end{rem}

\begin{defn}
For $k$ a field, let $k^{const}$ denote the algebraic closure of its prime subfield in $k$. In particular, if $k$ is a field which is finitely generated over $\Q$, the field $k^{const}$ is a number field.
\end{defn}

\begin{lemma}
Let $(T,\chi,n)$ be a Jacobi datum on $k$, a field of characteristic $0$. Let $d$ denote the lowest common multiple of the $d_j$s, and let $A := \Spec(\mathcal{O}_{k^{const}}[\frac1d])$. Then there exists a Jacobi datum $(T',\chi',n')$ on $\Spec(A)$ such that the character $J_{(T,\chi,n)}$ is given by the composition
$$
\Gal_k \to \pi_1^{\et}(\Spec(A)) \xto{J_{(T',\chi',n')}} \overline{\Q_{\ell}}^\times.
$$
\end{lemma}
\begin{proof}
This follows instantly by the second half of the Corollary on page 418 of \cite{S0}.
\end{proof}
\begin{cor}
The character $J_{(T,\chi,n)}$ factors as $\Gal_k \twoheadrightarrow \Gal_{k^{const}} \to \overline{\mathbb{Q}_{\ell}}^\times$.
\end{cor}
\begin{proof}
First note that$J_{(T,\chi,n)}$ factors through the map $\Gal_k \to \pi_1^{\et}(\Spec(A))$, which in turn factors as $\Gal_k \to \Gal_{k^{const}} \to \pi_1^{\et}(\Spec(A))$. The map $\Gal_k \to \Gal_{k^{const}}$ is surjective since $k^{const}$ is algebraically closed in $k$ by construction.
\end{proof}
\begin{defn}
Since $\Gal_k \to \Gal_{k^{const}}$ is surjective, the associated character $\Gal_{k^{const}} \to \overline{\mathbb{Q}_{\ell}}^\times$ appearing in the corollary above is unique. Call this character $J_{(T,\chi,n)}^{const}$. 
\end{defn}

\begin{cor}
Let $p$ be a prime number coprime to $d_j$ for every $j$, and let $\mathfrak{p}$ be a prime of $k^{const}$ with residue characteristic $p$. Then the character $J_{(T,\chi,n)}^{const}$ is unramified at $\mathfrak{p}$.
\end{cor}
\begin{proof}
By construction, $J_{(T,\chi,n)}^{const}$ factors as $\Gal_{k^{const}} \to \pi_1^{\et}(\Spec(A)) \to  \overline{\mathbb{Q}_{\ell}}^\times$. For $\mathfrak{p}$ as above, we see that $\mathfrak{p} \in \Spec(A)$.  Therefore the inertia group of $\mathfrak{p}$ is in the kernel of the morphism $\Gal_{k^{const}} \to \pi_1^{\et}(\Spec(A))$. In particular, it is in the kernel of $J_{(T,\chi,n)}^{const}$, as required. 
\end{proof}

\begin{cor}
Suppose we are now in the situation of Theorem $\ref{saito-big}$, and let $J_{D,\mathcal{F}}$ be the $\ell$-adic character appearing in the theorem. Suppose that $p$ is a prime number satisfying 
$$
p > \mathrm{rank}(\mathcal{F}) \cdot [k^{const}:\Q] + 1,
$$ 
and let $\mathfrak{p}$ be a prime of $k^{const}$ with residue characteristic $p$. Then $J_{D,\mathcal{F}}^{const}$ is unramified at $\mathfrak{p}$. 
\end{cor}
\begin{proof}
By the above Corollary, we only need to show that $p$ does not divide any $d_j$. The $d_j$ appear through the decomposition of the semisimplification of the monodromy representation associated to $\mathcal{F}$:
$$
\rho^{ss} = \bigoplus_j \mathrm{Tr}_{k_j/k}(\chi_j),
$$
where $\chi_j$ is a representation on $\mu_{d_j}(k_j)$, and moreover, $\mathrm{dim}(\rho^{ss}) = \mathrm{rank}(\mathcal{F})$. Note that 
$$
\mathrm{dim}\mathrm{Tr}_{k_j/k}(\chi_j) = [k_j:k] \geq [k(\zeta_{d_j}):k] = [k^{const}(\zeta_{d_j}):k^{const}],
$$
 where $\zeta_{d_j}$ denotes a primitive $d_j^{\text{th}}$ root of unity. We also see
 \begin{align*}
 [k^{const}(\zeta_p):\Q] &= [k^{const}(\zeta_p):k^{const}] \cdot [k^{const}:\Q],\\
 [k^{const}(\zeta_p):\Q] &\geq [\Q(\zeta_p):\Q] = p-1.
 \end{align*}
  Putting this together gives that $[k^{const}(\zeta_p):k^{const}] \geq \frac{p-1}{[ k^{const}:\Q]}$, and by the assumption on $p$, this gives $[k^{const}(\zeta_p):k^{const}]> \mathrm{rank}(\mathcal{F})$. Suppose that $p$ divides $d_j$ for some $d_j$, seeking a contradiction. Then 
$$
\mathrm{dim}(\mathrm{Tr}_{k_j/k}(\chi_j)) \geq [k^{const}(\zeta_{d_j}):k^{const}] \geq [k^{const}(\zeta_p):k^{const}] >  \mathrm{rank}(\mathcal{F}) = \mathrm{dim}(\rho^{ss}),
$$
which is a contradiction, since $\mathrm{Tr}_{k_j/k}(\chi_j)$ is a factor in the direct sum decomposition of $\rho^{ss}$. 
\end{proof}
\begin{cor}\label{restrictedramification}
Suppose we are in the situation of Corollary $\ref{saitofactorsurface}$. Then the character $\mathrm{disc}(\kappa_W)$ factors as
$$
\Gal_k \twoheadrightarrow \Gal_{k^{const}} \xto{\kappa_W^{const}} \{\pm 1\} \subseteq \mathbb{Q}_{\ell}^\times,
$$ 
and the character $\mathrm{disc}(\kappa_W^{const})$ is unramified at all primes of $k^{const}$ whose residue characteristic is bigger than $\mathrm{rank}(R^if_*\mathbb{Q}_{\ell}) \cdot [k^{const}:\Q] + 1$ for all $i$.
\end{cor}
\begin{proof}
We see that $\mathrm{disc}(\kappa_W^{const})$ is given by $\prod_i J_{D,R^if_*\Q_{\ell}}^{(-1)^i}$. Note that each $J_{D,R^if_*,\Q_{\ell}}$ is unramified at all primes whose residue characteristic is strictly bigger than $\mathrm{rank}(R^if_*\mathbb{Q}_{\ell}) \cdot [k^{const}:\Q] + 1$. Therefore, if a prime has residue characteristic bigger than $\mathrm{rank}(R^if_*\mathbb{Q}_{\ell}) \cdot [k^{const}:\Q] + 1$ for all $i$, it will be unramified for all of the characters $J_{D,R^if_*\Q_{\ell}}$, and so will be unramified for $\mathrm{disc}(\kappa_W)$.
\end{proof}

\subsection{A na{\"i}ve Fubini theorem for algebraic groups}
In the case that $f: X \to Y$ is a torsor under an algebraic group, we can achieve a stronger Fubini theorem by using Lang's theorem.
\begin{thm}\label{fubini1} Let $G$ a connected algebraic group over $k$ and let $f:X\to Y$ be such that for every closed point $y \in Y$ with residue field $\kappa(y)$,  the fibre $X_y$ over $y$ is a torsor under $G_{\kappa(y)}$. Then 
$$
\chi^{mot}(X)\equiv\chi^{mot}(G)\cdot\chi^{mot}(Y) \pmod{J}.
$$
\end{thm}
\begin{proof}
We prove this statement for the rank, signature, and discriminant. For the rank and signature, this follows by Theorem $\ref{fubini2}$, so it only remains to show this for the discriminant. Without loss of generality, assume that $k$ is finitely generated over $\mathbb{Q}$. By Proposition \ref{rel} it is enough to show a similar claim for the augmented determinant of cohomology. We can assume that $X \to Y$ and $G$ are defined on a ring of finite type over $\Z$, denoted by $A$, and let $\mathcal{X}, \mathcal{Y}$ denote smooth projective models for $X,Y$ over $\Spec(A)$. By Chebotar\"ev's density theorem, it is enough to prove that for every $s$ a closed point of $\Spec(A)$ such that the fibre of $X$ over $s$ is smooth
$$
\mathrm{det}_{\ell}(X)(\mathrm{Frob}_s) = \mathrm{det}_{\ell}(G)(\mathrm{Frob}_s) \cdot \mathrm{det}_{\ell}(Y)(\mathrm{Frob}_s).
$$
Theorem $\ref{zetafunctionsdetermine}$ tells us that $\mathrm{det}_{\ell}(X)(\mathrm{Frob}_s)$ is determined by the $\zeta$-function of $\mathcal X_s$. Write $X_s, Y_s$ and $G_s$ for the reductions of $\mathcal X, \mathcal Y$ and $G$ at $s$ respectively, and let $\F_s$ denote the residue field of $s$. The result reduces to showing $\zeta(X_s, t) = \zeta(G_s \times Y_s,t)$. Lang's theorem, Theorem 2 of \cite{Lang}, gives us $H^1(\F_s,G)=0$, since $\F_s$ is finite. Therefore the fibres of $f_s: X_s \to Y_s$ are all isomorphic to a suitable base change of $G$. Let $p$ be a closed point of $Y_s$ with residue field $\F_p$, and let $X_p$ be the fibre of $X_s$ over $p$. Then $X_p \cong G_{\F_p}$. Counting the points of $X$ gives an equality of zeta functions $\zeta(X_s, t) = \zeta(Y_s \times_{\F_s} G,t)$ as required.
\end{proof}

\begin{thm}\label{abelian} Let $X$ be a positive dimensional abelian variety over $k$. Then $\chi^{mot}(X)=0$.
\end{thm}
\begin{proof}
Since $X$ is smooth and proper we have $\chi^{mot}(X) = \chi^{dR}(X)$, so we may work with de Rham cohomology. By Lemma 15.2 of \cite{Mil}, we have an isomorphism of graded commutative rings $H^*_{dR}(X) \cong \bigoplus_{i=0}^{2n} \bigwedge^i H^1_{dR}(X)$. In particular, this implies that $\mathrm{rank}(\chi^{dR}(X)) = \sum_{i=0}^{2n} (-1)^i {n \choose i} = 0$. 

Moreover, choosing a basis of $H^1_{dR}(X)$ and taking the wedge product gives rise to a basis of $H^n_{dR}(X) = \bigwedge^n H^1_{dR}(X)$. With respect to this basis, the quadratic form on $H^n_{dR}(X)$ is clearly hyperbolic, so there is some $m$ such that $[H^n_{dR}(X)]=m\cdot \mathbb{H}$. Applying the previous lemma tells us that $\chi^{dR}(X) = (b^++m) \cdot \mathbb{H}$.  Finally, note that $\mathrm{rank}(\chi^{dR}(X)) = 2(b^++m)$. Since $\mathrm{rank}(\chi^{dR}(X)) = 0$ this gives the result.
\end{proof}
\begin{rem}
Readers familiar with motivic homotopy theory should note that we may prove the above theorem directly using the categorical Euler characteristic and the motivic Gauss--Bonnet formula of \cite{LR}. Let $p_X$ denote the structure morphism $p_X: X \to \Spec(k)$. Since $X$ is smooth and projective, we have $\chi^{mot}(X) = \chi^{cat}(X)$ by Theorem $\ref{levine1}$. Page $3$ of \cite{LR} allows us to define an Euler class of the tangent bundle of $X$ valued in the Hermitian $K$-theory spectrum of $X$, which we call $e^{BO}(T_{X/k})$. Since $X/k$ is an abelian variety, $T_{X/k}$ is trivial, so this Euler class vanishes. The motivic Gauss--Bonnet formula, Theorem 1.5 of \cite{LR}, says that $\chi^{cat}(X) = p_{X,*}( e^{BO}(T_{X/k}))$. This is akin to the classical Gauss--Bonnet formula in that it says that the Euler characteristic of a variety may be obtained by integrating the Euler class of the tangent bundle of the variety. The theorem follows since $e^{BO}(T_{X/k})$ is trivial, as $T_{X/k}$ is.
\end{rem}
\begin{cor}\label{fubini1b} Let $A$ be a connected commutative algebraic group over $k$ whose maximal proper quotient group variety has positive dimension. Then $\chi^{mot}(A) \in J$.
\end{cor}
\begin{proof} There is a short exact sequence of algebraic groups:
$$0 \to G \to A \xto{\pi}  B
\to 1,$$
where $G$ is a connected linear algebraic group and $B$ is the maximal proper quotient group variety of $X$. The map $\pi:A\to B$ is a principal $G$-bundle, so by Theorem $\ref{fubini1}$, we see that $\chi^{mot}(A)\equiv \chi^{mot}(G)\cdot\chi^{mot}(B) \pmod{J}$. Since $B$ is a positive dimensional abelian variety, Theorem \ref{abelian} gives $\chi^{mot}(B)=0$ as required.
\end{proof}

\section{The motivic Euler characteristic of compactified Jacobians of curves}

\subsection{The motivic Euler characteristics of compactified Jacobians of nodal curves}

\begin{defn}
Let $G$ be a one-dimensional torus. The group of cocharacters is a $\Gal_k$-module which is isomorphic to $\mathbb Z$ as a group. Let $\chi_G:\Gal_k
\to\mathrm{Aut}(\mathbb Z)\cong\mathbb Z/2\mathbb Z$ denote the homomorphism corresponding to this group action. Since we can identify $k^\times/k^{\times2}$ with $\mathrm{Hom}(\Gal_k,\mathbb Z/2\mathbb Z)$, there exists some $\alpha\in k^\times/k^{\times2}$ corresponding to $\chi_G$. We will write $\mathbb{G}_{m,k}^{[\alpha]}$ for the one dimensional torus whose cocharacter lattice has $\Gal_k$ action corresponding to the class $\alpha \in k^\times/k^{\times2}$. A model for $\mathbb{G}_{m,k}^{[\alpha]}$ is given by the subscheme $x^2-\alpha y^2 = 1$ in $\mathbb{A}^2_k$.
\end{defn}

\begin{lemma}\label{1dtorus} We have $\chi^{mot}(\mathbb{G}_{m,k}^{[\alpha]})=\mathbb H-
[\langle2\rangle\oplus\langle2\alpha\rangle] = [\langle -2 \rangle] - [\langle 2\alpha \rangle]$.
\end{lemma}
\begin{proof} Note that we have a smooth compactification of $\mathbb{G}_{m,k}^{[\alpha]}$ sitting inside $\mathbb{P}^1_k$ with compliment
$\mathrm{Spec}(k(\sqrt{\alpha}))$. Using the additivity of the motivic Euler characteristic we obtain
$$\chi^{mot}(\mathbb{G}_{m,k}^{[\alpha]})=\chi^{mot}(\mathbb P_k^1)-
\chi^{mot}(\mathrm{Spec}(k(\sqrt{\alpha})))
=\mathbb H-[\mathrm{Tr}_{k(\sqrt{\alpha})/k}].$$
 Computing the trace form in the basis $1,\sqrt{\alpha}$ gives $[\mathrm{Tr}_{k(\sqrt{\alpha})/k}]=[\langle2\rangle\oplus\langle2\alpha\rangle]$. The final part of the claim follows since $\mathbb{H} = \langle 1 \rangle + \langle -1 \rangle  = \langle 2 \rangle + \langle -2 \rangle$. 
\end{proof}

\begin{defn}\label{compactifiedjac}
Following Definition 5.9 of \cite{AK}, let $f: Y \to S$ be a finitely presented morphism with integral geometric fibres. An $\mathcal{O}_Y$ module $I$ is called \emph{relatively torsion free and rank $1$ over $S$} if $I$ is locally finitely presented, $S$-flat, and the fibre $I_s$ is a rank $1$ torsion free $\mathcal{O}_{Y_s}$ module for every geometric point $s$ of $S$. 

For the rest of this subsection let $C$ denote a projective, geometrically integral curve over $k$. As in Definition 5.11 and Theorem 8.1 of \cite{AK}, rank $1$ torsion free sheaves on $C$ of degree $n$ are parameterized by the \emph{compactified Jacobian} $\overline{\mathrm{Pic}}^n(C)$, which by Theorem 8.1 of \cite{AK} is a representable as an étale sheaf by a $k$ scheme which we will also call $\overline{\mathrm{Pic}}^n(C)$. Similarly by Theorem 8.1 of \cite{AK}, if $\mathcal{C} \to \mathbb{P}^g$ is a complete linear system of genus $g$ curves, then we may form the relative compactified Jacobian $\overline{\mathrm{Pic}}^g(\mathcal{C}) \to \mathbb{P}^g$, whose fibres are the compactified Jacobians of the fibres of $\mathcal{C} \to \mathbb{P}^g$.
\end{defn}
Let $C$ be a curve of genus $g$. The insight to prove the complex Yau--Zaslow formula is that $e(\overline{\mathrm{Pic}}^g(C))=0$ unless $C$ is rational, in which case, if $C$ is nodal this counts the appropriate multiplicity of rational curves. Therefore counting the number of rational curves can be done by looking at the Euler characteristic of the compactified Jacobians of the linear system. Here we compute $\chi^{mot}(\overline{\mathrm{Pic}}^g(C)) \pmod{J}$ instead. We do this by first showing that if $C$ is not rational, then $\chi^{mot}(\overline{\mathrm{Pic}}^g(C))=0 \pmod{J}$. For $C$ a nodal and rational curve, we compute $\chi^{mot}(\overline{\mathrm{Pic}}^g(C))$ by successively expressing $\chi^{mot}(\overline{\mathrm{Pic}}^g(C)) \pmod{J}$ in terms of the motivic Euler characteristic of the compactified Jacobian of the normalisation of $C$ at a singular point.

\begin{defn} 
 A {\it partial normalisation} of $C$ is a surjective map $f:C'\to C$ of integral projective curves which induces isomorphism between the function fields of $C'$ and $C$. Two such partial normalisations $f_1:C_1\to C$ and $f_2:C_2\to C$ are isomorphic if there is an isomorphism $g:C_1\to C_2$ such that $f_1=f_2\circ g$. Since every partial normalisation is dominated by the normalisation of $C$, the set $\mathcal N(C)$ of isomorphism classes of partial normalisations is finite. 
\end{defn}

There is a stratification of $\overline{\mathrm{Pic}}^n(C)$ by locally closed subvarieties 
$$
\overline{\mathrm{Pic}}^n(C)=\coprod_{f\in\mathcal N(C)}
\overline{\mathrm{Pic}}^n_f(C).
$$
Here, the morphism $f: C' \to C$ runs over all partial normalisations of $C$, and $\overline{\mathrm{Pic}}^n_f(C)$ is such that the $\overline k$-valued points of $\overline{\mathrm{Pic}}^n_f(C)$ correspond exactly to those  rank $1$ sheaves ${\mathcal L}$ on $C_{\overline k}$ such that the endomorphism ring of $\mathcal L$, as an $\mathcal O_{C_{\overline k}}$-subalgebra of the sheaf of rational functions on $C$, is $f_*(\mathcal O_{C'_{\overline k}})$. Note that in the case where $C'=C$ and $f$ is the identity, then $\overline{\mathrm{Pic}}^n_{\mathrm{Id}}(C) = \mathrm{Pic}^n(C)$. More generally, if $f: C' \to C$ is a partial normalisation of the point $p$ where the residue field $\kappa(p)$ is a degree $d$ extension of $k$, then $\overline{\mathrm{Pic}}^n_{f}(C)$ is the image of $\mathrm{Pic}^{n-d}(C')$ under $f_*$.

\begin{lemma}\label{beau3} For every partial normalisation $f:C'
\to C$ the stratum $\overline{\mathrm{Pic}}^n_f(C)$ is equipped with the free action of the quotient group variety $\mathrm{Pic}^0(C)/\mathrm{Ker}(f^*)$ over $k$, where $f^*:\mathrm{Pic}^0(C)\to
\mathrm{Pic}^0(C')$ is the map induced by functoriality.
\end{lemma}
\begin{proof} There is an action
$$\mathrm{Pic}^0(C)\times\overline{\mathrm{Pic}}^n(C)
\to\overline{\mathrm{Pic}}^n(C)$$
of $\mathrm{Pic}^0(C)$ on $\overline{\mathrm{Pic}}^n(C)$, which at the level of represented functors is an action given by the rule $(\mathcal L,\mathcal G)\mapsto\mathcal L\otimes\mathcal G$. This action of $\mathrm{Pic}^0(C)$ leaves the stratification invariant. By Lemma 2.1 of \cite{Be} the stabiliser of any geometric point of the stratum $\overline{\mathrm{Pic}}^n_f(C)$ is $\mathrm{Pic}^0(C)/\mathrm{Ker}(f^*)$. The claim is now clear.
\end{proof}
\begin{propn}\label{beau4} Assume that the geometric genus of $C$ is non-zero. Then the motivic Euler characteristic $\chi^{mot}(\overline{\mathrm{Pic}}^n(C))$ lies in $J$ for any $n$.
\end{propn}
\begin{proof} Using the additivity of $\chi^{mot}$ it is enough to show that $\chi^{mot}(\overline{\mathrm{Pic}}^n_f(C))$ lies in $J$ for every partial normalisation $f:C'\rightarrow C$. Let $\widetilde f:\widetilde C\rightarrow C$ denote the normalization of $C$. There is a short exact sequence of algebraic groups:
$$
0 \to G \xto{\iota}
\mathrm{Pic}^0(C) \xto{\widetilde f^*}
\mathrm{Pic}^0(\widetilde C) \to 1,
$$
where $G$ is a connected linear algebraic group and $\widetilde f^*$ is induced by the pullback with respect to $\widetilde f$. By Lemma \ref{beau3} the algebraic group $\mathrm{Pic}^0(C)/\mathrm{Ker}(f^*)$ acts freely on $\overline{\mathrm{Pic}}^n_f(C)$. Since the genus of $\widetilde C$, which is the geometric genus of $C$, is positive, the dimension of the abelian variety $\mathrm{Pic}^0(\widetilde C)$ is also positive. Moreover $\mathrm{Ker}(f^*)
\subseteq G$ as $\widetilde f$ factors through $f$.  Therefore $\chi^{mot}(\overline{\mathrm{Pic}}^n_f(C))\in J$ by Corollary \ref{fubini1b}.
\end{proof}
\begin{cor}\label{detofcoh}
Assume that the geometric genus of $C$ is at least $1$. Then both $\mathrm{det}_{\ell}(\overline{\mathrm{Pic}}^n(C))$ and $e(\overline{\mathrm{Pic}}^n(C))$ are trivial.
\end{cor}
\begin{proof}
This follows by showing that $\chi_{\ell}(C)=0$, where $\chi_{\ell}$ is the ring homomorphism from Definition $\ref{elemental}$. Since $\chi^{mot}(\overline{\mathrm{Pic}}^n(C))\in J$, this gives the result.
\end{proof}
We now turn to computing the motivic Euler characteristics of the compactified Jacobian of nodal rational curves. 
\begin{lemma}\label{clemb} Let $f:C'\to C$ be a partial normalization of $C$ at the nodal point $p$ of degree $d$. The morphism $f_*:\overline{\mathrm{Pic}}^{n-d}(C')\to\overline{\mathrm{Pic}}^n(C)$
is a closed embedding. 
\end{lemma}
\begin{proof} See Lemma 3.1 of \cite{Be}.
\end{proof}
Let $\overline{\mathrm{Pic}}^n_b(C)$ denote the image of the morphism $f_*$ above, and let $\overline{\mathrm{Pic}}^n_o(C)$ be the complement of $\overline{\mathrm{Pic}}^n_b(C)$ in
$\overline{\mathrm{Pic}}^n(C)$.

\begin{notn} In the remainder of this section we assume that $C$ is semi-stable. Let $\mathcal S(C)$ denote the set of nodes of $C$. For every $p\in\mathcal S(C)$ let $k(p)$ denote the residue field of $p$ and let $\deg(p)$ denote the degree of $k(p)$ over $k$. Now let $p$ be a node of $C$ and pick a $k(p)$-rational point $\widetilde p$ of $C$ whose set theoretic image is exactly $p$. The formal completion of the base change of $C$ to $k(p)$ at $\widetilde p$ is isomorphic to the formal completion of $x^2=\alpha_py^2$ at $(0,0)$ for some
$\alpha_p\in k(p)^\times$. The element $\alpha_p$ is unique up to multiplication by an element of $(k(p)^\times)^2$, so the isomorphism class of the quadratic form $\langle2\alpha_p\rangle$ over $k(p)$ is unique and only depends on $p$. 
\end{notn}

We have the following computation of the compactified Jacobians of normalisations, which serves as an arithmetic generalisation of Proposition 2.8 of \cite{KR1}.
\begin{propn}\label{torusbundle}
Let $f: C' \to C$ be the partial normalisation of $C$ at the nodal point $p$. The pullback with respect to $f$ induces a map
$$
\overline{\mathrm{Pic}}^n_o(C) \to \overline{\mathrm{Pic}}^n(C'),
$$
which exhibits $\overline{\mathrm{Pic}}^n_o$ as a torsor under $\mathrm{Res}_{k(p)/k}(\mathbb{G}^{[\alpha_p]}_{m, k(p)})$.
\end{propn}

\begin{proof}
By naturality we obtain $\overline{\mathrm{Pic}}_o^n(C) \times_k \kbar = \overline{\mathrm{Pic}}_o^n(C_{\kbar})$. Therefore $\overline{\mathrm{Pic}}_o^n(C)$ is a variety, so is determined by $\overline{\mathrm{Pic}}_o^n(C)(\kbar)$ along with its (right) action of $\Gal_k$.

Let $d := [k(p):k]$, and let $q,q' \in C'$ be the preimages of $p$ in $C'$. Let $p_1, \ldots, p_d$ be the points of $C(\kbar)$ that lie over $p$, similarly $q_1, \ldots, q_d$ the points of $C'(\kbar)$ lying over $q$ and $q'_1, \ldots, q'_d$ the points of $C'(\kbar)$ lying over $q'$, so that the map $C'(\kbar) \to C(\kbar)$ sends $q_i$ and $q'_i$ to $p_i$.   We also see that $\{p_1, \ldots, p_d\}$ is isomorphic as a right $\Gal_k$ set to $I_{k(p)}$.

We explicitly construct an action of $\mathrm{Res}_{k(p)/k}(\mathbb{G}^{[\alpha_p]}_{m, k(p)})$ on $\overline{\mathrm{Pic}}^n_o(C)$. That is, we construct a map $\phi: \overline{\mathrm{Pic}}^n_o(C) \times_k \mathrm{Res}_{k(p)/k}(\mathbb{G}^{[\alpha_p]}_{m, k(p)}) \to \overline{\mathrm{Pic}}^n_o(C)$ that fits into the following commutative diagram

\begin{center}
\begin{tikzcd}[cramped]
\overline{\mathrm{Pic}}^n_o(C) \times_k \mathrm{Res}_{k(p)/k}(\mathbb{G}^{[\alpha_p]}_{m, k(p)}) \ar[r, "\phi"] \ar[d, "\mathrm{pr}_1"] & \overline{\mathrm{Pic}}^n_o(C) \ar[d] \\
\overline{\mathrm{Pic}}^n_o(C) \ar[r] & \overline{\mathrm{Pic}}^n(C'),
\end{tikzcd}
\end{center}
where $\mathrm{pr}_1$ is the projection onto the first factor and the unlabelled maps are those from the statement. Our map $\phi$ is therefore given by a $\Gal_k$ equivariant map
$$
\overline{\mathrm{Pic}}_o^n(C)(\kbar) \times \mathrm{Res}_{k(p)/k}(\mathbb{G}^{[\alpha_p]}_{m, k(p)})(\kbar) \to  \overline{\mathrm{Pic}}_o^n(C)(\kbar).
$$
Identify $\mathrm{Res}_{k(p)/k}(\mathbb{G}^{[\alpha_p]}_{m, k(p)})(\kbar)  \cong (\kbar^\times)^{I_{k(p)}}$, where the right-action of $\Gal_k$ on $(\kbar^\times)^{I_{k(p)}}$ is given by 
\begin{align*}
(\lambda_i)_{\iota_i} \cdot \sigma &= (\sigma^{-1}(\lambda_i))_{\sigma^{-1} \circ \iota_i} \text{ if $q_i \cdot \sigma = q_j$, and}\\
 (\lambda_i)_{\iota_i} \cdot \sigma &= (\sigma^{-1}(\lambda_i)^{-1})_{\sigma^{-1} \circ \iota_i} \text{ if $q_i \cdot \sigma = q_j'$.}
\end{align*}

Let $\mathcal{F} \in \overline{\mathrm{Pic}}_o^n(C)(\kbar)$, and let $\mathcal{G}$ be the image of $\mathcal{F}$ in $\overline{\mathrm{Pic}}^n(C')(\kbar)$. The data of $\mathcal{F}$ is given by $\mathcal{G}$, along with isomorphisms $\psi_i: \mathcal{G}_{q_i} \cong \mathcal{G}_{q'_i}$ for every $i$. By the proof of Lemma 2.5 of \cite{KR1}, we can canonically identify $\mathrm{Isom}(\mathcal{G}_{q_i}, \mathcal{G}_{q_i'}) \cong \kbar^\times$, so that $\psi_i, \psi_j \in \kbar^\times$. There is a canonical $\Gal_k$ action on $\prod_i \mathrm{Isom}(\mathcal{G}_{q_i}, \mathcal{G}_{q'_i})$, which we can describe.

Let $\sigma \in \Gal_k$ be such that $p_i \cdot \sigma = p_j$, and let $\psi_i \in \mathrm{Isom}(\mathcal{G}_{q_i}, \mathcal{G}_{q'_i})$, be such that $\psi_i \cong \lambda_i \in \overline{k}^\times$. Then $\psi_i \cdot \sigma = \sigma^{-1}(\lambda_i) \in \mathrm{Isom}(\mathcal{G}_{q_i \cdot \sigma}, \mathcal{G}_{q'_i \cdot \sigma})$. There are now 2 cases.

In the case when $q_i \cdot \sigma = q_j$, we have $ \mathrm{Isom}(\mathcal{G}_{q_i \cdot \sigma}, \mathcal{G}_{q'_i \cdot \sigma}) = \mathrm{Isom}(\mathcal{G}_{q_j}, \mathcal{G}_{q'_j})$. 

In the case when $q_i \cdot \sigma = q'_j$, we have $ \mathrm{Isom}(\mathcal{G}_{q_i \cdot \sigma}, \mathcal{G}_{q'_i \cdot \sigma}) = \mathrm{Isom}(\mathcal{G}_{q'_j}, \mathcal{G}_{q_j})$. We can canonically identify this with $\mathrm{Isom}(\mathcal{G}_{q_j}, \mathcal{G}_{q'_j})$, under the isomorphism $\lambda \mapsto \lambda^{-1}$.

Putting this together, we see there is an isomorphism 
$$
\prod_i \mathrm{Isom}(\mathcal{G}_{q_i}, \mathcal{G}_{q'_i}) \cong \mathrm{Res}_{k(p)/k}(\mathbb{G}^{[\alpha_p]}_{m, k(p)})(\kbar)$$ which is compatible with the $\Gal_k$ actions on each side. Explicitly, if we identify $$\mathrm{Res}_{k(p)/k}(\mathbb{G}^{[\alpha_p]}_{m, k(p)})(\kbar)  \cong (\kbar^\times)^{I_{k(p)}},$$ and each $\mathrm{Isom}(\mathcal{G}_{q_i}, \mathcal{G}_{q'_i}) \cong \kbar^\times$, this isomorphism is given by $(\lambda_i)_{i} \mapsto (\lambda_i)_{\iota_i}$ and the action of $\Gal_k$ is given by $(\lambda_i)_i \cdot \sigma = (\sigma^{-1}(\lambda_i))_{\sigma^{-1} \circ \iota_i}$ if $q_i \cdot \sigma = q_j$ for some $j$, and $(\lambda_i)_i \cdot \sigma = (\sigma^{-1}(\lambda_i)^{-1})_{\sigma^{-1} \circ \iota_i}$ if $q_i \cdot \sigma = q'_j$ for some $j$, which is exactly saying that $\prod_i \mathrm{Isom}(\mathcal{G}_{q_j}, \mathcal{G}_{q'_j}) \cong (\kbar^\times)^{I_{k(p)}}$. Putting this together, we can write every $\mathcal{F} \in \overline{\mathrm{Pic}}_o^n(C)(\kbar)$ as a pair $(\mathcal{G}, (\psi_i))$, where $\mathcal{G} \in \overline{\mathrm{Pic}}^n(C')(\kbar)$ and $(\psi_i) \in (\kbar^\times)^{I_{k(p)}}$.

Define the morphism
$$
\phi: \overline{\mathrm{Pic}}_o^n(C)(\kbar) \times \mathrm{Res}_{k(p)/k}(\mathbb{G}^{[\alpha_p]}_{m, k(p)})(\kbar) \to  \overline{\mathrm{Pic}}_o^n(C)(\kbar)
$$
as follows. Let $(\mathcal{G}, \{ \psi_i \}) \in \overline{\mathrm{Pic}}_o^n(C)(\kbar) $, and let $\{\lambda_i\} \in \mathrm{Res}_{k(p)/k}(\mathbb{G}^{[\alpha_p]}_{m, k(p)})(\kbar)$. Define 
$$
\phi ( (\mathcal{G}, \{\psi_i\}), \{\lambda_i\}) := (\mathcal{G}, \{ \lambda_i \psi_i \}).
$$
We can check that $\phi$ is Galois equivariant, so gives us a map of the underlying $k$ varieties, i.e.~a $\mathrm{Res}_{k(p)/k}(\mathbb{G}^{[\alpha_p]}_{m, k(p)})$ action on $\overline{\mathrm{Pic}}_o^n(C)$ over $\overline{\mathrm{Pic}}^n(C')$. This action is Galois equivariant and acts freely and transitively on the fibres, so gives the result.
\end{proof}
\begin{cor}
Let $f: C' \to C$ be a partial normalisation of $C$ at the nodal point $p$ of degree $d$. Then
$$
\chi^{mot}(\overline{\mathrm{Pic}}^n(C)) \equiv \overline{\mathrm{Pic}}^{n-d}(C') +  \left(N_{k(p)/k}(\langle -2 \rangle - \langle 2\alpha_p \rangle)\right) \chi^{mot} (\overline{\mathrm{Pic}}^{n}(C')) \pmod{J}.
$$
\end{cor}
\begin{proof}
Lemma $\ref{clemb}$ tells us that 
$$
\chi^{mot}(\overline{\mathrm{Pic}}^n(C))  =  \chi^{mot} (\overline{\mathrm{Pic}}^{n-d}(C')) + \chi^{mot}(\overline{\mathrm{Pic}}_o^n(C)).
$$
The proposition above, along with Theorem $\ref{fubini1}$ allows us to compute
$$
\chi^{mot}(\overline{\mathrm{Pic}}_o^n(C)) \equiv \chi^{mot}( \mathrm{Res}_{k(p)/k}(\mathbb{G}_{m,k(p)}^{[\alpha_p]})) \chi^{mot}( \overline{\mathrm{Pic}}^{n}(C')) \pmod{J}.$$
 Finally Theorem $\ref{WRs}$ and Lemma $\ref{1dtorus}$ allow us to see 
$$
\chi^{mot}( \mathrm{Res}_{k(p)/k}(\mathbb{G}_{m,k(p)}^{[\alpha_p]})) =  N_{k(p)/k}(\langle -2 \rangle - \langle 2\alpha_p \rangle),
$$
as required.
\end{proof}

\begin{notn}
The above corollary motivates the following definition. Let $C$ be a nodal curve defined over $k$, and let $p \in \mathcal{S}(C)$. Define 
$$
\Phi_p :=  1 + N_{k(p)/k}(\langle -2 \rangle - \langle 2\alpha_{p} \rangle) \in \widehat{\mathrm{W}}(k).
$$
\end{notn}

\begin{cor}\label{zero} Assume that $C$ has geometric genus $0$ and arithmetic genus $g$. Then for any integer $n$, there is an equality in $\widehat{\mathrm{W}}(k)/J$:
$$\chi^{mot}(\overline{\mathrm{Pic}}^n(C))=\prod_{p\in\mathcal S(C)}\Phi_p.$$
\end{cor}
\begin{proof} We prove the claim by induction on the cardinality of $\mathcal S(C)$. When $\mathcal S(C)$ is empty then $C$ is smooth and genus $0$ and hence $\overline{\mathrm{Pic}}^n(C)$ is isomorphic to $\mathrm{Spec}(k)$ for every $n$. Therefore the claim holds in this case, since the product over an empty index is $1$. Assume that the claim holds for every nodal curve $C'$ of geometric genus zero such that $|\mathcal S(C')|<|\mathcal S(C)|$.  Let $p' \in \mathcal S(C)$ be a point of degree $d$, and let $C'$ be the normalisation of $C$ at $p'$, so that $\mathcal{S}(C) = \mathcal{S}(C') \cup \{p'\}$. Then by the previous corollary
$$
\chi^{mot}(\overline{\mathrm{Pic}}^n(C)) \equiv  \chi^{mot} (\overline{\mathrm{Pic}}^{n-d}(C')) + (N_{k(p')/k}(\langle -2 \rangle - \langle 2\alpha_{p'} \rangle)\chi^{mot} (\overline{\mathrm{Pic}}^{n}(C)) \pmod{J}.
$$
The induction hypothesis gives 
$$
\chi^{mot} (\overline{\mathrm{Pic}}^{n-d}(C')) \equiv \chi^{mot} (\overline{\mathrm{Pic}}^{n}(C')) \equiv \prod_{p \in \mathcal{S}(C')} \Phi_p \pmod{J}.
$$
Therefore
\begin{align*}
\chi^{mot}(\overline{\mathrm{Pic}}^n(C)) &\equiv \left(\prod_{p \in \mathcal{S}(C')} \Phi_p\right) (1 + (N_{k(p)/k}(\langle -2 \rangle - \langle 2\alpha_{p} \rangle)) \pmod{J} \\
&\equiv \left(\prod_{p \in \mathcal{S}(C')} \Phi_p \right) \cdot \Phi_{p'} \pmod{J}\\
&\equiv \prod_{p \in \mathcal{S}(C)} \Phi_p \pmod{J},
\end{align*}
 as required.
\end{proof}
\begin{cor}\label{complexpicard}
Suppose $k=\mathbb{C}$ in the previous corollary. Then $e(\overline{\mathrm{Pic}}^g(C))=1$.
\end{cor}
\begin{proof}
This follows immediately, as $\mathrm{rank}(\Phi_p)=1$ for every $p$.
\end{proof}
The above recovers Corollary 3.4 of \cite{Be}. Another corollary is the real closed case, which generalises Proposition 1.8 of \cite{KR1} to all real closed fields. Suppose now that $k$ is a real closed field, and identify $k^\times/k^{\times2} \cong \{\pm 1\}$.

\begin{cor}\label{realpicard}
There is an equality $\mathrm{sign}( \chi^{mot}(\overline{\mathrm{Pic}}^g(C))) = (-1)^r$, where $r$ is the number of points $p \in \mathcal{S}(C)$ such that $k(p)=k$ and $\alpha_p= 1$, i.e.~the number of singular points that are both defined over $k$ and their tangent spaces are defined over $k$.
\end{cor}
\begin{proof}
Let $p \in \mathcal{S}(C)$ be one of the singular points. If $k(p)=\overline{k}$, then we see that the corresponding factor at $p$ is given by $\mathbb{H}- (\langle 2 \rangle + \langle 2 \alpha_p \rangle )$, which is zero. Therefore $\mathrm{sign}(\Phi_p)=1$. If $k(p)=k$, then the singular point $p$ is a $k$ point. If $\alpha_p=-1$, then the tangent spaces at $p$ are not defined over $k$, and $\mathbb{H} - (\langle 2 \rangle + \langle 2 \alpha_p \rangle )=0$, so $\mathrm{sign}(\Phi_p)=1$.  If $\alpha_p=1$, then $\mathbb{H} - (\langle 2 \rangle + \langle 2 \alpha_p \rangle ) = \langle -1 \rangle - \langle 1 \rangle$, and we obtain 
$$
N_{k/k}(\langle -1 \rangle - \langle 1 \rangle) = \langle -1 \rangle - \langle 1 \rangle,$$
which means $\Phi_p = \langle -1 \rangle$, so has sign $-1$. Putting this together gives us
$$
\chi^{mot}(\overline{\mathrm{Pic}}^g(C)) = (-1)^r,
$$
where $r$ is the number of singular points on $C$ such that $p$ and the tangent spaces to $p$ are both defined over $k$. 
\end{proof}

\begin{cor}\label{discpicard}
Let $k$ be any field of characteristic $0$ . Then 
$$
\mathrm{disc}(\chi^{mot}(\overline{\mathrm{Pic}}^g(C))) = (-1)^g\prod_{p \in \mathcal{S}(C)}N_{k(p)/k}(\alpha_p) \in k^\times/k^{\times2}.
$$
\end{cor}
\begin{proof}
We obtain this by applying $\mathrm{disc}$ to Corollary $\ref{zero}$. By definition,
$$
\mathrm{disc}(\Phi_p) = \mathrm{disc}(1- N_{k(p)/k}(\mathbb{H}-(\langle 2\rangle + \langle 2 \alpha_p \rangle))= \mathrm{disc}(N_{k(p)/k}(\mathbb{H}-(\langle 2\rangle + \langle 2 \alpha_p \rangle)).
$$
Note that $\mathrm{disc}(\mathbb{H}-(\langle 2\rangle + \langle 2 \alpha_p \rangle)) = -\alpha_p$, so Corollary $\ref{normdisc}$ gives us
$$\mathrm{disc}( N_{k(p)/k}(\mathbb{H}-(\langle 2\rangle + \langle 2 \alpha_p \rangle)) = (-1)^{[k(p):k]} N_{k(p)/k}(\alpha_p).$$

Multiplying all the $\Phi_p$ terms gives us
$$
\mathrm{disc}(\prod_{p \in \mathcal{S}(C)} \Phi_p) = \prod_{p \in \mathcal{S}(C)}(-1)^{[k(p):k]}N_{k(p)/k}(\alpha_p) = (-1)^g  \prod_{p \in \mathcal{S}(C)}N_{k(p)/k}(\alpha_p)
$$
as required. 
\end{proof}

\subsection{Motivic Euler characteristics of the compactified Jacobian of the linear system}
Let $X$ be a polarised $K3$ surface equipped with a degree $2g$ polarisation and let $\mathcal{C}$ be the associated $g$-dimensional linear system of genus $g$ curves on $X$, as in the introduction. Let $\pi: \mathcal{C} \to \mathbb{P}^g$ be the associated morphism given by the linear system. We now combine above results with our Fubini theorem to compute $\chi^{mot}(\overline{\mathrm{Pic}}^g(\mathcal{C}))$.

\begin{defn} For $q$ a closed point of $\mathbb P^g_k$ write $k(q)$ to mean the residue field at $q$. As a slight abuse of notation, we also write $q: \Spec(k(q)) \to \mathbb P^g_k$ to be the corresponding closed immersion, and $\mathcal C_q$ to mean the fibre of $\mathcal{C} \xto{\pi} \mathbb P^g_k$ over $q$. Assume that the curves in $\mathcal{C}$ belong to a primitive divisor class and are geometrically integral, which is will hold if $g\geq 2$ and $X$ is of Picard rank $1$ and the curves in the linear system belong to a primitive divisor class. Assume that the rational curves in $\mathcal{C}$ are all nodal, which is generically the case by Chen's theorem, Theorem 1.1 of \cite{Ch}. Let $G(X)$ be the set of closed points $q$ of $\mathbb P^g_k$ such that $C_q$ has geometric genus zero. We see that $G(X)$ is finite, since the set of geometric points lying above $G(X)$ is finite by the argument presented in Corollary 2.3 of \cite{Be}, which holds over any algebraically closed field of characteristic $0$. By slight abuse of notation let $G(X)$ also denote the reduced closed subscheme of $\mathbb P^g_k$ whose underlying set of closed points is $G(X)$. Let $\pi_g:\overline{\mathrm{Pic}}^g(\mathcal C)\to\mathbb P^g_k$ be the relative compactified Picard variety. Let $Y := \mathbb P^g_k \setminus G(X)$ and define $W := \pi_g^{-1}(Y)$. Write $\pi_{g,W}: W \to Y$ for the restriction of $\pi_g$ to $W$. 
\end{defn}

\begin{thm}\label{bettibound}
There exists an integer $\Delta(g)$ such that for any $K3$ surface as above
$$
\mathrm{dim}_{\Q_\ell} H^i( \overline{\mathrm{Pic}}^g(C_q)_{\kbar}, \Q_\ell) \leq \Delta(g)
$$
 for all closed points $q$ of $\mathbb{P}^g$ and for all $i$.
\end{thm}
\begin{proof}
Since this statement is only about the dimension of the cohomology groups, it is enough to prove this statement after base changing to the algebraic closure of the field, so that $k=\kbar$. By Chapter 5, \S1 of \cite{Huy}, let $\mathcal{M}_g$ denote the moduli stack of $K3$ surfaces equipped with a polarisation of degree $2g$ (note that the construction in \cite{Huy} makes sense over any algebraically closed field of characteristic $0$). This is a Deligne--Mumford stack by Proposition 4.10 of Chapter $5$ of \cite{Huy}, so there exists a $\kbar$ variety $\mathbf{M}_g$, equipped with an étale surjective morphism $\mathbf{M}_g \to \mathcal{M}_g$. Since $\mathcal{M}_g$ is the moduli stack representing polarised $K3$ surfaces over $\kbar$, pulling back the universal family along this morphism equips $\mathbf{M}_g$ with a flat family $\mathcal{X} \to \mathbf{M}_g$ of polarised $K3$ surfaces. Moreover since $\mathbf{M}_g \to \mathcal{M}_g$ is surjective, every $K3$ surface equipped with a polarisation of degree $2g$ appears as a fibre over some point in the family $\mathcal{X} \to \mathbf{M}_g$. 

Consider $\mathbf{L} \in \mathrm{Pic}_{\mathcal{X}/\mathbf{M}_g}(\mathbf{M}_g)$. This determines a linear system of divisors $\mathbf{C}_g$ on $\mathcal{X}$ regarded as a $\mathbf{M}_g$-variety, which in turn comes with a map $\mathbf{C}_g \to \mathbb{P}^g_{\mathbf{M}_g} \cong \mathbb{P}^g_{\kbar} \times_{\kbar} \mathbf{M}_g$. Let $\mathbf{q} \in \mathbb{P}^g_{\mathbf{M}_g}(\kbar)$, so that $\mathbf{q}$ is given by a pair $(\alpha, q)$, where $\alpha \in \mathbf{M}_g(\kbar)$ and $q \in \mathbb{P}^g(\kbar)$. Then the fibre of $\mathbf{C}_g \to \mathbb{P}^g_{\mathbf{M}_g}$ over $\mathbf{q}$ is given by the fibre of $\mathcal{C}_{\alpha} \to \mathbb{P}^g$ over $q$, so in particular is an integral curve of arithmetic genus $g$. Therefore Theorem 8.1 of \cite{AK} allows us to form the relative compactified Jacobian $\Pi_g: \overline{\mathrm{Pic}}^g(\mathbf{C}_g) \to \mathbb{P}^g_{\mathbf{M}_g}$. The fibres of this morphism have dimension $2g$, since the fibres are the relative compactified Jacobians of the curves as above.

Consider the map $\Pi_g$, which is proper, and let $i$ be a positive integer. Since the $\ell$-adic sheaf $\Q_\ell$ is locally constant, the derived pushforward $R^i \Pi_{g,*}\Q_\ell$ is constructible by \S XIV, Theorem 1 of \cite{SGA}. If $i \leq 2g$, then since $R^i \Pi_{g,*} \Q_\ell$ is constructible, there exists a stratification $U^i_1, \ldots, U^i_n$ of $\mathbb{P}^g_{\mathbf{M}_g}$ such that $R^i\Pi_{g,*}\Q_\ell$ is locally constant on $U^i_j$. Let $\Delta^{i,j}(g)$ denote the rank of $R^i\Pi_{g,*}\Q_\ell$ on $U^i_j$, and let $\Delta^i(g) := \mathrm{max}_j \Delta^{i,j}(g)$, which is finite since the stratification is finite. If $i>2g$, then $R^i\Pi_{g,*}\Q_\ell = 0$, so set $\Delta^i(g) = 0$. Let $\Delta(g) := \mathrm{max}_i \Delta^i(g)$. Precisely by construction if $f: U \to \mathbb{P}^g_{\mathbf{M}_g}$ is such that $f^*R^i\Pi_{g,*}\Q_\ell$ is locally constant, then $\mathrm{rank}(f^*R^i\Pi_{g,*}\Q_\ell)\leq \Delta(g)$. 

Now suppose $X$ is our polarised $K3$ surface over $\kbar$. Let $\mathcal{C} \to \mathbb{P}^g$ denote the linear system on $X$, let $q \in \mathbb{P}^g(\kbar)$, let $C_q$ denote the curve over $q$. We claim that $\mathrm{dim}_{\Q_\ell} H^i(\overline{\mathrm{Pic}}^g(C_q), \Q_\ell) \leq \Delta(g)$. 

By the definition of the moduli stack $X$ determines a point $[X] \in \mathcal{M}_g(\kbar)$. Pick $\alpha \in \mathbf{M}_g(\kbar)$ lying over $[X]$ under the morphism $\mathbf{M}_g(\kbar) \to \mathcal{M}_g(\kbar)$, which exists since $\mathbf{M}_g \to \mathcal{M}_g$ is étale and surjective. Then the fibre of $\overline{\mathrm{Pic}}^g(\mathbf{C}_g) \to  \mathbb{P}^g_{\mathbf{M_g}}$ over the point $(\alpha, q)$ is isomorphic to $\overline{\mathrm{Pic}}^g(C_q)$ by construction.

Fix an integer $i$ and consider $H^i(\overline{\mathrm{Pic}}^g(C_q), \Q_\ell)$. This vanishes for $i > 2g$ so assume $i \leq 2g$. Consider $(\alpha, q) \in \mathbb{P}^g_{\mathbf{M}_g}(\kbar)$. By construction, the fibre of $\Pi_g$ over $(\alpha, q)$ is $\overline{\mathrm{Pic}}^g(C_q)$. Therefore by the proper base change theorem (see \S XII, Theorem 5.1(iii) of \cite{SGA4}), we see 
$$
\mathrm{dim}_{\Q_\ell}( H^i(\overline{\mathrm{Pic}}^g(C_q), \Q_\ell)) = \mathrm{rank} ( (\alpha, q)^* R^i\Pi_{g,*}\Q_\ell) \leq \Delta(g),
$$
as required.
\end{proof}
\begin{rem}
In the case that $C_q$ is a nodal curve, we may use Proposition $\ref{torusbundle}$ and the Leray--Serre spectral sequence to see that $ \mathrm{dim}_{\Q_\ell}( H^i(\overline{\mathrm{Pic}}^g(C_q), \Q_\ell)) \leq {2g \choose g}$, and if $C_q$ is smooth, $H^i(\overline{\mathrm{Pic}}^g(C_q), \Q_\ell) = {2i \choose i}$. Therefore $\Delta(g) \geq {2g \choose g}$. 
\end{rem}

\begin{defn}\label{jgdefinition}
Let $\Delta(g)$ be the constant from Theorem $\ref{bettibound}$. For $k^{const}$ a number field, let $\Delta^{k^{const}}_g$ be the finite subgroup of $\Hom(\Gal_{\Q}, \Z/2\Z)$ consisting of characters $\chi$ such that $\chi$ is unramified at all primes $\mathfrak{p}$ whose residue characteristic $p$ satisfies $p> \Delta(g) \cdot [k^{const}:\Q]+1$. For $k$ a finitely generated field over $\Q$, let $\Delta_g^k$ be the finite subgroup of $\Hom(\Gal_k, \Z/2\Z)$ consisting of elements in the image of $\Delta^{k^{const}}_g$ under the map $\Hom(\Gal_{k^{const}}, \Z/2\Z) \hookrightarrow \Hom(\Gal_k, \Z/2\Z)$. We will also think of $\Delta_g^k$ as a subgroup of $\Hom(\Gal_k, \Q_{\ell}^\times)$ using the inclusion $\Z/2\Z \hookrightarrow \Q_{\ell}^\times$.  

The map $\widehat{\mathrm{W}}(k) \to E( \Hom(\Gal_k, (\Z/2\Z)/ \Delta^k_g)$ given by $[q] \mapsto (\mathrm{rank}(q), \mathrm{disc}(q)) \pmod{\Delta^k_g}$ is a ring homomorphism with kernel $I_g$ containing $I^2$. Let $J_g := I_g \cap T$. The homomorphisms $\mathrm{rank}$ and $\mathrm{sign}$ clearly both factor through $\widehat{\mathrm{W}}(k)/J_g$, so we can still recover these invariants after passing to this quotient.
\end{defn}

\begin{rem}
Let $k_g$ be the abelian extension of $k$ corresponding to the fixed field of the subgroup $\bigcap_{\chi \in \Delta^k_g} \mathrm{ker}(\chi)$. The canonical map $\widehat{\mathrm{W}}(k)/J \to \widehat{\mathrm{W}}(k_g)/J$ factors through $J_g$. We can therefore view $J_g$ as recovering a discriminant in $k_g$. However, we may lose information about sign by passing to $\widehat{\mathrm{W}}(k_g)/J$. For example, if $-1$ is not a square in $k$, and if the character corresponding to $k(\sqrt{-1})$ lies in $\Delta^k_g$, then $\mathrm{Spr}(k_g) = \emptyset$, but this may not be the case for $\mathrm{Spr}(k)$. We however do not lose information about sign in the case that $k$ is a real closed field.
\end{rem}

\begin{lemma}\label{chimotW}
We have $\chi^{mot}(W) \equiv 0 \pmod{J_g}$. 
\end{lemma}
\begin{proof}
This proof is an application of Corollary $\ref{saitofactorsurface}$ to this situation. We first show that this situation satisfies the properties for this result.

Let $i\geq0$, and consider $R^i\pi_{g,*} \Q_{\ell}$. Since $\pi_g$ is proper, Theorem 1 of \S XIV of \cite{SGA4} gives us that $R^i\pi_{g,*}\Q_\ell$ is constructible, so there exists a a stratification $Y^i_1, \ldots, Y^i_{n_i}$ of $Y$ into locally closed subvarieties such that $R^i\pi_*\Q_{\ell}$ is locally constant on $Y^i_j$ for all $j$. Moreover since $\pi_g$ has relative dimension $2g$, this is the trivial sheaf whenever $i > 2g$. Since the restriction of a locally constant sheaf is locally constant and each stratification is finite, we may find a common stratification $Y_1, \ldots, Y_n$ such that $R^i\pi_{g,W,*}\Q_\ell$ is locally constant on $Y_j$ for all $i$ and for all $j$. Define a partition of $W$ by setting $W_j := \pi_{g,W}^{-1}(Y_j)$, so $\chi^{mot}(W) = \sum_j \chi^{mot}(W_j)$. It is therefore enough to show that $\chi^{mot}(W_j)\equiv 0 \pmod{J_g}$ for each $W_j$.

By construction, for every $i$, $R^i\pi_{g, W_j,*}\Q_\ell$ is locally constant on each $Y_j$, so they are smooth. We would like to apply Corollary $\ref{saitofactorsurface}$ to each morphism $W_j \to Y_j$. In order for this to do this, we first need $Y_j$ to have a compactification $\overline{Y_j}$ such that $\overline{Y_j} \setminus Y_j$ is a divisor with simple normal crossings and the ramification of $R^i\pi_{g,W_j,*}\Q_\ell$ on $\overline{Y_j} \setminus Y_j$ is tame. This follows immediately by taking $\overline{Y_j}'$ to be the closure of $Y_j$ in $\mathbb{P}^g$, and blowing up $\overline{Y_j}'$ at $\overline{Y_j}' \setminus Y_j$ in order to obtain that the boundary is a divisor.

The fibres of $\chi^{mot}(W_j)$ are given by $\overline{\mathrm{Pic}}^g(C)$ for $C$ a non-rational curve, and so have motivic Euler characteristic lying in $J$ by Proposition $\ref{beau4}$. Moreover for all $i$, $\mathrm{rank}(R^i\pi_{g,W_j,*}\Q_{\ell}) \leq \Delta(g)$ by Theorem $\ref{bettibound}$. Corollary $\ref{saitofactorsurface}$ then gives that $\chi^{mot}(W_j)$ has rank $0$, signature $0$, and discriminant given by $\prod_{i=0}^{2g}J_{D,R^if_*\mathbb{Q}_{\ell}}^{\otimes(-1)^{i+1}}$. Corollary $\ref{restrictedramification}$ gives us the required bound on the ramification of this character, and so we see $\chi^{mot}(W) \in J_g$.
\end{proof}

\begin{defn}\label{BmotC}
Define the following element of $\widehat{\mathrm{W}}(k)$
$$\mathbf B^{mot}_{\mathcal{C}}(X):= \sum_{q \in G(X)} \mathrm{Tr}_{k(q)/k} \left( \prod_{p\in\mathcal S(C_q)} \Phi_p \right) \in \widehat{\mathrm{W}}(k).$$
\end{defn}

This element makes an appearance in our final arithmetic Yau--Zaslow formula. Note that $\mathrm{rank}(\Phi_p)=1$, so $\mathrm{rank}(\mathbf B^{mot}_{\mathcal{C}}(X)) = \# G(X)(\kbar)$. In particular, we see $\mathbf B^{mot}_{\mathcal{C}}(X)$ ``counts" the number of rational curves in $\mathcal{C}$, as well as providing arithmetic information about nodes of the rational curves.
\begin{lemma}\label{9.2}
We have an equality modulo $J_g$
$$
\mathbf{B}_{\mathcal{C}}^{mot}(X) \equiv  \chi^{mot}(\overline{\mathrm{Pic}}^g(\mathcal{C})) \pmod{J_g}.
$$
\end{lemma}
\begin{proof}
Note that $\overline{\mathrm{Pic}}^g(\mathcal C) = W \amalg \pi_g^{-1}(G(X))$, so by the additivity of $\chi^{mot}$, we get
$$
\chi^{mot}(\overline{\mathrm{Pic}}^g(\mathcal C)) = \chi^{mot}_k(W) + \chi^{mot}_k( \pi_g^{-1}(G(X))),
$$
and Lemma $\ref{chimotW}$ gives $\chi^{mot}_k(W)=0 \pmod {J_g}$. By additivity of the motivic Euler characteristic
$$
\chi^{mot}( \pi_g^{-1}(G(X))) = \sum_{q \in G(X)} \chi^{mot}_k( \overline{\mathrm{Pic}}^g(C_q) ).
$$
Note that $C_q \to \Spec(k(q))$ is a nodal curve of geometric genus $0$ curve over $k(q)$, so applying Corollary $\ref{zero}$ gives
$$\chi^{mot}_{k(q)}(\overline{\mathrm{Pic}}^g(C_q))=\prod_{p\in\mathcal S(C_q)}\Phi_p \in \widehat{\mathrm{W}}(k(q))/J.$$
Applying Lemma $\ref{trace}$ gives 
$$
\chi^{mot}_k(\overline{\mathrm{Pic}}^g(C_q)) = \mathrm{Tr}_{k(q)/k}\left(\prod_{p\in\mathcal S(C_q)}\Phi_p\right) \in \widehat{\mathrm{W}}(k)/J,
$$
so summing over $q \in G(X)$ gives the result as required.
\end{proof}

\section{The motivic Euler characteristic of Calabi-Yau varieties}
 Recall that $X$ is a K3 surface admitting a $g$-dimensional linear system $\mathcal C$ of geometrically integral curves of genus $g$ and that the curves in this linear system correspond to a primitive class in the N\'eron--Severi group. The assumption that the curves in $\mathcal{C}$ are geometrically integral allows us to apply Example 0.5 of \cite{Mu} to show that the relative compactified Picard varieties $\overline{\mathrm{Pic}}^g(\mathcal C)$ can be identified with connected components of the moduli space of simple sheaves on $X$. In particular $\overline{\mathrm{Pic}}^g(\mathcal C)$ is nonsingular and has trivial canonical bundle, by Theorem 0.1 of \cite{Mu}. Note that if $X$ has Picard rank $1$ and $g \geq 2$, the assumption that the curves in $\mathcal C$ are geometrically integral is automatically satisfied and $\mathcal C$ will be unique with these properties. 

\begin{propn}[(Beauville, Proposition 1.3 of \cite{Be})] \label{beau1} There exists a birational morphism from relative compactified Picard variety $\overline{\mathrm{Pic}}^g(\mathcal C)$ to $X^{[g]}$.
\end{propn}
\begin{rem} In \cite{Be} the result is stated over $\mathbb C$, however the construction holds over any field. 
\end{rem}
In \cite{Ba}, it is shown that birational Calabi--Yau varieties have the same Betti numbers, so also the same complex Euler characteristic. This section is dedicated to proving the required analogue of this theorem for the motivic Euler characteristic of the Calabi-Yau varieties we are considering.

\begin{thm}\label{batyrev}
In the situation as above, there is an equality of the motivic Euler characteristics modulo $J$: 
$$
\chi^{mot}( \overline{\mathrm{Pic}}^g\mathcal{C}) \equiv  \chi^{mot}(X^{[g]}) \pmod{J}.
$$ 
\end{thm}
It is natural to conjecture that the above equality holds for the whole motivic Euler characteristic, rather than just modulo $J$. Note that if this theorem holds for the whole motivic Euler characteristic, then Theorem $\ref{levine1}$, Theorem $\ref{levine2}$ and Theorem \ref{etale1} would imply Lemmas $\ref{BarKon1}$, $\ref{BarKon2}$ and $\ref{BatKon3}$ respectively. We begin with a lemma that allows us to use the result of Batyrev and Kontsevich in our context.
\begin{lemma}\label{BarKon1}
Let $X$ and $Y$ be two birational Calabi--Yau varieties over $k$. Then
$$
\mathrm{rank}(\chi^{mot}(X)) = \mathrm{rank}(\chi^{mot}(Y)).
$$
\end{lemma}
\begin{proof}
Over $\mathbb{C}$, the motivic Euler characteristic $\chi^{mot}$ is just the compactly supported topological Euler characteristic. The result follows by Theorem 1.1 of \cite{Ba}, which says that two birational Calabi--Yau varieties over $\mathbb{C}$ have the same Betti numbers. Suppose now $k \neq \mathbb{C}$. As in Theorem $\ref{levine1}$, there exists a finitely generated field $L$ such that $X$ is defined over $L$, and therefore by enlarging $L$, the varieties $X$ and $Y$ are both defined over some $L$. Since $L$ has finite transcendence degree over $\Q$, we can embed $L$ into $\mathbb{C}$. The result holds for $X_{\mathbb{C}}$, i.e.~$\mathrm{rank}(\chi^{mot}(X)) = \mathrm{rank}(\chi^{mot}(Y))$. Since the motivic Euler characteristic commutes with base change and base changing preserves the rank, this gives the result.
\end{proof}
We wish to now prove a lemma that allows us to use Proposition 3.1 of \cite{KR1} over any real closed field. 

\begin{lemma}\label{BarKon2}
For every real closure $k_<$ of $k$, we have an equality of real closed Euler characteristics
$$
\chi^{rc}\left( (\overline{\mathrm{Pic}}^g\mathcal{C})_{k_<}(k_<)\right) = \chi^{rc}\left( (X^{[g]})_{k_<}(k_<)\right).
$$
\end{lemma}
\begin{proof}
Firstly, since $\overline{\mathrm{Pic}}^g(-)$ and $(-)^{[g]}$ are representable functors, they commute with base change, so $\overline{\mathrm{Pic}}^g(\mathcal{C}_{k_<}) \cong (\overline{\mathrm{Pic}}^g\mathcal{C})_{k_<}$, and similarly $(X^{[g]})_{k_<} \cong (X_{k_<})^{[g]}$. It therefore suffices to show this statement in the case that $k$ is a real closed field. For $k=\mathbb{R}$, the first equality follows by Proposition 3.1 of \cite{KR1}, which says our birational map $\phi: \overline{\mathrm{Pic}^g\mathcal{C}}(\mathbb{R}) \dashrightarrow X^{[g]}(\mathbb{R})$ defined in Proposition 1.3 of \cite{Be} gives rise to a diffeomorphism $\overline{\phi}$ between these two real manifolds, which in turn tells us that they have the same Euler characteristic. We claim that we can choose $\overline{\phi}$ to be semi-algebraic, so that $\overline{\phi}$ is a Nash diffeomorphism in the sense of Definition 2.9.9 of \cite{BCR}. 

 We first claim that if $Y/\mathbb{R}$ is a smooth projective variety, then there exists a Nash diffeomorphism from $Y(\mathbb{R})$ to a closed, bounded Nash submanifold of $\mathbb{R}^N$ for some $N$. It is enough to show this in the special case when $X$ is a projective space. In this case this is Theorem 3.4.4 on page 72 of \cite{BCR}. In particular $\overline{\mathrm{Pic}^g\mathcal{C}}(\mathbb{R})$ (resp. $X^{[g]}(\mathbb{R})$) is Nash diffeomorphic to closed, bounded Nash submanifolds of $\mathbb{R}^N$, which we will call $M$ (resp. $M')$. By Theorem 2 of \cite{Sh}, since $M,M'$ are closed bounded Nash submanifolds of $\mathbb{R}^N$, there is a diffeomorphism between them if and only if there is a Nash diffeomorphism between them by Theorem 2 of \cite{Sh}. In particular, since $\overline{\mathrm{Pic}^g\mathcal{C}}(\mathbb{R})$ and $X^{[g]}$ are diffeomorphic, $M$ and $M'$ are diffeomorphic, so are Nash diffeomorphic. This in turn gives that $\overline{\mathrm{Pic}^g\mathcal{C}}(\mathbb{R})$ and $X^{[g]}$ are Nash diffeomorphic, which proves the claim.

 To prove the statement in general, let $k$ be a general real closed field. Note that if $K$ is a real closed field containing $k$ and $X/k$ is a variety, then the motivic Euler characteristic of the base change $\chi^{mot}(X_K)=\chi^{mot}(X) \in \Z \oplus \Z$, where we identify $\widehat{\mathrm{W}}(k) \cong \Z \oplus \Z$ by sending $n\langle 1 \rangle + m \langle -1 \rangle$ to $(n,m)$. Therefore, by base changing to an extension of $k$ containing $\mathbb{R}$, it suffices to prove the statement when $k$ is a real closed field containing $\mathbb{R}$. We proceed by parameterising the $K3$ surfaces using moduli spaces over $k$. First, let $\mathbb{Q}^{real}$ denote the real closure of $\mathbb{Q}$. For $d\geq0$, define $\mathrm{Hilb}^d := \mathrm{Hilb}^{9dt^2 + 2}_{\mathbb{P}_{\mathbb{Q}^{real}}^{9d + 1}}$. Let $H^d$ denote the open subvariety of $\mathrm{Hilb}^{d}$ from Proposition 2.1 of \cite{Huy}. It is defined by the universal property that a morphism $T \rightarrow \mathrm{Hilb}^d$ factors through $H^d$ if and only if the pullback of the universal family $f: \mathcal{Z}_T \rightarrow T$ satisfies:
\begin{enumerate}
\item The morphism $f$ is a smooth family and all the fibres are $K3$ surfaces.
\item If we have $p: \mathcal{Z}_T \rightarrow \mathbb{P}^{9d+1}$ is a projection, then $p^*\mathcal{O}(1) \cong L^3 \otimes f^* L_0$ for $L \in Pic(\mathcal{Z}_T)$ and $L_0 \in \mathrm{Pic}(T)$. 
\item The line bundle $L$ is primitive on each geometric fibre.
\item For all fibres $\mathcal{Z}_s$ of $f: \mathcal{Z}_T \rightarrow T$, the restriction map gives isomorphisms:
$$
H^0(\mathbb{P}^N_{k(s)}, \mathcal{O}(1)) \cong H^0(\mathcal{Z}_s, L_s^3).
$$
\end{enumerate}
Now let $X$ be a $K3$ surface over $k$ with a polarisation given by an ample line bundle $L$ such that $\mathrm{deg}(L)=d$. Then the discussion before Proposition 2.1 of \cite{Huy} tells us that we can view $X$ as a projective subvariety of $\mathbb{P}_{k}^{9d+1}$, with Hilbert polynomial $9d t^2 + 2$. By the definition of the Hilbert scheme, this corresponds to a morphism $\Spec(k) \rightarrow  \mathrm{Hilb}^d$ such that the pullback of the universal family is isomorphic to $X$. By the properties that $H^d$ has, this morphism factors through $H$, so we get a morphism $\Spec(k) \rightarrow H^d$. Therefore every $K3$ surface over $k$ with an ample line bundle $L$ is the pullback of the universal family of a morphism $f: \Spec(k) \rightarrow H^d$ for some $d$. We will abuse notation slightly and write $(X,L) \in \mathrm{Hom}(\Spec(k), H^d)$ to mean the $K3$ surface and line bundle specified by the morphism. This choice of $(X,L)$ specifies a birational map $\phi_{X,L}: \overline{\mathrm{Pic}}^g(\mathcal{C}) \rightarrow X^{[g]}$, as defined in Proposition 1.3 of \cite{Be}. Consider the following statement.

 ``For $d \in \mathbb{Z}_{>0}$ and $(X,L) \in \mathrm{Hom}(\Spec(k), H^d)$, there exists a Nash diffeomorphism
 $${\overline{\phi}_{(X,L)}: \overline{\mathrm{Pic}}^g(\mathcal{C})(k) \rightarrow X^{[g]}(k)}."$$

This can be written as a first order formula, and we have seen that this formula is true for $k=\mathbb{R}$. Therefore by the model completeness of real closed fields, it is true for every real closed field $k$ containing $\mathbb{R}$. The existence of such a Nash diffeomorphism means these semialgebraic sets have the same real closed Euler characteristic, as required.
\end{proof}

\begin{lemma}\label{BatKon3} Let $X$ and $Y$ be two smooth, projective, birational Calabi--Yau varieties. Then
$\mathrm{det}_{\ell}(X) = \mathrm{det}_{\ell}(Y)$. 
\end{lemma}

\begin{proof}
This lemma essentially follows from the proof of Theorem 1.1 of \cite{Ba}. We recall the argument here for convenience. Assume without the loss of generality that $k$ is finitely generated over $\Q$, since we can take $k$ to be the smallest field over which both $X$ and $Y$ are defined. Let $S$ be an integral scheme of finite type over $\Spec(\mathbb Z)$ whose function field is $k$. We may also assume that both $X$ and $Y$ have smooth projective models over
$S$, denoted by $\mathcal X,\mathcal Y$ respectively. For every closed point $s$ of $S$ let
$\F_s,\mathcal X_s,\mathcal Y_s$ denote the residue field of $s$, and the fibre of $\mathcal X,\mathcal Y$ over $s$, respectively. Let $\F_{s^n}$ denote the unique extension of $\F_s$ with Galois group $\mathbb Z/n\mathbb Z$. As noted on page 7 in the proof of Theorem 1.1 of \cite{Ba}, since $\mathcal{X}_s$ and $\mathcal{Y}_s$ are birational and Calabi--Yau, we have $|\mathcal X_s(\F_{s^n})|=|\mathcal Y_s(\F_{s^n})|$ for every $s$ and $n$. This in turn gives us an equality of zeta functions $\zeta(\mathcal X_s,t)=\zeta(\mathcal Y_s,t)$ for every $s$. Theorem $\ref{zetafunctionsdetermine}$ then gives the result.
\end{proof}

\begin{proof}[Proof of Theorem $\ref{batyrev}$]
Example 0.5 of \cite{Mu} says that $\overline{\mathrm{Pic}}^g(\mathcal{C})$ is Calabi--Yau, Example 0.4 of \cite{Mu} says that $X^{[g]}$ is also Calabi--Yau, and Proposition $\ref{beau1}$ says that these two varieties are birational. We can therefore apply Lemmas $\ref{BarKon1}$ and $\ref{BatKon3}$ to obtain that $\chi^{mot}( \overline{\mathrm{Pic}}^g\mathcal{C})$ and $\chi^{mot}(X^{[g]})$ have the same rank and $\ell$-adic determinant of cohomology. Moreover, since they are birational, they also have the same dimension so the same total weight, so by Proposition $\ref{rel}$, we have $\mathrm{disc}(\chi^{mot}( \overline{\mathrm{Pic}}^g\mathcal{C})) = \mathrm{disc}(\chi^{mot}(X^{[g]}))$. Finally, by Lemma $\ref{BarKon2}$, we see $\overline{\mathrm{Pic}}^g\mathcal{C}$ and $X^{[g]}$ have the same signature. Therefore by Lemma $\ref{modJ}$, they differ only by an element of $J$.
\end{proof}

\section{The motivic G\"ottsche formula}
In this section, we prove a refinement of the G\"ottsche formula for the classical Euler characteristic, which is Theorem 0.1 of \cite{Go} on page 193. This formula allows us to compute the motivic Euler characteristic of $X^{[g]}$ modulo $J$.
\begin{notn} Let $X$ be a smooth geometrically irreducible projective surface over $k$, let $X^{(n)}$ denote its $n$-th symmetric product, and let $X^{[n]}$ be the Hilbert scheme of $0$-dimensional subschemes of $X$ of length $n$. Let $\mathfrak P(n)$ denote the set of partitions of $n$. The elements of $\mathfrak P(n)$ are sequences $\underline m=(m_1,m_2,\ldots,m_r)$ of non-negative integers such that $1m_1+2m_2+\cdots+rm_r=n$. The length $l(\underline m)$ of the partition $\underline m$ is defined as $l(\underline m)=m_1+m_2+\cdots+m_r$.
\end{notn}
\begin{notn}\label{notn8.1} Let $\mathrm{Chow}(k)$ denote the category of rational Chow motives over $k$, and for every smooth projective variety $X$ and idempotent correspondence $\alpha$ from $X$ to $X$ over $k$ let $(X,\alpha)$ denote the object of $\mathrm{Chow}(k)$ corresponding to this pair. For every object $A$ of
$\mathrm{Chow}(k)$ and integer $n$ let $A(n)$ denote the $n$-th Tate twist of $A$, and denote direct sums and tensor products in the category $\mathrm{Chow}(k)$ by the usual symbols for sums and products respectively. For $X$ as above let $M(X)$ denote $(X,\Delta_X)$, where 
$\Delta_X$ is the diagonal cycle on $X\times X$. 
\end{notn}

\begin{thm}[(de Cataldo--Migliorini)]\label{goettche1} There is an isomorphism of Chow motives with rational coefficients:
$$M(X^{[n]})(-n)\cong\!\!\!\!\!\!\!\!\!\!\!\!\!
\sum_{\underline m=(m_1,m_2,\ldots,m_r)\in\mathfrak P(n)}\!\!\!\!\!\!\!\!\!\!\!\!\!
M(X^{(m_1)})
M(X^{(m_2)})
\cdots M(X^{(m_r)})(-l(\underline m))
\textrm{ in }\mathrm{Chow}(k).$$
\end{thm}
\begin{proof} The claim follows immediately from Theorem 6.2.5 of \cite{dCM}, and Remark 6.2.2 of \cite{dCM} where the authors mentioned that their result holds over any base field. 
\end{proof}
\begin{defn}\label{k0chow}
Define $K_0(\mathrm{Chow}(k))$ to be the Grothendieck ring of rational Chow motives, i.e., the ring generated by isomorphism classes of Chow motives over $k$ subject to the obvious addition and multiplication relations coming from the direct sum and tensor product. Let $K_0(\mathrm{Var}_k)^{Chow}$ denote the subring of $K_0(\mathrm{Chow}(k))$ generated by classes of the form $[M(X)(j)]$ where $X/k$ is a smooth projective variety and $j$ is an integer. 
\end{defn}

\begin{cor}\label{k0chowcor}
There is an equality in $K_0(\mathrm{Var}_k)^{Chow}$:
$$
[M(X^{[n]})(-n)]=\!\!\!\!\!\!\!\!\!\!\!\!\!
\sum_{\underline m=(m_1,m_2,\ldots,m_r)\in\mathfrak P(n)}
\prod_{i=1}^r \left(
[M(X^{(m_i)})(-m_i)]\right).
$$
\end{cor}
\begin{proof}
As noted in Remark 6.2.4 of \cite{dCM}, this follows immediately by Theorem $\ref{goettche1}$ and the definition of $K_0(\mathrm{Var}_k)^{Chow}$, where we absorb the Tate twists into the appropriate terms to obtain the formula in the form as above.
\end{proof}

\begin{lemma}
There is a ring homomorphism $M: K_0(\mathrm{Var}_k) \to K_0(\mathrm{Var}_k)^{Chow} \subseteq K_0(\mathrm{Chow}(k))$ such that if $X/k$ is a smooth projective variety, $M([X]) = [M(X)]$. 
\end{lemma}
\begin{proof}
We first check that this is a well defined homomorphism of abelian groups. By Theorem 3.1 of \cite{Bi}, the ring $K_0(\mathrm{Var}_k)$ is generated as an abelian group by the classes of smooth projective varieties, subject to a single blow up relation, given by
$$
[\mathrm{Bl}_Z(X)] + [Z] = [X] + [E]
$$
for any smooth projective variety $X$ and smooth closed subvariety $Z$, where $\mathrm{Bl}_Z(X)$ is the blow up of $X$ along $Z$ with exceptional divisor $E$. This same blow up relation is satisfied in the category of Chow motives by the Corollary on page 463 of \cite{Ma}. This assignment therefore gives us a well defined homomorphism of abelian groups, and the fact that it is a ring homomorphism follows instantly by the K{\"u}nneth formula. 
\end{proof}

In this section we construct a ring homomorphism 
$$
\chi^{Chow}: K_0(\mathrm{Var}_k)^{Chow} \to \widehat{\mathrm{W}}(k)/J
$$
 such that $\chi^{Chow} \circ M(X) = \chi^{mot}(X)$ in $\widehat{\mathrm{W}}(k)/J$ and examine how it applies to Tate twists. We construct $\chi^{Chow}$ by constructing versions of the complex Euler characteristic, real Euler characteristic and discriminant for elements of $K_0(\mathrm{Var}_k)^{Chow}$ and then examining how they behave under Tate twists. We then appeal to Lemma $\ref{modJ}$ to recover an element of $\widehat{\mathrm{W}}(k)$. In Corollary $\ref{goettche2}$, we apply $\chi^{Chow}$ to Theorem $\ref{goettche1}$ in order to obtain a formula for the motivic Euler characteristic of $X^{[n]}$ modulo $J$.

\begin{defn} Let $M$ be a Chow motive over a field $k$. We will abuse notation slightly and use the same notation as for varieties. Let $e: K_0(\mathrm{Chow}(k)) \to \Z$ denote the Euler characteristic of $M$, i.e., $e(M) := \sum_i (-1)^i \mathrm{dim}( H^i_{\ell-adic}(M))$. 

When $k$ is a real closed field we may define its real Euler characteristic as
$$\epsilon(M):=\mathrm{Tr}(c^M)$$
where $c^M:H^*_{Betti}(M)\to H^*_{Betti}(M)$ is complex conjugation, and $\mathrm{Tr}$ denotes graded trace as in Notation $\ref{gradedtrdet}$. Similarly to in Definition $\ref{2.13a}$, the above data furnishes a ring homomorphism $\chi^{real}: K_0(\mathrm{Chow}(k)) \to \mathcal{C}(\mathrm{Spr}(k), \Z)$.

We can also define the $\ell$-adic determinant of cohomology $\mathrm{det}_\ell(M)$ to be the one dimensional $\ell$-adic representation $\prod_i ( \mathrm{det} ( H^i_{\ell-adic}(M))^{(-1)^i})$. Similarly to Definition $\ref{chil}$, there is a ring homomorphism $\chi_l: K_0(\mathrm{Chow}(k)) \to E(\Hom(\Gal_k, \mathbb Q_\ell^\times))$ given by $M \mapsto (e(M), \mathrm{det}_\ell(M))$. 
\end{defn}

\begin{lemma}\label{notsmooth}
The compactly supported Euler characteristic $e: K_0(\mathrm{Var}_k) \to \Z$, the real closed compactly supported Euler characteristic $\chi^{real}: K_0(\mathrm{Var}_k) \to \mathcal{C}( \mathrm{Spr}(k), \Z)$,  and the $\ell$-adic determinant of cohomology $\mathrm{det}_{\ell}: K_0(\mathrm{Var}_k) \to \Hom(\Gal_k, \Q_{\ell}^\times)$ all factor through the map $M: K_0(\mathrm{Var}_k) \to K_0(\mathrm{Chow}(k))$ from the previous lemma.
\end{lemma}
\begin{proof}
This is clear since all of these invariants for smooth projective varieties can be read off from the Galois representation on the underlying cohomology ring for our Weil cohomology theories. In particular, this implies that $e(M(X)) = e(X)$, $\chi^{real}(M(X)) = \chi^{real}(X)$, and $\det_\ell(M(X)) = \det_\ell(X)$. 
\end{proof}
\begin{rem}
The reason the above lemma only applies to these invariants, rather than the whole motivic Euler characteristic, is that the definition of the motivic Euler characteristic using de Rham cohomology in Definition $\ref{DREulerCharacteristic}$ makes essential use of Poincaré duality as well as the cohomology groups. 
\end{rem}
We now study how these homomorphisms out of $K_0(\mathrm{Chow}(k))$ behave under Tate twists.

\begin{notn} Let $G$ be a group, and let $\rho:G\to GL_L(V)$ an $L$-linear representation on a finite dimensional graded vector space over the field $L$. Moreover, suppose that $V = \bigoplus_i V_i$ is graded, and that $\rho$ respects the grading. Let $\chi:G\to L^\times$ be a one-dimensional representation. Define the twisted representation $\rho(\chi):
G\to GL_L(V)$ by the formula:
$$\quad(\forall g\in G),  \rho(\chi)(g)=\chi(g)\mathrm{Id}_V\cdot\rho(g).$$
\end{notn}
We can compute the character and the determinant of the representation $\rho(\chi)$ in terms of these invariants of $\rho$ and $\chi$.
\begin{lemma}\label{twisteddeterminants} We have:
$$\mathrm{Tr}(\rho(\chi)(g))=\chi(g)\cdot\mathrm{Tr}(\rho(g))
\quad\textrm{and}\quad\det((\rho(\chi)(g)))=\chi(g)^{e(V)}\cdot
\mathrm{det}(\rho(g)),$$
where $\mathrm{Tr}$ and $\mathrm{det}$ denote the graded versions as in Notation $\ref{gradedtrdet}$.
\end{lemma}
\begin{proof} We may view $\chi(g) \in L^\times$ as a scalar, and $\rho(\chi)(g) = \chi(g) \cdot \rho(g)$. Shifting the grading multiplies the trace by $-1$ and raises the determinant to the power of $-1$, so it is enough to prove an ungraded version for $V=V_0$. This then follows immediately. 
\end{proof}

\begin{prop}\label{tatetwistofinvariants}The functions $e,\chi^{real}$ and $\mathrm{det}_\ell$ out of $K_0(\mathrm{Chow}(k))$ satisfy the following formulae under Tate twists.
\begin{enumerate}
\item [(i)] We have that $e(M(j)) = e(M)$.
\item[(ii)] When $k$ is real closed $\epsilon(M(j))=(-1)^j
\epsilon(M)$ and so $\chi^{real}(M(j)) = (-1)^j \chi^{real}(M)$.
\item[(iiii)] We have $\mathrm{det}_\ell(M(j))=\mathrm{det}_\ell(M)\cdot \Q_\ell(j\cdot e(M))$.
\end{enumerate}
\end{prop}
\begin{proof} The first part is immediate, since the dimension of the Betti realisation does not change under the Tate twist.

For the second part, it suffices to prove the statement when $k$ is real closed and we are looking at the real closed Euler characteristic $\epsilon$. Define $\chi: \Gal_k \to \overline{k}^\times$ by sending the non-trivial element of $\Gal_k$ to $-1 \in \overline{k}^\times$. We may identify the vector spaces $H^*_{Betti}(M(j)) \cong H^*_{Betti}(M)$, and under this identification, we see that $c^{M(j)} = (-1)^j c^M$. In particular, 
$$
\epsilon(M(j)) = \mathrm{Tr}(c^{M(j)}) = (-1)^j \mathrm{Tr}(c^M) = (-1)^j \epsilon(M),$$
as required. 

For the statement about $\mathrm{det}_\ell(M(j))$, let $g \in \mathrm{Gal}_k$. Write $\rho^M$ for the representation of $\Gal_k$ on $H^*_{\ell-adic}(M)$. We have an isomorphism of the underlying vector spaces $H^*_{\ell-adic}(M(j)) \cong H^*_{\ell-adic}(M)$, and under this identification, we see that $\rho^{M(j)}(g) = \chi(j) \cdot \rho^M(g)$, where $\chi$ denotes the $\ell$-adic cyclotomic character on $\Gal_k$. By definition, $\mathrm{det}_\ell(M)(g) = \mathrm{det}( \rho^M(g))$, so Lemma $\ref{twisteddeterminants}$ gives the result.
\end{proof}
The goal is to use these invariants to define a map $\chi^{Chow}$ which extends the motivic Euler characteristic to Chow motives. While we can define the $\ell$-adic determinant of cohomology for a general Chow motive, it is unclear how to define the discriminant in an analogous way. However, for elements of $K_0(\mathrm{Var}_k)^{Chow}$, we can define the discrimimant as follows.
\begin{defn}\label{discandcoh}
Suppose that $k/\Q$ is finitely generated, and let $M$ be a rational Chow motive over $k$ such that $M = M(X)(j)$ for some variety $X/k$. Then $H^i_{\ell-adic}(M) \cong H^i_{\et}(X_{\kbar}, \Q_\ell) \cdot \Q_\ell(j)$ for each $i$. Putting this together gives
$$
\mathrm{det}_{\ell}(M) = \mathrm{det}_\ell(X) \cdot \Q_\ell(j \cdot e(X))
$$
Multiplying again by $\Q_\ell(w(X)- j\cdot e(X))$ gives
$$
\mathrm{det}_{\ell}(M) \cdot \Q_\ell(w(X)-j\cdot e(X)) = \mathrm{det}_\ell(X) \cdot \Q_\ell(w(X)).
$$
Theorem $\ref{saito-full}$ implies that the above representation is a $2$-torsion character whose class in $k^\times/k^{\times2}$ is $(-1)^{w(X)} \mathrm{disc}(\chi^{mot}(X))$. As in Remark $\ref{wxtorsion}$, define the integer $w(M) := w(X) - j \cdot e(X)$, so that $w(M)$ is the unique integer such that $\mathrm{det}_{\ell}(M) \cdot \Q_\ell(w(M))$ is torsion. For such Chow motives define $\mathrm{disc}(M)$ to be the element of $k^\times/k^{\times2}$ such that
$$
\mathrm{disc}(M) \cdot (-1)^{w(M)} := \mathrm{det}_{\ell}(M) \cdot \Q_\ell(w(M)).
$$
We may also argue as in Remark $\ref{augdet}$ to obtain a ring homomorphism 
$$
\chi^{disc}_k: K_0(\mathrm{Var}_k)^{Chow} \to E(\mathrm{Hom}(\Gal_k, \mathbb Z /2\mathbb Z))$$
 given by $M \mapsto (e(M), \mathrm{disc}(M))$. 
 
If $k$ is any field of characteristic $0$, since $k$ is a union of its finitely generated subfields, we have that $K_0(\mathrm{Var}_k)^{Chow} = \varinjlim_{L' \subseteq k} K_0(\mathrm{Var}_L)^{Chow}$. 
 
For $k$ an arbitrary field of characteristic $0$, define 
$$
\chi^{disc}_k: K_0(\mathrm{Var}_k)^{Chow} \to E(\mathrm{Hom}(\Gal_k, \mathbb Z /2\mathbb Z))
$$
 to be the inductive limit of all of the $\chi^{disc}_L: K_0(\mathrm{Var}_L)^{Chow} \to E(\mathrm{Hom}(\Gal_L, \mathbb Z /2\mathbb Z))$ where the fields $L$ run over all finitely generated subfields of $k$. When the base field is clear, we will simply write $\chi^{disc}$. 
\end{defn}

\begin{lemma}\label{tatetwistofdisc}
For $X/k$ a smooth projective variety $$\mathrm{disc}(M(X)(j)) = \langle (-1)^{j\cdot e(M)} \rangle \cdot \mathrm{disc}(\chi^{mot}(X)).$$
\end{lemma}
\begin{proof}
This follows instantly since $(-1)^{w(M(X)(j))} = (-1)^{w(X)} \cdot (-1)^{j\cdot e(M)}$.
\end{proof}

\begin{thm}\label{chichow}
The homomorphism 
$$
e \times \chi^{real} \times \chi^{disc}: K_0(\mathrm{Var}_k)^{Chow} \to \Z \times \mathcal{C}(\mathrm{Spr}(k), \mathbb Z) \times E( \mathrm{Hom}(\Gal_k, \mathbb Z/2\mathbb Z))
$$
factors through the inclusion $\widehat{\mathrm{W}}(k)/J \hookrightarrow \Z \times \mathcal{C}(\mathrm{Spr}(k), \mathbb Z) \times E( \mathrm{Hom}(\Gal_k, \mathbb Z/2\mathbb Z))$ from Lemma $\ref{modJ}$. 
\end{thm}

\begin{proof}
It is enough to show that for $X$ a variety and $j$ an integer there exists some $q \in \widehat{\mathrm{W}}(k)/J$ such that  $e (M(X)(j))=\mathrm{rank}(q)$, $\chi^{real}(M(X)(j))=\mathrm{sign}(q)$ and $\chi^{disc}(M(X)(j)) = E_{disc}(q)$. Proposition $\ref{tatetwistofinvariants}$ and Lemma $\ref{tatetwistofdisc}$ tells us that taking $q= \langle (-1)^j \rangle \cdot \chi^{mot}(X)$ gives us the result.
\end{proof}

\begin{defn}
In light of the above, define $\chi^{Chow}: K_0(\mathrm{Var}_k)^{Chow} \to \widehat{\mathrm{W}}(k)/J$ to be the unique homomorphism such that the following diagram is commutative:
\begin{center}
\begin{tikzcd}
K_0(\mathrm{Var}_k)^{Chow} \ar[d, "\chi^{Chow}", dashed] \ar[drrr, "e \times \chi^{real} \times \chi^{disc}"] &&& \\
\widehat{\mathrm{W}}(k)/J \ar[rrr, "\mathrm{rank} \times \mathrm{sign} \times E_{disc}"', hookrightarrow] &&& \Z \times \mathcal{C}(\mathrm{Spr}(k), \mathbb Z) \times E( \mathrm{Hom}(\Gal_k, \mathbb Z/2\mathbb Z)).
\end{tikzcd}
\end{center}
\end{defn}

\begin{cor}\label{ChowTwist}
For $X/k$ a variety and $j$ an integer, we have
$$
\chi^{Chow}(M(X)(j)) = \langle (-1)^j \rangle \chi^{mot}(X) \in \widehat{\mathrm{W}}(k)/J.
$$
\end{cor}
\begin{proof}
Note that the twist by $j$ map $(j): K_0(\mathrm{Var}_k)^{Chow} \to K_0(\mathrm{Var}_k)^{Chow}$ given by sending $M(X)(i)$ to $M(X)(i+j)$ is a homomorphism of the underlying additive abelian group. Similarly, the map $\widehat{\mathrm{W}}(k)/J \to \widehat{\mathrm{W}}(k)/J$ given by multiplication by $\langle (-1)^j \rangle$ is a homomorphism on the underlying abelian group. Therefore, the compositions
$$
\chi^{Chow} \circ (j) \circ M: K_0(\mathrm{Var}_k) \to \widehat{\mathrm{W}}(k)/J
$$
and
$$
(\langle (-1)^j \rangle \cdot) \circ \chi^{mot}: K_0(\mathrm{Var}_k) \to \widehat{\mathrm{W}}(k)/J
$$
are both homomorphisms of abelian groups. It therefore suffices that these homomorphisms are equal on generators of the abelian group $K_0(\mathrm{Var}_k)$. Theorem 3.1 of \cite{Bi} tells us that we can check this only for $X/k$ a smooth projective variety by Theorem 3.1 of \cite{Bi}. This is then immediate by the proof of Theorem $\ref{chichow}$. 
\end{proof}

\begin{cor}[(Motivic G\"ottsche formula)]\label{goettche2} We have:
$$1+\sum_{n=1}^{\infty}\chi^{mot}(X^{[n]}) \langle (-1)^n \rangle t^n
= \prod_{d=1}^{\infty}\left(\sum_{m = 0}^{\infty}  \chi^{mot}(X^{(m)}) \langle -1 \rangle^{m} t^{md}\right)  \text{ in } \left( \widehat{\mathrm{W}}(k)/J \right)[[t]].
$$
\end{cor}

\begin{proof}
Fix $n$ to be a positive integer. Applying the ring homomorphism $\chi^{Chow}$ to the formula in Corollary $\ref{k0chowcor}$ gives
$$\chi^{Chow}(M(X^{[n]})(-n))=\!\!\!\!\!\!\!\!\!\!\!\!\!\!\!\!
\sum_{(m_1,m_2,\ldots,m_r)\in\mathfrak P(n)} \prod_{i=1}^r \left(
\chi^{Chow}(M(X^{(m_i)})(-m_i))\right) \in \widehat{\mathrm{W}}(k)/J.
$$
By Corollary $\ref{ChowTwist}$, we may rewrite this as
$$
\chi^{mot}(X^{[n]})\langle (-1)^n \rangle = \!\!\!\!\!\!\!\!\!\!\!\!\!\!\!\!
\sum_{(m_1,m_2,\ldots,m_r)\in\mathfrak P(n)} \left(\prod_{i=1}^r  \chi^{mot}(X^{(m_i)})\langle (-1)^{m_i} \rangle\right).
$$
We claim that the right hand side of this is the coefficient of $t^n$ in the power series 
$$
\prod_{d=1}^{\infty}\left( \sum_{m = 0}^{\infty} \chi^{mot}(X^{(m)})\langle (-1)^m \rangle t^{md} \right).
$$
Relabeling our $m$s by $m_d$s in order to keep track of where they appear in the product, the above series becomes
$$
\prod_{d=1}^{\infty}\left( \sum_{m_d = 0}^{\infty} \chi^{mot}(X^{(m_d)})\langle (-1)^{m_d} \rangle t^{m_dd} \right).
$$
Expanding the product and looking at the at the coefficient of $t^n$ in the above power series, we see that the $t^n$ term is given by
$$
\sum_{\{(m_1, \ldots, m_r): \sum dm_d = n\}}\left( \prod_{d=1}^r \chi^{mot}(X^{(m_d)}) \langle (-1)^{m_d}\rangle t^{dm_d}\right).
$$
As noted in Notation $\ref{notn8.1}$, we may identify $\mathfrak P(n)$ with $\{(m_1,\ldots,m_r): \sum dm_d = n\}$, which gives us that the coefficient of $t^n$ is exactly given by 
$$
\!\!\!\!\!\!\!\!\!\!\!\!\!\!\!\!
\sum_{(m_1,m_2,\ldots,m_r)\in\mathfrak P(n)} \left(\prod_{d=1}^r  \chi^{mot}(X^{(m_d)}) \langle (-1)^{m_d} \rangle\right).
$$
 Putting all of this together, $\chi^{mot}(X^{[n]}) \langle (-1)^n \rangle \in \widehat{\mathrm{W}}(k)/J$ is equal to the coefficient of $t^n$ in the generating series
$$
\prod_{d=1}^{\infty} \left(\sum_{m = 0}^{\infty} \chi^{mot}(X^{(m)})\langle (-1)^m \rangle t^{md}\right) \in (\widehat{\mathrm{W}}(k)/J)[[t]],
$$
which gives the theorem.
\end{proof}

We now show that we can recover the rank, sign, and discriminant of $\chi^{mot}(X^{(m)})$ to extract other G{\"o}ttsche formulae based on this. 

\begin{lemma}\label{realcomplexsymm}
Let $X$ be a $K3$ surface. Then $e(X^{(m)})={ e(X)+m-1 \choose m}$. If $k$ is a real closed field, then we may write $\chi^{mot}(X) = b\mathbb{H} + a\langle \pm 1 \rangle$ for some non-negative integers $a,b$. Moreover
$$
\epsilon(X) = \mathrm{sign}(\chi^{mot}(X^{(m)})) =  \pm \sum_{j=0}^{\lfloor \frac{m}{2} \rfloor} {a + m-2j -1 \choose m-2j} {b + j - 1 \choose j},$$ where the sign is $(-1)^m$ if $\mathrm{sign}(\chi^{mot}(X))$ is negative, and $1$ if $\mathrm{sign}(\chi^{mot}(X))$ is positive.
\end{lemma}
\begin{proof}
Since $X$ is a $K3$ surface, the odd Betti numbers vanish, so $H^i_{dR}(X/k)=0$ if $i$ is odd. Therefore, $\chi^{mot}(X) = [H^{2*}_{dR}(X/k)]$ is an actual quadratic form, so $\chi^{mot}(X)=b\mathbb{H} + a\langle \pm 1 \rangle$ for some $a,b \geq 0$.  For the rank, we may use Theorem $\ref{levine1}$ to reduce this to a statement about symmetric powers of CW complexes. The main result of \cite{Mac} then gives the result.

Suppose $k$ is a real closed field, and consider the complex $H^*_{sing}(X^{(m)}(\kbar), k)$, along with the action of $\sigma$, the non-trivial element of $\Gal_k$.  Moreover $\sigma$ is diagonalisable, so pick a basis of eigenvectors for $H^*_{sing}(X^{(m)}(\kbar), k)$, denoted by $v_1, \ldots, v_{e(X)}$. As noted in the second part of Remark 1.3 in \cite{Le}, we apply the Lefschetz trace formula to compute $\mathrm{Trace}(\sigma)=\epsilon(X)=e(X(k))$.

 Let $q$ be a quadratic form on $H^*_{sing}(X(\kbar), k)$ such that $q$ is diagonal with respect to $v_1, \ldots, v_{e(X)}$, and such that $\sigma(v_i) = q(v_i)v_i$. This means we obtain $e(X) = \mathrm{rank}(q)$ and $\epsilon(X)=\mathrm{sign}(q)$. By the K{\"u}nneth formula, we can identify $H^*_{sing}(X^{(m)}(\kbar), k) = S^m(H^*_{sign}(X(\kbar), k))$, where $S^m$ denotes the symmetric product, and where the action of $\sigma$ on the left hand side is given by $S^m(\sigma)$ on the right hand side. Therefore the quadratic form we obtain from looking at the action of $\sigma$ on both sides is the same and given by $S^m(q)$ the symmetric power of the form $q$. Recall that $q= b\mathbb{H} + a\langle \pm 1 \rangle$, where $a,b \geq 0$ and the $\pm$ comes from whether the $1$ or $-1$ eigenspace of $\sigma$ on $H^*_{sign}(X(\kbar), k)$ is larger. For now, suppose that the $-1$ eigenspace is larger, as this is the more difficult case.

 We see $S^m(q) = \bigoplus_{i=0}^m S^i(b\mathbb{H})S^{m-i}(a\langle - 1 \rangle)$. Proposition 4.12 of \cite{Mc} tells us that if $i$ is odd, then $\mathrm{sign}(S^{m-i}(b\mathbb{H})) = 0$, and $\mathrm{sign}(S^{i}(b\mathbb{H})) = {b + j - 1 \choose j}$ if $i=2j$ is even. We calculate $$
S^{m-2j}(a\langle -1 \rangle) = \langle (-1)^{m-2j} \rangle {a + m-2j-1 \choose m-2j-1}.
$$
Applying $\mathrm{sign}$ then gives
$$
\mathrm{sign}(S^m(q)) = (-1)^m \sum_{j=0}^{\lfloor \frac{m}{2} \rfloor} {a + m-2j -1 \choose m-2j} {b + j - 1 \choose j}.
$$
The argument for the case when the $1$ eigenspace is larger is similar.
\end{proof}

\begin{lemma}\label{detsymm}
Let $X$ be a $K3$ surface. Then 
$$
\mathrm{det}_{\ell}(X^{(m)}) = \mathrm{det}_{\ell}(X)^{ {m+e(X)-1 \choose m}}.$$
\end{lemma}
\begin{proof}
Since $X$ is a $K3$ surface, we have $H^i_{\et}(X, \Q_{\ell})=0$ for all $i$ odd. Therefore
$$
\mathrm{det}_{\ell}(X) = \mathrm{det}(R\Gamma(X_{\kbar}, \Q_{\ell})) = \mathrm{det}(H^*(X_{\kbar}, \Q_{\ell})).
$$ 
Note that $H^*(X_{\kbar}^{(m)}, \Q_{\ell}) = S^m(H^*(X_{\kbar}, \Q_{\ell}))$. Since $X^{(m)}$ is proper, the compactly supported cohomology agrees with the normal cohomology, so $\mathrm{det}_{\ell}(X)$ is the determinant of the Galois representation $H^*(X_{\kbar}^{(m)}, \Q_{\ell})$. Therefore, for all $\sigma \in \Gal_k$, we have
$$
\mathrm{det}_{\ell}(X^{(m)})(\sigma) = \mathrm{det}_{\ell}(X)^{ {m+e(X)-1 \choose m}},
$$
where this follows since for $T: V \to V$ a linear transformation, we have that $\mathrm{det}(S^m(T)) = \mathrm{det}(T)^{ {\mathrm{dim}(V) + m - 1 \choose m}}$.
\end{proof}

\begin{cor}\label{totsymm}
Let $X$ be a $K3$ surface, and let $w_m = {e(X)+m-1 \choose m}$. Then $\chi^{mot}(X^{(m)})$ is the unique element of $\widehat{\mathrm{W}}(k)/J$ satisfying the following properties.
\begin{enumerate}
\item The rank is given by $\mathrm{rank}(\chi^{mot}(X^{(m)}))= w_m$.
\item The signature is given by $\mathrm{sign}(\chi^{mot}(X^{(m)}_{k_<})) =  \pm \sum_{j=0}^{\lfloor \frac{m}{2} \rfloor} {a + m-2j -1 \choose m-2j} {b + j - 1 \choose j}$ for any choice of $k_< \in \mathrm{Spr}(k)$, where $a$ is given by $|\mathrm{sign}(\chi^{mot}(X)_{k_<})|$, $b=\frac12(w_m-a)$ and the sign at the front is equal to $1$ if $\mathrm{sign}(\chi^{mot}(X)_{k_<})$ is positive, and $(-1)^m$ if $\mathrm{sign}(\chi^{mot}(X)_{k_<})$ is negative.
\item The discriminant is $\mathrm{disc}(\chi^{mot}(X^{(m)})) =  \mathrm{disc}(\chi^{mot}(X))^{w_m} $. 
\end{enumerate}
\end{cor}
\begin{proof}
Uniqueness is immediate due to the definition of $J$. The first two equalities follow by the previous two lemmas.  For the final equality, Lemma $\ref{detsymm}$ applied to the $\ell$-adic cohomology gives $\mathrm{det}_\ell(X^{(m)}) = \mathrm{det}_{\ell}(X)^{w_m}$, and the result follows by Remark $\ref{wxtorsion}$ and Theorem $\ref{saito-full}$. 
\end{proof}

\begin{rem}\label{3.8got} Applying the rank homomorphism to the equation in Corollary \ref{goettche2} and using Theorem \ref{levine1} recovers G\"ottsche's formula for Euler characteristics
$$1+\sum_{n=1}^{\infty}e(X^{[n]})t^n
=\prod_{d=1}^{\infty} \left(1+\sum_{m=1}^{\infty} {m + e(X) - 1 \choose m} t^{md}\right)  =  
\prod_{d=1}^{\infty}\big(1-t^d\big)^{-e(X)}
\textrm{ in }
\mathbb Z[[t]].$$
\end{rem}
Assume that $k$ is real closed, so $\mathrm{Spr}(k)$ is a single point and $\mathcal C(\mathrm{Spr}(k),\mathbb Z)=\mathbb Z$.
\begin{cor}[(Real closed G\"ottsche formula)]\label{goettchereal} We have
$$1+\sum_{n=1}^{\infty}\epsilon(X^{[n]})t^n
=\prod_{i=1}^{\infty}\big(1+(-t)^i)\big)^{-a}\cdot
\prod_{j=1}^{\infty}\big(1-t^{2j}\big)^{-b}
\textrm{ in }\mathbb Z[[t]],$$
where $a=\epsilon(X)$ and $b=\frac{1}{2}(e(X)-\epsilon(X))$.
\end{cor}
\begin{proof}
First, suppose $a\geq0$, so $\chi^{dR}(X) = b\mathbb{H} + a\langle 1 \rangle$ where both $a,b\geq0$. Then taking $\mathrm{sign}$ of both sides of the equation in Corollary $\ref{goettche2}$ and using Corollary $\ref{totsymm}$ we obtain:
\begin{align*}
1+\sum_{n=1}^{\infty}\epsilon(X^{[n]})(-1)^n t^n
&= \prod_{d=1}^{\infty} \sum_{m=0}^{\infty}\left( \sum_{j=0}^{\lfloor \frac{m}{2} \rfloor} {a + m-2j -1 \choose m-2j} {b + j - 1 \choose j} \right) (-1)^m t^{md} \\
&= \prod_{d=1}^{\infty} \sum_{m=0}^{\infty}\left( \sum_{j=0}^{\lfloor \frac{m}{2} \rfloor} {a + m-2j -1 \choose m-2j} (-1)^{m-2j} t^{(m-2j)d} {b + j - 1 \choose j}t^{2jd} \right) \\
&= \prod_{d=1}^{\infty} \left( \left( \sum_{i=0}^{\infty} {a+i-1 \choose i} (-1)^i t^{id} \right) \cdot \left( \sum_{j=0}^\infty {b + j - 1 \choose j} t^{2jd} \right) \right)\\
&= \prod_{d=1}^{\infty} \left(  (1+t^d)^{-a} (1-t^{2d})^{-b} \right) \\
&=\prod_{i=1}^{\infty}\big(1+t^i\big)^{-a}\cdot \prod_{j=1}^{\infty}\big(1-t^{2j}\big)^{-b}.
\end{align*}
Relabelling $t=-t$, we obtain
$$
1+\sum_{n=1}^{\infty}\epsilon(X^{[n]}) t^n  = \prod_{i=1}^{\infty}\big(1+(-t)^i)\big)^{-a}\cdot \prod_{j=1}^{\infty}\big(1-t^{2j}\big)^{-b}
$$
as required. It remains to show the case when $a<0$. Define the elements $\tilde{a}=-a$ and $\tilde{b}=\frac12(e(X)-\tilde{a})$, so $\tilde{b} = b-a$. Taking $\mathrm{sign}$ gives
\begin{align*}
\sum_{n=1}^{\infty}\epsilon(X^{[n]})(-1)^nt^n &= \prod_{d=1}^{\infty} \sum_{m=0}^{\infty}\left( \sum_{j=0}^{\lfloor \frac{m}{2} \rfloor} {\tilde{a} + m-2j -1 \choose m-2j} {\tilde{b} + j - 1 \choose j} \right)t^{md}\\
&=\prod_{d=1}^{\infty} \sum_{m=0}^{\infty} \left( \sum_{j=0}^{\lfloor \frac{m}{2} \rfloor} {\tilde{a} + m-2j -1 \choose m-2j}  t^{(m-2j)d} {\tilde{b} + j - 1 \choose j}t^{2jd} \right) \\
&= \prod_{d=1}^{\infty} \left(  (1-t^d)^{-\tilde{a}} (1-t^{2d})^{-\tilde{b}} \right).
\end{align*}
Using that $(1-t^d) = (1-t^{2d}) \cdot (1+t^d)^{-1}$, we can rewrite
$$
\prod_{d=1}^{\infty} \left((1-t^d)^{-\tilde{a}} (1-t^{2d})^{-\tilde{b}}\right) = \prod_{d=1}^{\infty} \left((1+t^d)^{\tilde{a}} \cdot (1-t^{2d})^{-\tilde{b} -\tilde{a}}\right).
$$
Rewriting $\tilde{a}=-a$ and $\tilde{b}=b-a$ gives
$$
\prod_{d=1}^{\infty} \left((1+t^d)^{\tilde{a}} \cdot (1-t^{2d})^{-\tilde{b} -\tilde{a}}\right) = \prod_{i=1}^{\infty}\big(1+t^i\big)^{-a}\cdot \prod_{j=1}^{\infty}\big(1-t^{2j}\big)^{-b}.
$$
Putting everything together, and relabelling $t$ for $-t$ gives the result. 
\end{proof}
In the case that $k=\mathbb{R}$, this exactly recovers the real G\"ottsche formula of Theorem 1.2 of \cite{KR1}. Let $k$ be again an arbitrary field of characteristic $0$. Identify $\widehat{\mathrm{W}}(k)/I^2 = E( k^\times/k^{\times2})$ as in Remark $\ref{augdet}$. Write $\chi^{disc}$ for the composition $K_0(\mathrm{Var}_k) \xrightarrow{\chi^{mot}} \widehat{\mathrm{W}}(k) \to \widehat{\mathrm{W}}(k)/I^2$, so that after this identification, we have that $\chi^{disc}(X) = (e(X), \mathrm{disc}(\chi^{mot}(X))$. 

\begin{thm}\label{goettche3}
There is an equality in the ring $E( k^\times/k^{\times2})[[t]]$:
$$
1+\sum_{n=1}^\infty (\chi^{disc}(X^{[n]}))t^n = \prod_{d=1}^{\infty} (1 - (1, \mathrm{disc}(\chi^{mot}(X)))t^m)^{-e(X)}.
$$
\end{thm}
\begin{proof}
Note that by Corollary $\ref{totsymm}$,
\begin{align*}
\chi^{disc}(X^{(m)}) &= \left(  {m + e(X) - 1 \choose m}, \mathrm{disc}(\chi^{mot}(X))^{ {m + e(X) - 1 \choose m}}\right) \\&=  {m + e(X) - 1 \choose m}(1, \mathrm{disc}(\chi^{mot}(X))).\end{align*}

Applying the homomorphism $\chi^{disc}$ to Corollary $\ref{goettche2}$ gives
\begin{align*}
1+\sum_{n=1}^\infty \chi^{disc}(X^{[n]})(1, (-1))^nt^n &= \prod_{d=1}^\infty \sum_{m=0}^\infty {m + e(X) - 1 \choose m}(1, \mathrm{disc}(\chi^{mot}(X)) \cdot (1,-1)^m t^{md} \\
&=\prod_{d=1}^{\infty} (1, \mathrm{disc}(\chi^{mot}(X))) \cdot \sum_{m=0}^{\infty} {m + e(X) - 1 \choose m} \left( (1,-1)t^d \right)^m \\
&= \prod_{d=1}^{\infty} (1,\mathrm{disc}(\chi^{mot}(X))) (1- (1,-1)t^d)^{-e(X)}.
\end{align*}
Relabelling $t = (1,-1)t$ gives
$$
1+\sum_{n=1}^\infty \chi^{disc}(X^{[n]})t^n  = \prod_{d=1}^{\infty} (1,\mathrm{disc}(\chi^{mot}(X))) (1- t^d)^{-e(X)},
$$
as required.
\end{proof}

\section{The arithmetic Yau--Zaslow formula}
Recall that $X$ is a K3 surface admitting a complete $g$-dimensional linear system $\mathcal C$ of geometrically integral curves of genus $g$ and the curves in this linear system correspond to a primitive divisor class. Suppose that the rational curves in $\mathcal C$ are all nodal, which is generically the case by Chen's theorem, Theorem 1.1 of \cite{Ch}.
\begin{lemma}\label{yz1}
The motivic Euler characteristic $\chi^{mot}(\overline{\mathrm{Pic}}^g(\mathcal{C})) \pmod{J}$ is given by $\langle (-1)^g \rangle$ multiplied by the coefficient of $t^g$ in the generating series
$$
\prod_{d=1}^{\infty} \sum_{m = 0}^{\infty} \chi^{mot}(X^{(m)})\langle (-1)^m \rangle t^{md} \in (\widehat{\mathrm{W}}(k)/J)[[t]].
$$
\end{lemma}
\begin{proof}
By Theorem $\ref{batyrev}$, we have that $\chi^{mot}(\overline{\mathrm{Pic}}^g(\mathcal{C})) \equiv \chi^{mot}(X^{[g]}) \pmod J$, so applying Corollary $\ref{goettche2}$ gives us the result.
\end{proof}
\begin{cor}\label{dependencycor}
The element $\mathbf B^{mot}_{\mathcal{C}}(X)$ from Definition $\ref{BmotC}$ only depends on $X$ and $g$ modulo $J_g$, and not on $\mathcal{C}$ itself.
\end{cor}
\begin{proof}
Lemma $\ref{9.2}$ gives that $\mathbf B^{mot}_{\mathcal{C}}(X) \equiv \chi^{mot}(\overline{\mathrm{Pic}}^g(\mathcal C)) \pmod{J_g}$, so the result holds by the above lemma.
\end{proof}
The above corollary allows us to write $\mathbf B^{mot}_g(X) := \mathbf B^{mot}_{\mathcal{C}}(X)$. Our calculation of $\overline{\mathrm{Pic}}^g(\mathcal{C})$ only holds modulo $J_g$, so we need the following definition.
\begin{defn}
Let $\mathcal{J}$ denote the subset of $\widehat{\mathrm{W}}(k)[[t]]$ given by power series of the form $\sum_{i=0}^\infty a_i t^i$, where each $a_i \in J_i$, where $J_i$ is the ideal from Definition $\ref{jgdefinition}$. Since each $J_i$ is an ideal of $\widehat{\mathrm{W}}(k)$ with $J_i \subseteq J_{i+1}$, we see that $\mathcal{J}$ is an ideal of $\widehat{\mathrm{W}}(k)[[t]]$. 
\end{defn}

\begin{thm}[(Arithmetic Yau--Zaslow formula)]\label{yau-zaslow}
Let $X, \mathcal{C}$ be as above. Then $\langle (-1)^g \rangle \mathbf{B}^{mot}_g(X)$ is given by the coefficient of $t^g$ in the power series 
$$
\prod_{d=1}^{\infty} \sum_{m = 0}^{\infty} \chi^{mot}(X^{(m)})\langle (-1)^m \rangle t^{md} \in \widehat{\mathrm{W}}(k)[[t]]/\mathcal{J}. 
$$
\end{thm}
\begin{proof}
Apply Lemma $\ref{9.2}$ to the result of Theorem $\ref{yz1}$.
\end{proof}

\begin{rem}\label{complexyauzaslowformula}
Consider the rank of $\mathbf B^{mot}_g(X)$
$$
\mathrm{rank}( \mathbf B^{mot}_g(X)) =  \mathrm{rank} \sum_{q \in G(X)} \mathrm{Tr}_{k(q)/k} \left( \prod_{p\in\mathcal S(C_q)}\Phi_p\right).
$$
Note that $\mathrm{rank}( \mathrm{Tr}_{k(q)/k}(w)) = [k(q):k]\mathrm{rank}(q)$, so since $\Phi_p$ has rank $1$, this gives
$$
\mathrm{rank}( \mathbf B^{mot}_g(X)) =  \sum_{q \in G(X)}[k(q):k]=\# G(X)(\kbar).
$$
That is, applying rank to $\mathbf B^{mot}_g(X)$ recovers the number of rational curves in $\mathcal{C}$ after base changing to the algebraic closure. Applying Remark $\ref{3.8got}$ to the right hand side of the above theorem therefore recovers the complex Yau--Zaslow formula of \cite{Be} when working over $\mathbb{C}$.
\end{rem}

\begin{defn}Assume that $k$ is real closed. Let $C$ be a semi-stable curve over $k$. Following \cite{KR1} we call $(-1)^s$ the {\it Welschinger number} $W(C)$ of $C$, where $s$ is the number of zero-dimensional semi-algebraic connected components of $C(k)$. We set
$$\mathbb W(X, \mathcal{C}):=\sum_{q\in G(X)(k)}W(C_q)\in \mathbb Z.$$
\end{defn}
\begin{lemma}
For $g$ the dimension of $\mathcal{C}$, we have
$$\mathbb W(X, \mathcal{C}) = (-1)^g\mathrm{sign}( \mathbf B^{mot}_\mathcal{C}(X) ).$$
\end{lemma}
\begin{proof}
By Corollary $\ref{realpicard}$, the signature $\mathrm{sign}(\prod_{p \in \mathcal{S}(C_q)} \Phi_p) = (-1)^{c_{C_q}}$, where $c_{C_q}$ denotes the number of points in $\mathcal{S}(C_q)$ such that both $p$ and the tangent spaces at $p$ are defined over $k$. Let $s_{C_q}$ denote the number of points in $\mathcal{S}(C_q)$ that are defined over $k$, but the tangent spaces are not defined over $k$, and let $n_{C_q}$ denote the number of complex points in $\mathcal{S}(C_q)$. Note that $s_{C_q} + c_{C_q} + 2n_{C_q} = g$, so $(-1)^{c_{C_q}} = (-1)^g W(C_q)$. Therefore
$$
\mathrm{sign}( \mathbf B^{mot}_\mathcal{C}(X) ) = \sum_{q \in G(X)(k)} (-1)^g W(C_q) = (-1)^g \mathbb W(X, \mathcal{C})
$$ as required.
\end{proof}

\begin{cor}[(Real Closed Arithmetic Yau--Zaslow formula)]\label{yau-zaslow4} The number $\mathbb W(X, \mathcal{C})$ depends only on $g$ and $X$ so we can write $\mathbb W(X, g)$. Moreover, $\mathbb W(X,g)$ is given by the coefficient of $t^g$ in the power series
$$
\prod_{i=1}^{\infty}\big(1+(-t^i)\big)^{-a} \cdot \prod_{j=1}^{\infty}\big(1-t^{2j}\big)^{-b} \in \Z[[t]],$$
where $a=e(X(k))$ and $b=e(X(\kbar))-\frac{1}{2}e(X(k))$.
\end{cor}
\begin{proof} 
Applying $\mathrm{sign}$ to the power series in Theorem $\ref{yau-zaslow}$, and using the G\"ottsche formula $\ref{goettchereal}$ gives that $(-1)^g \mathrm{sign}( \mathbf B^{mot}_\mathcal{C}(X) ) = \mathbb W(X, \mathcal{C})$ is the coefficient of $t^g$ in the power series
$$
\prod_{i=1}^{\infty}\big(1+(-t)^i)\big)^{-a}\cdot \prod_{j=1}^{\infty}\big(1-t^{2j}\big)^{-b}.
$$
\end{proof}
\begin{rem} The corollary above generalises the main theorem of Kharlamov--R\u asdeaconu from \cite{KR1} to all real closed fields.
\end{rem}

We now turn to the discriminant. Let $k$ be an arbitrary field of characteristic $0$. Define $r_g := \mathrm{rank}(\mathbf{B}^{mot}_{g}(X))$, and let $\mathbb D^{mot}_{g}(X) := \mathrm{disc}(\mathbf{B}^{mot}_{g}(X)) \in k^\times/k^{\times2}$. 

\begin{lemma}
There is an equality in $\Hom(\Gal_k, \Z/2\Z)/\Delta^k_g$:
$$
\mathbb{D}_{g}^{mot}(X) = (-1)^{r_g}\prod_{q\in G(X)} \mathrm{disc}(k(q)/k)
\prod_{p\in\mathcal S(\pi_g^{-1}(q))}\mathrm{N}_{k(p)/k}(\alpha_p) .
$$
\end{lemma}
\begin{proof}
By Lemma $\ref{tracedisc}$ and Corollary $\ref{normdisc}$, we obtain 
$$
\mathrm{disc}(\mathbf{B}^{mot}_{g}(X)) = \prod_{q\in G(X)} \mathrm{disc}(k(q)/k) \cdot (-1)^{[k(q):k]}
\prod_{p\in\mathcal S(\pi_g^{-1}(q))}\mathrm{N}_{k(p)/k}(\alpha_p),
$$
which gives the result since $\sum_{q \in G(X)} [k(q):k] = r_g$.
\end{proof}
This allows us to obtain the following Yau--Zaslow formula for discriminants.
\begin{cor}\label{yau-zaslow5} Suppose that $k$ is finitely generated over $\Q$. Then the pair $(r_g, \mathbb D^{mot}_g(X))$, viewed as an element of $E(\left(k^\times/k^{\times2}\right)/\Delta^k_g)$, is given by the coefficient of $t^g$ in the generating series
$$
\prod_{d=1}^{\infty} \left(1,\mathrm{disc}(\chi^{mot}(X))\right) (1- t^d)^{-e(X)} \in E(\left(k^\times/k^{\times2}\right)/\Delta^k_g)[[t]].
$$
\end{cor}
\begin{proof}Theorem $\ref{batyrev}$ gives $\chi^{disc}(\overline{\mathrm{Pic}}^g(\mathcal{C}))=\chi^{disc}( X^{[g]})$. By the previous lemma, this is $(r_g, \mathbb D^{mot}_{g}(X))$. Applying Theorem $\ref{goettche3}$ gives the result.
\end{proof}
\begin{rem}
In particular, the above formula implies $\mathbb D^{mot}_g(X) \in k^\times/k^{\times2} \pmod{\Delta^k_g}$ is always either $1$ or $\mathrm{disc}(\chi^{mot}(X))$. 
\end{rem}

\end{document}